\documentclass[10pt]{amsart} 

\usepackage{fullpage}

\usepackage{url}

\usepackage{amssymb}

\usepackage{cite}

\usepackage{graphicx}

\usepackage[usenames]{color}

\usepackage{accents}

\usepackage{hyperref}

\renewcommand{\a}{\alpha}
\renewcommand{\b}{\beta}
\newcommand{\g}{\gamma}
\renewcommand{\d}{\delta}

\newcommand{\e}{\varepsilon}
\newcommand{\f}{\varphi}
\newcommand{\s}{\sigma}
\newcommand{\Si}{\Sigma}
\renewcommand{\k}{\kappa}
\renewcommand{\l}{\lambda}

\renewcommand{\O}{\Omega}
\renewcommand{\o}{\omega}

\newcommand{\cC}{{\mathcal C}}

\newcommand{\cT}{{\mathcal T}}

\newcommand{\cB}{{\mathcal B}}
\newcommand{\cL}{{\mathcal L}}
\newcommand{\cE}{{\mathcal E}}
\newcommand{\cU}{{\mathcal U}}

\newcommand{\cN}{{\mathcal N}}

\newcommand{\cH}{{\mathcal H}}

\newcommand{\cW}{\mathcal W}

\newcommand{\cI}{\mathcal I}

\newcommand{\bR}{\mathbb R}

\newcommand{\bZ}{\mathbb Z}

\newcommand{\bE}{\mathbb E}

\newcommand{\bV}{\mathbb V}
\newcommand{\Diff}{\mathrm{Diff}}
\newcommand{\Ric}{\mathrm{Ric}}

\newcommand{\tT}{\mathsf{T}}

\newcommand{\be}{\begin{equation}}
\newcommand{\ee}{\end{equation}}
\newcommand{\bes}{\begin{equation*}}
\newcommand{\ees}{\end{equation*}}

\newcommand{\tr}{\mathrm{tr}}

\newcommand{\beaa}{\begin{eqnarray*}}
\newcommand{\bea}{\begin{eqnarray}}
\newcommand{\beal}[1]{\begin{eqnarray}\label{#1}}
\newcommand{\bean}{\begin{eqnarray}\nonumber}
\newcommand{\beadl}[1]{\begin{deqarr}\label{#1}}
\newcommand{\eeadl}[1]{\arrlabel{#1}\end{deqarr}}
\newcommand{\eeal}[1]{\label{#1}\end{eqnarray}}
\newcommand{\eead}[1]{\end{deqarr}}
\newcommand{\eea}{\end{eqnarray}}
\newcommand{\eeaa}{\end{eqnarray*}}

\newcommand{\p}{\partial}

\renewcommand{\to}{\rightarrow}

\renewcommand{\exp}{\operatorname{exp}}

\renewcommand{\phi}{\varphi}
\renewcommand{\epsilon}{\varepsilon}
\renewcommand{\hat}{\widehat}

\newcommand{\<}{\langle}
\renewcommand{\>}{\rangle}

\newcommand{\dm}{{\partial M}}

\newcommand{\w}{\widetilde}

\theoremstyle{plain}
\newtheorem{lemma}{Lemma}[section]
\newtheorem{proposition}[lemma]{Proposition}
\newtheorem{theorem}[lemma]{Theorem}

\theoremstyle{remark}
\theoremstyle{definition}
\newtheorem{remark}[lemma]{Remark}

\def\blacksquare{\hbox to .60em {\vrule width .60em height .60em}}

\numberwithin{equation}{section}

\begin{document}

\title[ ]{Well-Posed Geometric Boundary data in General Relativity, I: Dirichlet boundary data} 

\author{Zhongshan An and Michael T. Anderson}

\address{Institute of Geometry and Physics, 
University of Science and Technology of China,
No. 99 Xiupu Road, Shanghai, China, 201305}
\email{zhshan.an@gmail.com}

\address{Department of Mathematics, 
Stony Brook University,
Stony Brook, NY 11794}
\email{michael.anderson@stonybrook.edu}


\begin{abstract} 
In this first work in a series, we prove the local-in-time well-posedness of the IBVP for the vacuum Einstein equations with Dirichlet 
boundary data on a finite timelike boundary, provided the Brown-York stress tensor of the boundary is a Lorentz metric of the same signature 
(up to an overall sign) as the induced Lorentz metric on the boundary. This is a convexity-type assumption which is an exact analog of a similar 
result in the Riemannian setting. This assumption on the (extrinsic) Brown-York tensor cannot be dropped in general. 

\end{abstract}

\maketitle 

\section{Introduction}

  This work is the first in a series on the initial boundary value problem (IBVP) in general relativity. Let $(M, g)$ be a globally hyperbolic 
vacuum solution of the Einstein equations 
\be \label{vacuum}
\Ric_g = 0,
\ee
on a finite domain $M \simeq I\times S$, where the Cauchy surface $S$ is a compact, connected $n$-manifold with boundary $\p S = \Si$ and $\cC = I\times \Si$ is a 
timelike boundary $\cC$ of $M$. Consider the space $\bE$ of all such solutions; this is also known as the covariant pre-phase space on a domain with 
finite boundary, cf.~\cite{HW} and references therein. The quotient space 
$$\cE = \bE/{\rm Diff}_0(M)$$
of $\bE$ by the group $\Diff_0(M)$ of diffeomorphisms $\f$ of $M$ equal to the identity on $S\cup \cC$ is the moduli space of vacuum solutions, closely related to the 
covariant phase space of vacuum Einstein metrics. 

  The basic issue here is whether the space $\cE$ can be characterized and be given an effective description at least locally in time  in terms of 
initial data on $S$ and boundary data on $\cC$. Thus, let $\cI_0$ denote the space of initial data, i.e.~pairs $(\g_S, \k)$ consisting of a Riemannian 
metric $\g_S$ on $S$ and a symmetric bilinear form $\k$ on $S$, satisfying the vacuum Einstein constraint equations, (discussed in detail later). 
In this work, we choose the space of boundary data 
$$\cB = \cB_{Dir} = Met(\cC)$$
to be the space of Dirichlet boundary data, consisting of globally hyperbolic Lorentz metrics $\g_{\cC}$ on $\cC$. (Other choices of boundary data space 
$\cB$ will be discussed in later work). Thus, roughly speaking, we seek a faithful correspondence -- for example a homeomorphism or diffeomorphism 
\be \label{Phi0}
\begin{split}
&\Phi: \cE \to \cI_0 \times_c \cB,\\
&\Phi(g) = (g_S, K, g_{\cC}), \\
\end{split}
\ee
where $g_S$ is the metric on $S$ induced by $g$, $K$ is the second fundamental form of $S \subset (M, g)$ and $g_{\cC}$ is the metric on $\cC$ induced 
by $g$. It is well-known that local well-posedness does hold for the Cauchy or initial value problem (IVP) where there is no boundary, cf.~\cite{C}. In this 
case $\cB = \emptyset$ and $\Phi: \cE \to \cI_0$ is a homeomorphism. The subscript $c$ in \eqref{Phi0} denotes the compatibility conditions between 
data in $\cI_0$ and $\cB$, discussed in detail in \S 2.4. 

  Now on $\bE$, both the metric $g$ and the boundary $\cC$ are dynamical, and so on a slicing by Cauchy surfaces $S_t$, both the metric $g_t$ and 
the location or evolution of the boundary $\Si_t = \p S_t$ are to be determined from the initial and boundary data. In this work, we discuss only the 
local-in-time behavior. A full understanding of this is necessary to begin any more global-in-time considerations. Thus, in the following 
$$t^*: \cE \to \bR^+,$$
will denote a `time-of-existence' function. The metric $g$ is defined on the neighborhood $[0, t^*)\times S$, where $t^* = t^*(g)$ denotes the $g$-proper time 
between $S_0 = \{0\}\times S$ to $S_{t^*} = \{t^*\}\times S$. The well-posedness of the IBVP then requires the existence of such a time function 
$t^*$ for which the map $\Phi$ in \eqref{Phi0} is (at least) a homeomorphism; thus solutions $g$ of \eqref{vacuum} exist up to time $t^*$ and are 
uniquely determined by their initial and boundary data. Moreover, such solutions depend at least continuously on the initial and boundary data.

\medskip 

  Consider first the simpler Euclidean or Riemannian situation where $g$ is a Riemannian vacuum Einstein metric on a domain $M$ with 
boundary $\dm$ and with induced boundary metric $g_{\dm}$. This corresponds roughly to studying the space of time-independent or stationary 
solutions of the vacuum equations \eqref{vacuum}. Now, strictly speaking, here the Dirichlet BVP is never elliptic or well-posed; in fact the Hamiltonian constraint 
(or Gauss equation) is an obstruction to ellipticity. Briefly, to see this, the Hamiltonian constraint or Gauss equation for Riemannian Einstein metrics \eqref{vacuum} 
gives 
\be \label{Ham}
|A|^2 - H^2 + R_{g_{\dm}} = 0,
\ee
on $\dm$, where $A$ is the second fundamental form of $\dm$ in $(M, g)$, $H = \tr A$ is the mean curvature of $\dm$ and $R_{g_{\dm}}$ is the scalar curvature 
of $(\dm, g_{\dm})$. Ellipticity of the equations \eqref{vacuum}, in any suitable gauge, would imply that solutions $g$ on $M$ are as smooth as the smoothness of the 
boundary metric $g_{\dm}$, i.e. if $g_{\dm}$ is $C^{k,\a}$, then $g$ itself must be $C^{k,\a}$. (Similar statements hold with respect to Sobolev topologies). 
However, $g \in C^{k,\a}$ implies $A, H \in C^{k-1,\a}$ so that \eqref{Ham} implies $R_{\dm} \in C^{k-1,\a}$; this is clearly false for generic $C^{k,\a}$ boundary metrics. 
This shows that in function spaces with a given finite degree of differentiability, the Einstein equation \eqref{vacuum} is not solvable for generic Dirichlet 
boundary data. 

  Similar arguments apply in the Lorentzian case, where the Hamiltonian constraint implies there are no boundary stable energy estimates, cf.~\cite{A3}, 
in any gauge. The boundary problem is thus not maximally dissipative in any gauge, calling into question the well-posedness. 

  There are in fact numerous examples where both existence and uniqueness break down for Dirichlet boundary data. This occurs in both 
Riemannian and Lorentzian settings in regions where $A = 0$, cf.~\cite{AA2}. More interesting examples have recently been discussed in detail in 
\cite{AGAM}, cf.~\S 8 for further discussion.  

\medskip 

  Nevertheless, there is an important class of Euclidean (or stationary Lorentzian) vacuum solutions for which Dirichlet boundary data is effectively 
elliptic and so well-posed.  Thus if the symmetric form 
\be \label{BY}
\Pi = Hg_{\dm} - A,
\ee
is positive (or negative) definite on $\dm$, i.e.~of the same signature (up to sign) as the induced metric $g_{\dm}$ on $\dm$, then Dirichlet boundary 
data are effectively elliptic, if one allows for loss-of-dervative behavior, cf.~\cite{A1}, \cite{A2}, \cite{W}. In this case, the map (similar to $\Phi$ above) from bulk 
solutions $g$ to Dirichlet boundary data $g_{\dm}$ is a smooth tame Fredholm map in $C^{\infty}$ near $(M, g)$; this guarantees existence and uniqueness up to a finite 
dimensional indeterminacy. (Allowing for a loss of derivatives, this also holds in spaces with only finite differentiability). A key idea in the proof of such ellipticity, 
originating in Nash's work on the isometric embedding problem of Riemannian metrics into $\bR^N$ (and the associated Nash inverse function theorem) is the form 
of the variation of the boundary metric $g_{\dm}$ under infinitesimal deformations (diffeomorphisms) normal to the boundary; cf.~\S 2.5 for further discussion. 
 
   The form $\Pi$ arises naturally from a physics perspective. It is the momentum conjugate to the metric $g_{\dm}$ for the Euclidean Einstein-Hilbert 
action with Gibbons-Hawking-York boundary term
$$\int_M R_g dv_g + 2\int_{\dm}H_g dv_{g_{\dm}}.$$
It is also the Brown-York stress-energy tensor for Dirichlet boundary data on $\dm$, cf.~\cite{BY}. Note that in dimension 3, the definiteness of $\Pi$ is 
equivalent to the definiteness of the second fundamental form $A$, i.e.~the convexity (or concavity) of $\dm$ in $(M, g)$. In higher dimensions 
the condition $\Pi >0$ is a weaker version of the convexity condition $A > 0$. 

  Turning to the Lorentzian setting of interest, the Lorentzian analog of the definiteness of $\Pi$ on $\dm$ is the statement that $\Pi$ has 
the same signature, up to overall sign, as the induced Lorentz metric $g_{\cC}$ on $\cC$. Throughout much of this paper, we then 
make the following analogous convexity-type assumption: 

\smallskip

{\sf Convexity Assumption} {(*)}. Up to an overall sign, the symmetric form $\Pi = Hg_{\cC} - A$ in \eqref{BY} is non-degenerate on $\cC$, with signature $(1, n-1)$, 
so induces a Lorentz metric on $\cC$. We assume in addition that $\Pi(\p_t,\p_t) < 0$, (cf.~Remark \ref{cones}). 

\smallskip 

  To illustrate on a simple example, let $S$ be any compact smooth domain in $\bR^n$, $n \geq 3$, with smooth convex boundary $\p S = \Si$. Let $M = \bR \times S$ 
be the timelike product over $S$ in Minkowski space $(\bR^{1,n}, g_0)$. Let $\nu_{\cC}$ be the (spacelike) unit outward normal to $\cC$ in $M \subset \bR^{1,n}$. Since 
$$A = \tfrac{1}{2}\cL_{\nu_{\cC}}g_0,$$ 
it is easy to see that $A(\p_t,\p_t) = 0$ but $A$ is positive definite, $A > 0$ in the spatial directions $S$. It follows easily that $\Pi$ is a Lorentz metric on 
$\cC$ of the same signature as $g_{\cC}$. Thus, all such configurations satisfy Assumption $(*)$. (The same holds for the exterior problem on $\bR^{1,n}\setminus M$). 
Notice that this fails for $n = 2$ where $\Pi$ for the timelike product is degenerate in the spatial direction. Suitable perturbations of the product in outward directions 
will however satisfy $(*)$. 

  In general, the Assumption $(*)$ is an extrinsic condition on the metric $g$ at the boundary $\cC$ and so cannot be determined from the induced metric 
$g_{\cC}$ alone. It is an open condition, so is preserved under small perturbations, and does not depend explicitly or directly on the initial data on a Cauchy slice $S$. 
It is also gauge-invariant, in fact invariant under general diffeomorphisms $\f: M \to M$ mapping $\cC \to \cC$, but not necessarily preserving the Cauchy surface 
$S$. Note also that it holds for natural far-field or near-infinity boundaries in asymptotically Minkowski or asymptotically AdS spacetimes. 

  Since the issue is local-in-time well-posedness, the convexity Assumption $(*)$ is essentially an assumption at the corner $\Si$. As will be seen in 
\S 2.3, the restriction $A|_{\Si}$ of $A$ to $\Si$ is in fact determined by the initial and boundary data $(g_S, K, g_{\cC})$ and their compatibility at $\Si$. 
It is only the time component $A(T_{\cC}, \cdot)$ where $T_{\cC}$ is the unit timelike normal to $\Si \subset (\cC, g_{\cC})$, which is undetermined by 
this data. 

\medskip 

  Let $\cE^*$ denote the space of smooth ($C^{\infty}$) vacuum Einstein metrics $g$ on $M$ satisfying the Assumption (*) at the boundary $\cC$. 
Clearly $\cE^*$ is an open subset of $\cE$. Similarly, as above $\cI_0$ and $Met(\cC)$ denote the space of $C^{\infty}$ vacuum initial data and boundary data. 

   The following result implies the local well-posedness of the IBVP for the Einstein equations with Dirichlet boundary data satisfying 
Assumption (*). It is the first such well-posedness result with respect to any geometric boundary data for the Einstein equations. 
  
\begin{theorem} \label{mainthm}
There exists a smooth time function $t^*$ for which the space $\cE^*$ is a smooth Frechet manifold and the map 
\be \label{mainPhi}
\begin{split}
&\Phi: \cE^* \to \cI_0 \times_c Met(\cC),\\
&\Phi(g) = (g_S, K, g_{\cC}),\\
\end{split}
\ee
is a smooth open embedding into the Frechet manifold $\cI_0 \times_c Met(\cC)$. Thus ${\rm Im} \, \Phi$ is an open set in 
$\cI_0 \times_c Met(\cC)$ and $\Phi$ is a smooth diffeomorphism onto its image. 

\end{theorem}

  Theorem \ref{mainthm} implies that, given any $g_0 \in \cE^*$ with initial and boundary data $((g_0)_S, K_0, (g_0)_{\cC})$, for any corner 
compatible initial and boundary data $(g_S, K, g_{\cC})$ near $((g_0)_S, K_0, (g_0)_{\cC})$, there exists $g \in \cE^*$, unique up to 
isometry in $\Diff_0(M)$ for a definite positive proper time $t^* > 0$ depending on the initial and boundary data with $\Phi(g) = (g_S, K, g_{\cC})$. 
In addition, such solutions $g$ depend smoothly on the initial and boundary data (smooth Cauchy-Dirichlet stablilty). 

  Although stated in the smooth ($C^{\infty})$ context, Theorem \ref{mainthm} also holds with respect to Sobolev-type spaces $H^s$ with a finite 
number of derivatives, cf.~Remark \ref{findiff}. 
 
\medskip
   
   There has been considerable prior work on the IBVP for the vacuum Einstein equations, beginning with the work \cite{FN} of Friedrich-Nagy; see 
for example \cite{KRSW1}, \cite{KRSW2}, \cite{FS1}, \cite{FS2} and \cite{ST} for a general survey. Recent work \cite{LSW}, \cite{LRSW} on the IBVP with conformal 
boundary data will be discussed in later work. To avoid possible confusion, we recall that the work in \cite{FS1} shows that the initial value problem (IVP) is locally 
well-posed for the boundary condition $A = 0$. However, since the boundary data is fixed, this is a Cauchy problem with fixed Neumann-type boundary data $A = 0$ 
on $\cC$. This result does not address the well-posedness of the IBVP with boundary data $A$. In fact such well-posedness, where the boundary data $A$ (as well 
as the initial data) is allowed to vary, is false in regions where $A = 0$, as shown in \cite{AA2}. 

\medskip 

  We note that Theorem \ref{mainthm} holds in any dimension and for vacuum solutions with any cosmological constant $\Lambda \in \bR$. The 
statements and proofs of the results to follow require only minor and essentially obvious modifications when $\Lambda \neq 0$. We will not consider 
here the coupling to other matter field equations. We also note that Theorem \ref{mainthm} holds equally for the exterior problem, where $\Pi$ changes sign, 
$\Pi \to -\Pi$, cf.~Remark \ref{flip}. In fact, although the Cauchy surface $S$ is assumed to be connected, the boundary $\p S = \Si$ may have a finite number 
of distinct components. 

\medskip 

  In general it is not easy to understand the image of $\Phi$. There is no simple intrinsic characterization of the Assumption (*), i.e.~whether $(*)$ holds or 
not cannot usually be determined in terms of Dirichlet boundary data alone. 
  
  Consider for instance the simpler situation in $2+1$ dimensions: all Lorentz vacuum Einstein metrics in $2+1$ dimensions are 
flat (or constant curvature if $\Lambda \neq 0$). Solving the Dirichlet BVP for such metrics is thus the same as finding isometric embeddings 
of a timelike cylinder 
$$(I \times S^1, \g) \subset \bR^{1,2},$$
(bounding a compact solid cylinder in $\bR^{1,2}$), with prescribed Lorentz metric $\g$ into Minkowski $\bR^{1,2}$. This is the Lorentzian analog of 
the well-known isometric embedding (or immersion) problem of surfaces in $\bR^3$. The classical solution of the Weyl embedding problem by Nirenberg 
and Pogorelov gives the existence of a unique solution in the Euclidean setting when the Gauss curvature satisfies $K_{\g} > 0$. It is easy to see that this 
is exactly the condition that $\Pi$ is a positive definite form, i.e.~the Riemannian version of Assumption $(*)$. Thus Theorem \ref{mainthm} is the 
Lorentzian analog of the Weyl embedding theorem in $2+1$ dimensions. 

   Of course, the condition $K_{\g} > 0$ is an intrinsic condition on the boundary metric $\g$, by Gauss' Theorema Egregium. In the Lorentzian $2+1$ 
setting, Assumption (*) implies the intrinsic condition, 
\be \label{K0}  
K_{\g} > 0. 
\ee  
However, \eqref{K0} is not sufficient to imply Assumption $(*)$. This is due to the fact that symmetric forms $\Pi$ as in \eqref{BY}, although diagonalizable with 
respect to Riemannian (positive definite) metrics, are not necessarily diagonalizable with respect to metrics of Lorentz signature. Thus there are timelike surfaces in 
$\bR^{1,2}$ with positive Gauss curvature for which Assumption $(*)$ does not hold, cf.~\S 8 for further discussion. On the other hand, if $\Pi$, (i.e.~$A$), is 
diagonalizable with respect to $\g$, then \eqref{K0} does imply Assumption $(*)$. 
 
  In the Euclidean setting, very little is known about a general theory for isometric embedding or immersion of closed surfaces in $\bR^3$ 
beyond the convex case $K_{\g} > 0$ (or $K_{\g} \geq 0$), cf.~\cite{HH}, \cite{Sp}. Once the Gauss curvature is allowed to change sign, many complicated 
behaviors may occur. Even the local isometric embedding for surfaces in $\bR^3$ remains an open problem in general. By the remarks above, all 
of this holds for Dirichlet boundary data for Riemannian vacuum Einstein metrics in $3$ dimensions. It seems likely that similar vast complications hold for 
Lorentzian vacuum Einstein metrics in $2+1$ dimensions, cf.~also \S 8 for further discussion. It would of course be interesting to explore this further. 
  
\medskip 

  As indicated above, the proof of the well-posedness of the IBVP for Dirichlet boundary data does not proceed by using a suitable gauge reduction to 
a hyperbolic IBVP for which one can apply a standard, existing theory of well-posedness; in fact we believe that such an approach is not feasible. Instead, 
due to essential loss-of-derivative issues associated with Dirichlet boundary data, we use the Nash-Moser implicit function theorem in Frechet spaces 
to prove that well-posedness at the linear level implies well-posedness at the nonlinear level. This requires of course a considerable detailed analysis of the 
linearized problem. 

\medskip 

  A brief summary of the contents of the paper is as follows. In \S 2, we introduce preliminary material needed for the work to follow. Due to the use of 
Nash-Moser, this section is quite extensive, in particular with regard to the exact choices of spaces involved as well as compatibility conditions at the 
corner $\Si$. In \S 3, we derive the basic equations for the boundary data not part of the Dirichlet boundary data. Using these equations, we 
then derive the local apriori energy estimates for the (shifted) linearized problem which are crucial in order to be able to apply the nonlinear Nash-Moser theory. 
This is first used to construct local linearized solutions to the Dirichlet IBVP in \S 5. The results of \S 4-\S 5 are then globalized (in space) to the compact 
setting $(M, g)$ in \S 6. The Nash-Moser theory is then employed to prove the main result, Theorem \ref{mainthm} in \S 7. In \S 8, we conclude with 
several remarks, including a discussion of the $2+1$ dimensional case and present examples where well-posedness of the Dirichlet IBVP breaks down. 
Finally in the Appendix we prove several (less important) technical results not given in detail in the main text. 

\medskip 
 
{\bf Acknowledgments:} We thank D. Anninos, D. Galante, C. Maneerat, D. Marolf and E. Silverstein for interesting discussions related to this work.

\section{Initial Material}

  In this section, we describe the overall approach and detail the rather extensive preliminary material needed for the proof of Theorem \ref{mainthm}. 
  
  Let $S$ be a compact, connected and oriented $n$-manifold with boundary $\p S = \Si$; here $\Si$ may or may not be connected. Let $M = I\times S$, 
where $I = [0,1]$ is parametrized by a  time function $t: M \to I$. Let $S_t = \{t\}\times S$; this will be identified with $S_0$ upon introduction of local coordinates. 
The choice of time interval $[0,1]$ is for convenience; the value $1$ could be replaced by any other fixed constant. We consider Lorentz metrics $g$ on $M$ 
which are globally hyperbolic in the sense of manifolds with boundary, where the boundary $\cC = I\times \Si$ is timelike with respect to the metric $g$; the 
induced metric $g_{\cC}$ is thus a globally hyperbolic metric on $\cC$ with closed Cauchy surface $\Si$. 
 
   Let $\nu_S$ denote the future timelike unit normal vector field to $S \subset (M, g)$ and similarly let $\nu_{\cC}$ denote the (spacelike) 
 outward unit normal to $\cC \subset (M, g)$. Then 
 $$K = \tfrac{1}{2}\cL_{\nu_S}g |_S \ {\rm and} \  A = \tfrac{1}{2}\cL_{\nu_{\cC}}g |_{\cC}.$$

   In local coordinates $x^{\a} = (t, x^a)$ near the corner $\Si$, $x^0 = t \geq 0$ is a defining function for $S = \{t = 0\}$, while $x^1 \leq 0$ is chosen 
to be a defining function for $\cC = \{x^1 = 0\}$. The coordinates $x^{\a}, \a = 2, \cdots, n$ are local coordinates of the corner $\Si$ itself. 
As usual, Greek letters $\a, \b$ denote spacetime indices $0, \dots, n$, Unless otherwise indicated, Roman letters denote spacelike indices $1, \dots, n$ 
and capital Roman letters denote corner indices $2, \dots, n$. 

\medskip 
  
  As noted in the Introduction, in the (elliptic) case of Riemannian metrics, it is well-known that there is a loss-of-derivatives problem for vacuum 
Einstein metrics with Dirichlet boundary data. This is even more severe in the case of Lorentz metrics and implies that one cannot expect to 
approach the well-posedness of the IBVP by standard methods. In this section, we develop the necessary preliminary results needed to apply 
the Nash-Moser implicit function theorem. 

\begin{remark} \label{cones}
{\rm Recall that Assumption $(*)$ requires that, up to overall sign, the symmetric form $\Pi = Hg_{\cC} - A$ induce a Lorentz metric on $\cC$ of the same type as 
$g_{\cC}$. Let $C_{\Pi}$ denote the future timelike cone of $\Pi$ and $C_{g_{\cC}}$ denote the future timelike cone of $g_{\cC}$ at any point 
$p \in \cC$. Then Assumption $(*)$ requires 
$$C_{\Pi} \cap C_{g_{\cC}} \neq 0.$$
We will always assume that 
\be \label{pt}
\p_t \in C_{\Pi} \cap C_{g_{\cC}},
\ee
with respect to some local coordinates near $\cC$; cf.~also Remark \ref{flip}. 

}
\end{remark}

\subsection{The basic mappings, spaces and norms.}

  To begin, we must choose a suitable function space. There are various possibilities, but we work with the spaces naturally associated to 
energy estimates for wave-type equations. Thus let $D$ denote a domain given as either $M$, $S$, $\cC$, $\Si$ or the corresponding $t$-level 
sets $S_t$, $\Si_t$, or the corresponding $t$-sublevel sets, $M_t$, $\cC_t$; ($M_t = \{p \in M: t(p) \leq t\}$). Define the Sobolev $H^s$ norm on 
functions on $D$ by 
 $$||v||_{H^s(D)}^2 = \int_{D} |\p_D^s v|^2dv_D,$$
where $\p_D^s$ consists of all partial coordinate derivatives tangent to $D$ of order $\leq s$. In place of coordinate derivatives, one may use the 
components of covariant derivatives of $v$ up to order $s$, given a (background) Riemannian metric $g_D$ on $D$. The volume form $dv_D$ is also 
induced from the metric $g_{D}$.  

  Define the stronger $\bar H^s$ norm by including all space-time derivatives up to order $s$, so that 
\be \label{barnorm}  
||v||_{\bar H^s(D)}^2 = \int_{D} |\p_M^s v|^2dv_D,
\ee
where $\p_M$ denotes partial derivatives along all coordinates of $M$ at $D$. 
 
  Finally, for the Cauchy slices $S_t$, define the boundary stable $H^s$ norm on $S_t$ by 
\be \label{bsnorm}
||v||_{\bar \cH^s(S_t)}^2 = ||v||_{\bar H^s(S_t)}^2 + ||v||_{\bar H^s(\cC_t)}^2.
\ee
  
  We will use the function space  
\be \label{N}
\cN^s(M) = C^0(I, \bar H^{s}(S)),
\ee
for functions on $M = M_1$. With the usual $C^0$ norm in $I$, the space $\cN^s$, with associated norm $|| \cdot ||_{\cN^s}$ is a 
separable Banach space. Norms such as \eqref{N} are commonly used for the space of solutions of wave-type equations, 
particularly regarding the Cauchy problem. Note however the use of the weaker $\bar H^s$ norm \eqref{barnorm} in \eqref{N} in 
place of the stronger boundary stable norm $\w \cH^s$ in \eqref{bsnorm}. 

 It is easy to see that 
$$H^{s+1}(M) \subset \cN^s(M) \subset H^s(M).$$

\begin{remark} 
{\rm We will frequently use the well-known Sobolev trace theorem, which states that the restriction $f \to f|_{\cC}$ has the 
associated bound  
\be \label{Sob1}
||f|_\cC||_{H^{s-1/2}(\cC)} \leq C ||f||_{H^s(M)}.
\ee
Similarly, 
\be \label{Sob2}
||f|_{\Si_t}||_{H^{s-1/2}(\Si_t)} \leq C ||f||_{H^s(S_t)} \ {\rm and} \ ||f|_{\Si_t}||_{H^{s-1/2}(\Si_t)} \leq C ||f||_{H^s(\cC_t)} .
\ee
Here $C$ is a constant depending only on the background metric $g_M$. In addition, we always assume $s > \frac{n+1}{2} + 2$, so that by the Sobolev 
embedding theorem, metrics $g \in Met(M)$ are at least $C^{2,\a}$. We note also the following Cauchy-Schwarz type inequality to be used 
in \S 4:
\be \label{Sob3}
\int_{\Si}f_1 f_2 \leq ||f_1||_{H^{1/2}(\Si)} ||f_2||_{H^{-1/2}(\Si)}.
\ee
We refer to \cite{Ad} for these results.   
}
\end{remark}

 Let $Met^s(M)$ be the space of globally hyperbolic Lorentz metrics $g$ on $M$, with coefficients $g_{\a\b} \in \cN^s$. The initial data 
 of $g$ is given by the pair $(g_S, K)$ consisting the metric $g_S$ induced on the initial Cauchy surface $S = S_0$ and the 
 second fundamental form or extrinsic curvature $K$ of $(S, g_S) \subset (M, g)$. This pair takes its values in the space 
 $\cI^s = Met^s(S)\times (S^2(S))^{s-1}$ of all initial data $(\g, \k)$, with the $H^s(S) \times H^{s-1}(S)$ topology.  
 
   For this work, the space of boundary data is 
 $$\cB^{s-1/2} = Met^{s-1/2}(\cC),$$
 of globally hyperbolic Lorentz metrics on the boundary $\cC$. Again, the bulk metric $g$ induces a boundary metric $g_{\cC} \in \cB^{s-1/2}$.

  We consider the map 
\be \label{Phi}
\Phi: Met(M) \to [S^2(M) \times \cI(S)\times \cB]_c : = \cT,
\ee
$$\Phi(g) = (\Ric_g, g_S, K, g_{\cC}),$$
where $(g_S, K)$ and $g_{\cC}$ are the initial and Dirichlet boundary data induced by $g$. The subscript $c$ denotes the compatibility 
conditions between the initial and boundary data at the corner $\Si$; these are discussed in detail in \S 2.3. For simplicity, we have dropped 
here the notation for the topologies, indexed by $s$. 

This is, roughly speaking, the equation we wish to solve and prove well-posedness for the IBVP; compare with \eqref{Phi0}.  

\medskip 

  Now, as is well-known, the mapping $\Phi$ in \eqref{Phi} is not surjective, i.e.~the equation $\Phi(g) = \tau$ is not solvable for general 
$\tau \in \cT$. This is due to the constraint equations, coupling the bulk term $\Ric(\nu_S, \cdot)$ in the (timelike) direction $\nu_S$ normal 
to $S$ with the initial data $(g_S, K)$. 

  To describe this in more detail, first extend $\Phi$ by adding on data for the normal vector $\nu_S$ of $S \subset M$. Let $\hat \cT = [\cT \times \bV_S']_c$ 
where $\bV_S'$ is the space of ($H^s(S)$) vector fields $\nu$ along $S$ nowhere tangent to $S$. We then consider the extended map 
\be \label{hatPhi}
\hat \Phi: Met(M)  \to \hat \cT,
\ee
$$\Phi(g) = (\Ric_g, g_S, K, \nu_S, g_{\cC}),$$
where $\nu_S$ is the future unit timelike normal to $S \subset (M, g)$. 
 
  The Hamiltonian and momentum constraint equations, (or equivalently the Gauss and Gauss-Codazzi equations) along $S$ are given by: 
\be \label{Gauss}
|K|^2 - H^2 - R_{g_S} - R_g + 2\Ric_g(\nu_S,\nu_S) = 0,
\ee
\be \label{GC} 
{\rm div}_{g_S} (K - H g_S) - \Ric_g(\nu_S, \cdot) = 0. 
\ee 
Here $H = \tr_{g_S}K$ is the mean curvature of $S$ and $R_{g_S}$ is the scalar curvature of $g_S$. These equations are identities on the domain 
$Met(M)$ but induce non-trivial equations on the target $\hat \cT$. To see this, note that 
$$- R_g + 2\Ric_g(\nu_S,\nu_S) = -\tr_{g_S}\Ric_g + 3\Ric_g(\nu_S, \nu_S),$$
and the latter is well-defined on $\hat \cT$. The rest of the terms in \eqref{Gauss}-\eqref{GC} are already defined on $\hat \cT$. Thus, let 
$\hat \cT_0 \subset \hat \cT$ be the subset satisfying \eqref{Gauss}-\eqref{GC}, so that 
\be \label{hatPhi}
\hat \Phi: Met^s(M)  \to \hat \cT_0.
\ee

  Note that there are also such constraint equations along the boundary $\cC$. However, the second fundamental form $A$ of $\cC \subset (M, g)$ does 
not appear as boundary data, so there are no meaningful constraint equations to impose, cf.~\S 3 however.

\subsection{Gauge fixing}

A clearer understanding of the image of $\Phi$ or $\hat \Phi$ above requires an understanding the action of the gauge group on the domain and target 
spaces. Thus, consider the gauge group $\Diff_0(M)$ of diffeomorphisms $\f: M \to M$ which restrict to the identity on $S\cup \cC$. The initial and boundary 
data are invariant under the action of $\Diff_0(M)$, while in the bulk, the action $(\f, \Ric) \to \f^*(\Ric)$ is not invariant in general, unless one is on-shell, 
i.e. on the space $\bE$ of vacuum solutions.

The choice of gauge is described by a choice of vector field $V$ on $M$, leading to the reduced Einstein equations 
\be \label{Q}
Q := \Ric_g + \d^*V = 0.
\ee
Geometrically, perhaps the most natural choice of gauge is the choice of (generalized) harmonic or wave coordinates 
\be \label{harm}
V_g = \Box_g x^{\a}\p_{x^{\a}},
\ee
where $\Box_g = |g|^{-1/2}\p_{\a}(g^{\a\b}|g|^{1/2}\p_{\b} \cdot)$. This gauge will be used throughout the paper. It is well-known that the gauge \eqref{harm} can be 
naturally globalized to all of $M$. For instance, suppose $M \subset \bR^4$ topologically and consider the Euclidean target metric $(\bR^4, g_{Eucl})$. More 
generally one may assume $M \subset \w M$ and let $g_R$ be a complete Riemannian metric on $\w M$. Then $V_g$ is defined to be the tension field of 
the identity wave map $(M, g) \to (\w M, g_R)$, cf.~\cite{GG}. This gives \eqref{harm} when $g_R = g_{Eucl}$. 

  The gauged version of the map $\Phi$ is then given by 
\be \label{PhiH}
 \Phi^H: Met(M) \to \cT^H,
 \ee
$$\Phi^H(g) = \big((\Ric_g + \d_g^* V_g), (g_S, K, \nu_S, V|_S), (g_{\cC}, V|_{\cC}) \big),$$ 
where $V|_S$ and $V|_{\cC}$ are the restrictions of $V_g$ to $S$ and $\cC$ respectively and $\cT^H = [\hat \cT_0 \times \bV_S \times \bV_{\cC}]_c$; we 
refer again to \S 2.3 for a detailed description of the target space $\cT^H$, cf.~\eqref{TH}.   

  The overall method of proof of well-posedness is to prove that $\Phi^H$ in \eqref{PhiH} is a tame Fredholm map with a tame approximate inverse at the 
linearized level and to then apply the Nash implicit function theorem in Frechet spaces. Thus, most of the analysis to follow deals with the study of the linearization 
\be \label{DF2}
D\Phi_g^H: T_g Met(M) \to T_{\tau}(\cT^H),
\ee
where $\tau = \Phi^H(g)$. Actually the understanding of the map $D\Phi_g^H$ will be developed through the study of the shifted map $D\w \Phi^H$ introduced in 
\S 2.5 below. 

\medskip 

  The bulk term in the linearization $D \Phi_g^H$ is given by 
 \be \label{L}
L(h) = \Ric'_h + \d^*V'_h + (\d^*)'_hV.
\ee
The linearization of the gauge field $V = V_g$ is given by 
\be \label{V'}
V'_h = \b_g h - \<D^2 x^{\a}, h\>\p_{x^{\a}},
\ee
where $\b_g = \d_g + \frac{1}{2}d\tr_g h$ is the Bianchi operator. 

The equation $L(h) = F$ then has the form 
\be \label{L2}
L(h) = \tfrac{1}{2}D^*D h + S(h) = F,
\ee
where $L$ is the well-known Lichnerowicz operator; $S$ is a zero order linear curvature operator on $h$ and $D^*D$ is the tensorial wave operator on symmetric 
bilinear forms. Of course all the coefficients in \eqref{L2} depend on the background metric $g$ at which the linearization is formed. 
 
  The equation \eqref{L2} may be written in local coordinates in the form 
$$L(h) = -\tfrac{1}{2}(\Box_g h_{\a\b})dx^{\a}dx^{\b} + P_{\a\b}(\p h)dx^{\a}dx^{\b} = F_{\a\b}dx^{\a}dx^{\b},$$
where $P(\p h)$ involves $h$ and its first derivatives. In the localized context discussed in \S 2.4, the coefficients of $P$ are 
all of order $\e$, $P(\p h) = \e(\p h)$. Since $L$ is uncoupled at leading order, it may be written as a coupled system of scalar 
wave equations 
 \be \label{L4}
(L(h))_{\a\b} = -\tfrac{1}{2} \Box_g (h_{\a\b}) + P_{\a\b}(\p h) = F_{\a\b}. 
\ee

  In later sections, we will repeatedly need to solve linear wave equations of the form \eqref{L4}, i.e. 
\be \label{L0}
-\tfrac{1}{2}\Box_g v + p(\p v)  = \f,
\ee
for given $\f$. Here $v$ may be either scalar valued or vector-valued. Although solvability of the linear equation 
or system \eqref{L0} holds more generally, we will only need to actually solve such systems in $C^{\infty}$, so with $g$, $p$ and 
$\f$ in $C^{\infty}$. 

 The next Lemma is well-known and will be used repeatedly in \S 5. 

\begin{lemma} \label{DirSomm}
For either Dirichlet or Sommerfeld boundary data, the $C^{\infty}$ system \eqref{L0} is solvable for any smooth $\f$, with standard energy estimates. 
\end{lemma}

\begin{proof} 

  The existence and uniqueness of solutions to \eqref{L0} in Sobolev spaces, with given initial data and with either Dirichlet or 
Sommerfeld boundary data is standard, cf.~\cite{BS}, \cite{Sa}. We assume implicitly here that the initial data and boundary data for $v$ satisfy the 
requisite compatibility conditions at the corner $\Si$, cf.~\S 2.3. Energy estimates for such equations are also standard; for later reference, 
we state the details here. 

  The strong or boundary stable energy estimate with Dirichlet boundary data states that any (smooth) solution of \eqref{L0} satisfies 
the bound 
\be \label{DirE}
||v||_{\bar \cH^s(S_t)}^2 \leq C[||v||_{H^s(S_0)}^2 + ||\p_t v||_{H^{s-1}(S_0)}^2 + ||v||_{H^s(\cC_t)}^2 + ||\f||_{H^{s-1}(M_t)}^2],
\ee
where $C$ is a constant depending only on $g$ and the coefficients of $p$. Note the difference in the stronger and weaker norms on the left 
and right of \eqref{DirE}. We recall that $\cC_t = \{p \in \cC: t(p) \leq t\}$ and similarly for $M_t$. 

  There is a similar estimate for Sommerfeld boundary data, where the boundary data is given by 
$$b(v) = (\p_t + \nu_{\cC})(v).$$
In this case, one has  
\be \label{SomE}
||v||_{\bar \cH^s(S_t)}^2 \leq C[||v||_{H^s(S_0)}^2 + ||\p_t v||_{H^{s-1}(S_0)}^2 + ||b(v)||_{H^{s-1}(\cC_t)}^2 + ||\f||_{H^{s-1}(M_t)}^2].
\ee
Similar results hold for $b(v) = (a \p_t + b \nu_{\cC})(v)$ provided $a > 0$, $b > 0$. For both \eqref{DirE} and \eqref{SomE}, we assume 
$t \in I = [0,1]$. Assuming $g$ is defined, these estimates do hold for large $t$, but with $C$ depending then in addition on $t$. 

\end{proof}

\begin{remark} \label{Neumann} 
{\rm One also has existence and uniqueness for smooth solutions of the equation \eqref{L0} with given Neumann boundary data $b(v) = \nu_{\cC}(v)$. 
However, in this case, there is no effective energy estimate as in \eqref{DirE} or \eqref{SomE}. Instead there is such an estimate but with a loss 
of derivative - or more precisely a loss of half-a-derivative, i.e.
\be \label{NeuE}
|v||_{\bar \cH^s(S_t)}^2 \leq C[||v||_{H^s(S_0)}^2 + ||\p_t v||_{H^{s-1}(S_0)}^2 + ||\nu_{\cC}(v)||_{H^{s-1/2}(\cC_t)}^2 + ||\f||_{H^{s-1}(M_t)}^2].
\ee
We refer to \cite[Ch.8]{MT}, \cite{Tat} for further details. In some situations, the bound on $||\nu_{\cC}(v)||_{H^{s-1/2}(\cC_t)}$ can be improved to a stronger 
$H^s$ bound. However, it cannot be improved to an effective $H^s$ bound as in \eqref{DirE} or \eqref{SomE}. 
}
\end{remark}

   We next discuss the gauge term $V$ and its linearization $V'_h$. Consider the equation  
\be \label{V0}
\Phi^H(g) = \Ric_g + \d_g^*V = Q, 
\ee
where $Q$ is given data in the target data space $\cT^H$. We will need the following results on the gauge field. 
 
\begin{lemma}\label{Gauge-lemma}
Let $\Box_g$ denote the wave operator $\Box_g = - D^*D$ acting on vector fields $V$ on $M$. Then  
\be \label{V}
-\tfrac{1}{2}[\Box_g + \Ric_g](V) = \b_g(Q).
\ee
The initial data $V|_S$, $(\p_t V)|_S$ and boundary data $V|_{\cC}$ are determined by the target data  
$$(\Ric + \d^*V, (g_S, K_S, \nu_S, V|_S), V|_{\cC})$$
in $\cT^H$ and hence $V$ on $M$ is uniquely determined by target data. 

\end{lemma}

\begin{proof}
Applying the Bianchi operator $\b_g = \d_g + \frac{1}{2}d \tr_g$ to \eqref{V0} gives 
\be \label{V2}
\b_g \Ric_g + \b_g \d^*V = \b_g(Q),
\ee
which, by the Bianchi identity $\b_g \Ric_g = 0$, gives \eqref{V} via a standard Weitzenbock formula 

The Dirichlet data for $V$ along $S$ and $\cC$ are target data in $\cT^H$. We claim that the $t$-derivative $\p_t V$ on $S$ is also 
determined by target data. This follows from the Gauss and Gauss-Codazzi identities \eqref{Gauss}-\eqref{GC}. For $\nu = \nu_S$ the 
unit timelike normal to $S$, these identities show that the form $E(\nu, \cdot) = \Ric(\nu, \cdot) - \frac{1}{2}R g(\nu, \cdot)$ is determined by 
initial data $\iota = (g_S, K_S)$ along $S$. By \eqref{V0}, it follows that $\d^*V(\nu, \cdot) - \frac{1}{2}\tr \d^*V g(\nu, \cdot)$ is determined by 
initial data. Since $V$ is determined along $S$ by the target data, it follows easily that $\nabla_{\nu}V$ and hence $\p_t V$ is determined 
along $S$ by target data. The standard existence and uniqueness of solutions to the wave equation \eqref{V} with given initial 
and Dirichlet boundary data then shows that $V$ is uniquely determined by target data. 

\end{proof}

 A similar result holds for the linearization $V'_h$ of the gauge field $V = V_g$. 
\begin{lemma}\label{Gauge-lemma2}
We have 
\be \label{V'eqn}
-\tfrac{1}{2}(\Box_g + \Ric_g)V'_h + \tfrac{1}{2}\nabla_V V'_h = \b_g F + \b'_h \Ric_g + O_{2,1}(V,h),
\ee
where $O_{2,1}(V,h)$ is $2^{\rm nd}$ order in $V$ and $1^{\rm st}$ order in $h$. Moreover, $O_{2,1}(V,h) = 0$ if $V = 0$. 

  The initial data $(V'_h)|_S$, $(\p_t V'_h)|_S$ and boundary data $V'_h|_{\cC}$ are determined by the target data  
$$\big((Ric + \d^*V)'_h, (h_S, K'_h, (\nu_S)'_h, V'_h |_S), V'_h |_{\cC} \big)$$
in $T(\cT^H)$.
\end{lemma}

\begin{proof}
Again applying the Bianchi operator to both sides of \eqref{L} gives:
\bes
\b_g \Ric'_h+\b_g \d^* V'_h+\b_g[(\d^*)'_h V]=\b_g F
\ees
which via the Weitzenbock formula as before implies 
\be \label{V'1}
-\tfrac{1}{2}[\Box V'_h+\Ric_g(V'_h)]=\b_g F + \b'_h \Ric_g -\b_g[(\d^*)'_h V].
\ee
Simple calculation gives $(\d^*)'_hV = \frac{1}{2}\nabla_V h + \d^*V\circ h$, so that 
 $$\b[(\d^*)'_hV] = \tfrac{1}{2}\b(\nabla_Vh) + O_{2,1}(V,h) = \tfrac{1}{2}\nabla_V \b h + O_{2,1}(V,h) = \tfrac{1}{2}\nabla_V V'_h + O_{2,1}(V,h),$$
where we have used \eqref{V'} in the last equality. This gives \eqref{V'eqn} 

  As in Lemma \ref{Gauge-lemma}, the Dirichlet data for $V'_h$ along $S$ and $\cC$ are given as target space data and we use the 
constraint equations \eqref{Gauss}-\eqref{GC} to determine the initial velocity $\p_t V'_h$. As before, the bulk equation yields:
\bes
\begin{split}
&(\Ric-\tfrac{1}{2}Rg)'_h(\nu_S)+[\d^*V-\tfrac{1}{2}({\rm div}V)g]'_h(\nu_S)\\
&=[F-\tfrac{1}{2}(\tr F)g](\nu_S)+\tfrac{1}{2}\<h,\Ric_g+\d^* V_g\>g(\nu_S)-
\tfrac{1}{2}\tr(\Ric_g+\d^* V_g)h(\nu_S)
\end{split}\ees
and thus
\bes\begin{split}
&[\d^*V-\tfrac{1}{2}({\rm div}V)g]'_h(\nu_S)\\
&=-[(\Ric-\tfrac{1}{2}Rg)(\nu_S)]'_h+(\Ric-\tfrac{1}{2}Rg)((\nu_S)'_h)\\
&+[F-\tfrac{1}{2}(\tr F)g](\nu_S)+\tfrac{1}{2}\<h,\Ric_g+\d^* V_g\>g(\nu_S)-
\tfrac{1}{2}\tr(\Ric_g+\d^* V_g)h(\nu_S)
\end{split}\ees
on $S$.
By the constraint equations \eqref{Gauss}-\eqref{GC}, $[(\Ric-\tfrac{1}{2}Rg)(\nu_S)]'_h$ is given by 
\bes
\big(-\tfrac{1}{2}[|K_S|^2-(\tr_{g_S}K_S)^2+R_S]'_{(h_S,K'_h)},~[{\rm div}_{g_S}K_S-d_S(\tr_{g_S}K_S)]'_{(h_S,K'_h)}\big).
\ees
Thus the target data in $T(\cT^H)$ uniquely determine the vector field 
\bes
[\d^*V'_h](\nu_S)-\tfrac{1}{2}({\rm div}V'_h)g(\nu_S),
\ees
along $S$. Since the initial data of $V'_h$ is already determined, this uniquely determines the vector field 
\be
\nabla_{\nu_S} V'_h \ \ {\rm along} \ \ S.
\ee

\end{proof}

\begin{remark} \label{gaugerem} 
{\bf (i).}
{\rm Lemma \ref{Gauge-lemma} implies the standard result that solutions of the gauge-reduced Einstein equations 
$$\Ric_g + \d^*V = 0,$$
with target data $ V = 0$ at $S\cup \cC$ satisfy $V = 0$ on $M$, and so are solutions of the Einstein equations 
$$\Ric_g = 0$$
on $M$. The same result holds locally, in the context of the localization in \S 2.4, as well as for the linearized equations \eqref{V'eqn}. 

  Conversely, given any vacuum Einstein metric $g$, there is a unique diffeomorphism $\f \in \Diff_1(M)$ such that $\w g = \f^*g$ is in 
harmonic gauge, $V_{\w g} = 0$. Here $\Diff_1(M) \subset \Diff_0(M)$ is the group of smooth diffeomorphisms equal to the identity to 
first order on $S$. 
 
{\bf (ii).} As will be seen below, cf.~the discussion preceding \eqref{Va}, the initial data $h_{\a\b}$, $\p_t h_{\a\b}$ for $h$ along $S$ are determined 
by the target data in $T(\cT^H)$. 

{\bf (iii).} For later use, note that the gauge variation $V'_h$ satisfies the boundary stable energy estimate \eqref{DirE}, since the Dirichlet boundary 
value $V'_h |_{\cC}$ is part of the target data in $T(\cT^H)$. 

}
\end{remark}

\subsection{Geometry at the corner}
Next we discuss the compatibility conditions on the initial and boundary data at the corner $\Si$, and the corresponding structure of the metric 
at $\Si$. 

  Any smooth metric $g$ on $M$ induces initial data on $S$ and boundary data on $\cC$. The smoothness of $g$ imposes compatibility 
conditions on these sets of data. Conversely, the compatibility conditions on the (abstract or target) initial and boundary data are the conditions 
that such smooth data extend to a smooth metric on $M$. 
 
    In more detail, let $\cU$ denote the uncoupled space of target data, so that 
 $$\cU =  Sym^2(M) \times \cI \times \bV_S' \times \bV_S \times \cB \times \bV_{\cC}.$$
This is a product space, where $S^2(M)$ is the space of symmetric forms $Q$ with the $\cN^{s-2}$ topology, (cf.~\eqref{V0}), $\cI$ is the space of 
initial data $(\g_S, \k)$ with the $H^s(S) \times H^{s-1}(S)$ topology, $\bV'_S$ is the space of $H^s$ vector fields along $S$ nowhere tangent to $S$, 
$\bV_S$ is the space of $H^{s-1}$ vector fields $V_S$ along $S$, $\cB$ is the space of boundary data $\g_{\cC}$ with the $H^{s-1/2}$ topology as in \S 2.1 
while $\bV_{\cC}$ is the space of $H^{s-3/2}$ vector fields $V_{\cC}$ along $\cC$. By construction, $\cU$ is a separable Hilbert space. 

  The metric $g$ induces data on $\cU$ through the map $\Phi^H$. The two equations for $g$, namely the bulk equation $\Ric_g + \d^*V_g = Q$ in \eqref{V0}
and the definition of the gauge field $V_g$ in \eqref{harm}, as well as definition of the normal vector $\nu_S$ in \eqref{nuS}, determine compatibility 
conditions for $g$ at the corner $\Si$. These need to be expressed in terms of equations for the data in $\cU$. The gauged target space   
$$\cU_c = [\cI \times \bV_S' \times \bV_S \times \cB \times \bV_{\cC}]_c \subset \cU $$
is then the subspace satisfying these compatibility conditions. The compatibility conditions will be expressed as the zero-locus of a (large) collection of 
functions on $\cU$, ordered according to the degree of differentiability. 

\medskip 

  We note that these compatibility conditions are invariant under smooth diffeomorphisms $\f$ which 
 map $\f: S \to S$, $\f: \cC \to \cC$ and $\f: \Si \to \Si$. Due to this, we may assume without loss of generality that the metric $\g_{\cC}$ on $\cC$ 
 is in Gaussian (or Fermi) geodesic coordinates, i.e.
 $$\g_{\cC} = -dt^2 + \g_{\Si_t}$$
 where $\g_{\Si_t}$ is a curve of metrics on $\Si_t$. Similarly, we may assume $\g_S$ has the form 
 $$\g_S = dx_1^2 + \g_{\Si_{x^1}},$$
 where $\g_{\Si_{x^1}}$ is a curve of metrics on the level sets $\Si_{x^1}$ of $x^1: S \to \bR$. These assumptions simplify some of the discussion to follow. 
 
\medskip  
 
  We begin with the space $\cI \times \cB$ of initial data $(\g_S, \k)$ and boundary data $\g_{\cC}$. The $C^0$ compatibility of this data requires that 
\be \label{Com0}
(\g_{S}  -  \g_{\cC})|_{\Si}  = 0 .
\ee
This is the only $C^0$ compatibility condition between initial and boundary data $\cI \times \cB$ in $\cU$. It follows that the full ambient metric $g$ is determined 
at the corner, except possibly for the angle term 
\be \label{alpha}
\a = \<\nu_S, \nu_{\cC}\>
\ee
between $S$ and $\cC$; here $\nu_S$ is the future unit normal to $S$ in 
$(M, g)$ and $\nu_{\cC}$ is the outward unit normal to $\cC$ in $(M, g)$. From the geodesic normalizations above, it is easy to see 
that $\p_t$ is the future unit normal of $\Si \subset (\cC, \g_{\cC})$ and $\p_{x^1} = N$ is the unit outward normal of $\Si \subset (S, \g_S)$. 
Also, 
$$\a = \<\nu_S, \nu_{\cC}\> = -\<\p_t, \p_{x^1}\> = -g_{01},$$
and 
\be \label{TC}
\p_t = \sqrt{1+\a^2}\nu_S -\a \p_{x^1}.
\ee

  The following Lemma shows that $\a$ is in fact determined by the initial and boundary data under Assumption $(*)$, so that 
$$\a = \a(\g_S, \k, \g_{\cC}),$$
is well-defined as a function on $\cI \times \cB$ and hence on $\cU$ via the projection $\cU \to \cI \times \cB$.  

\begin{lemma}\label{alpha}
For any $C^1$ metric $g$, one has the relation 
\be \label{vel}
H_{\Si}^{\cC} = \sqrt{1+\a^2}\tr_\Si \k - \a H_\Si^S,
\ee
where $H_{\Si}^S$ and $H_{\Si}^{\cC}$ are the mean curvatures of $\Si$ in $(S, \g_S)$ and $(\cC, \g_{\cC})$ respectively. 
Consequently, under Assumption (*), $\a$ is determined by the data $(\g_S, \k, \g_{\cC})$. 
\end{lemma}

\begin{proof}
The formula \eqref{vel} follows directly by taking the trace of the covariant derivative of $\p_t$ in \eqref{TC}. 
Since all the terms $H_{\Si}^S$, $H_{\Si}^{\cC}$ and $\tr_{\Si} \k$ are determined by the initial and boundary data, it follows 
that $\a$ is uniquely determined, up to a sign choice, by the initial and boundary data, provided one of $\tr_{\Si}K \neq 0$ or $H_{\Si}^S \neq 0$. 
However, if both are zero, then $tr_{\Si}A_{\cC} = 0$. This is impossible by Assumption (*). 

Given a fixed value for the initial and boundary data, and so for $H_{\Si}^{\cC}, \tr_{\Si}\k, H_{\Si}^S$, the (generically) two values $\a_{\pm}$ of $\a$ 
satisfying \eqref{vel} are mapped to each other by the sign or orientation change $\nu_S \to -\nu_S$. As will be seen later from the proof of 
Theorem \ref{mainthm}, this orientation change, corresponding to future or past evolutions, gives rise to solutions $(M^{\pm}, g^{\pm})$ which 
are isometric under an orientation reversing diffeomorphism, cf.~Remark \ref{flip} for concrete examples. Thus fixing the choice of orientation 
fixes the angle $\a$. 

\end{proof} 

 We note that without Assumption (*), so if both terms $\tr_{\Si}K$, $H_{\Si}^S$ above vanish, perhaps even in a small open set in $\Si$, 
then $\a$ remains undetermined; this may reflect a lack of well-posedness/uniqueness when Assumption $(*)$ is dropped.

Thus the ambient metric $g_{\a\b}$ is uniquely determined along $\Si$ by $(\g_S, \k, \g_{\cC})$. The $C^1$ compatibility condition is then given by 
\be \label{LH}
\cL_{\p_0}\g_{\cC} |_{\Si}  - \sqrt{1+\a^2}\k |_{\Si}  + \a B_{\Si} = 0,
\ee
where $B_{\Si} = \cL_{\p_1}\g_S |_{\Si}$ is (twice) the second fundamental form of $\Si \subset (S, \g_S)$. Note that the $\g_{\Si}$ trace of 
\eqref{LH} is equivalent to \eqref{vel}. We write this in local coordinates as 
\be \label{Com1}
\p_0 \g_{AB} - \chi_1(\p \g_S, \k, \p g_\cC) = 0.
\ee
Here and in the following, $\chi, \zeta, \xi$ will denote general smooth functions of the jets of the indicated data in $\cU$ at $\Si$, with the subscript indicating 
the leading order of derivatives. The equation \eqref{Com1} implies that the full second fundamental form of $\Si \subset (\cC, \g_{\cC})$ is determined by its 
trace, i.e.~by $H_{\Si}^{\cC}$, given $(\g_S, \k)$. In the following, the specific form of $\chi, \zeta, \xi$ may change from line to line, but without a change in the notation. 

\begin{remark} \label{corner}
{\rm The discussion above shows that the $C^1$ jet of the metric $g$ at $\Si$ is determined by the initial and boundary data of $g$ in $\cU$ and 
their $C^0$ and $C^1$ compatibility conditions, except for the components $\p_t g_{01}, \p_1 g_{0\a}$ on $\Si$ in local coordinates. The same 
result holds for infinitesimal deformations $h$ of $g$. 

  Geometrically, the terms $\p_1 g_{0i}$, $i = 0,2, \dots, n$ correspond to the component $A(\p_0, \cdot)$ of the second fundamental 
form $A = A_{\cC}$ of $\cC$ in $(M, g)$, (since $\p_1$ is a linear combination of $\nu_{\cC}$ and $\p_t$ and the $\p_0 g_{0i}$ components 
are determined by $\g_{\cC}$). The remaining two terms $\p_0 g_{01}$ and $\p_1 g_{01}$ correspond to the time and radial variation of the 
corner angle in adapted coordinates. 
}
\end{remark}

  Next we discuss the compatibility conditions on the remaining initial and boundary data $\nu = \nu_S$ and $V_S$, $V_{\cC}$. First for $\nu$, 
we have  
\be \label{nuS}
\nu_S = -(|\nabla t|)^{-1}\nabla t = -(|\nabla t|)^{-1}g^{0\a}\p_{\a}, \  |\nabla t| = \sqrt{-g^{00}}.
\ee
At $\Si$, this gives the $C^0$ compatibility condition 
\be \label{Comnu}
\nu - (1+\a^2)^{-1/2}(\p_t - \a \p_{x^1}) = 0,
\ee
so that  
$$\nu - \chi_0(\a) = 0.$$
This couples the $C^0$ behavior of $\nu$ at $\Si$ to the metric data $(\g, \k, \g_{\cC})$ through the determination of $\a$. There are no 
$C^1$ compatibility conditions between $\nu$ and $(\g, \k, \g_{\cC})$, so the derivatives $\p_1 \nu^{\b} = \p_1 g^{0\b}$ (or equivalently the derivatives 
$\p_1 g_{0\b}$) may be freely prescribed. By Remark \ref{corner}, prescribing these prescribes the full first order jet of $g_{\a\b}$ except for the term 
$\p_0 g_{01}$. 

  Next we discuss the gauge field $V$. The data $V_S$ and $V_{\cC}$ must satisfy the $C^0$ condition 
\be \label{ComVo}
(V_S - V_{\cC}) |_{\Si} = 0.
\ee
This is the only compatibility condition for $V$ itself; it ensures that $V_S$ and $V_{\cC}$ extend to a $C^0$ vector field on $M$, when 
$V_S$ and $V_{\cC}$ are $C^0$ vector fields on $S$ and $\cC$ respectively. 
  
  Next, the gauge field $V$ couples to the other data $(\g_S, \k, \nu, \g_{\cC})$ by the definition \eqref{harm}, which we rewrite in the form 
\be\label{Vg}
V^\a=\Box_g x^\a=\p_\mu g^{\mu\a}+\tfrac{1}{2}g^{\a\mu}g^{\rho\s}\p_\mu g_{\rho\s}.
\ee
Setting $\a=1$ gives (since $g^{11} = 1$ along $S$)  
\bes
\begin{split}
V^1&=\p_0g^{01} + \tfrac{1}{2}g^{10}\p_0 g_{11}+g^{01}g^{01}\p_0 g_{01}+\tfrac{1}{2}g^{10}g^{AB}\p_0 g_{AB} +g^{01}\p_1 g_{01}+\tfrac{1}{2}g^{AB}\p_1 g_{AB}\\
&=\p_0g^{01}+(g^{01})^2\p_0g_{01}+\chi_1(\p \g_S,\k,\p \nu,\p \g_\cC).
\end{split}
\ees
Notice that $\p_0 g_{01}$ is not defined on $\cU$ so this is not a compatibility equation. Instead, this equation determines $\p_0 g_{01}$ (equivalent to $\p_0 g^{01}$ 
up to determined terms):
\bes
\p_0g_{01}=\chi_1(V^1,\p \g_S, \k,\p \nu,\p \g_\cC).
\ees 
It follows that all the first derivatives $\p_\rho g_{\a\b}$ are determined by the data $(V^1,\p \g_S,\k,\p \nu,\p \g_\cC)$ at $\Si$. Hence the components 
$V^a~(a\neq 1)$ are uniquely determined and satisfy the compatibility condition:
\be\label{Va}
V^a=\chi_1^a(V^1,\p \g_S,\k,\p \nu,\p \g_\cC), \ \ a=0,2,\dots, n.
\ee

\medskip 

Let $\cU_c^1 \subset \cU$ be the subspace satisfying the first order compatibility conditions above. The data in $\cU_c^1$ 
determines the full 1-jet of $g_{\a\b}$ at $\Si$. For simplicity of notation, we suppress here the index $s$ denoting the $H^s$ differentiability 
class. 
 
  We will need the Lemma below and those to follow for the non-linear analysis in \S 7. 
 
\begin{lemma}\label{Man1}
For $s$ sufficiently large (depending only on the dimension $n$), the space $\cU_c^1$ is a closed Hilbert submanifold of $\cU$. 
\end{lemma}
\begin{proof}
    Define the following map 
\be \label{CompOp1}
W^1: \cU \to Sym(\Si)^2 \times \bV'(\Si) \times \bV(\Si) \times \bV_{\cC}(\Si),
\ee
\bes
W^1(\g_S, \k, \nu_S, V_S, \g_{\cC}, V_{\cC}) = 
\left\{ 
\begin{array}{l}
(\g_S - \g_{\cC})|_{\Si} \\
\p_t \g_{\cC} - \chi(\p \g_S, \k, \p g_\cC) \\
\nu - \chi(\a) \\
(V_{\cC} - V_S)|_{\Si} \\
V_S^a - \chi^a(V_S^1,\p \g_S,\k,\p \nu,\p \g_\cC) \\
\end{array}\right.
\ees
In the above, all data on the right are evaluated at $\Si$; also $\bV_{\cC}(\Si)$ denotes vector fields at $\Si$ tangent to $\cC$. 
By definition, the inverse image $W^{-1}({\bf 0}) = \cU_c^1$. 

  We claim that $W^1$ is a submersion at $\cU_c^1$, i.e.~the derivative $DW^1$ is everywhere surjective on $\cU_c^1$. To see this, consider 
infinitesimal deformations of the form $(0, 0, \nu', V_S', \g_{\cC}', V_{\cC}' ) \in T(\cU)$. Then 
\bes 
\begin{split}
& DW^1(0, 0, \nu', V_S', \g_{\cC}', V_{\cC}' ) = 
\begin{cases}
- \g'_{\cC} \\
 \p_t (\g_{\cC}') - \chi(\a')\\
 \nu' - \chi(\a')\\
 V_S' - V_{\cC}'\\
 (V_S)^a)' - \chi^a((V_S^1)', \nu_S', \g_{\cC}') \\
 \end{cases}
\end{split}
\ees

  Again, all data on the right are evaluated at $\Si$. Using the triangular structure, it is easy to see this map is surjective. First $\g_{\cC}'$ may be prescribed 
arbitrarily. Next, $H'_{\cC}$, which gives $\a'$ via \eqref{vel}, may be arbitrarily prescribed, as is the trace-free part of $\p_0 \g_{\cC}'$. Having thus fixed the 
choice of $\a'$, $\nu'$ may be arbitrarily prescribed, and similarly for the terms $(V_S^a)'$. Given these choices, $V'_{\cC}$ 
may then also be freely and independently prescribed. This proves the surjectivity.

The kernel $Ker DW^1$ is thus a closed subspace of a Hilbert space, and so has a closed complement. It then follows from the implicit function theorem for 
Hilbert manifolds that $\cU_c^1 = W^{-1}({\bf 0})$ is a closed Hilbert submanifold of $\cU$. 

\end{proof}

  Next we consider the second and higher order compatibility conditions, as further conditions on the data in $\cU_c^1$. These are governed the coupling of the 
initial and boundary data above to the bulk equation $Q = \Ric_g + \d^*V_g$ in $\cT^H$ and to derivatives of the defining equation \eqref{harm} for $V$. 

Consider first the corner part $\g_{AB}$ of the metric. As in \eqref{L4}, this has the form 
$$-\tfrac{1}{2}\p_0\p_0 \g_{AB} + \tfrac{1}{2}g^{01}\p_{1}\p_0 \g_{AB} + \tfrac{1}{2}g^{ij}\p_i\p_j \g_{AB} + P_{AB}(\p g) - Q_{AB} = 0,$$
at $\Si$. Since $P(\p g)$ is first order in $g$, it is determined by the data in $\cU_c^1$, so $P(\p g) = \chi_1(\cU_c^1)$. Next, for the second term above, 
by \eqref{Com1}, $\p_1 \p_0 \g_{AB} = \p_{1}(\sqrt{1+\a^2}\k |_{\Si} - \a B_{\Si}) = \chi_2(\cU_c^1)$. Thus we can write the equation above as 
\be \label{com2g}
-\p_0\p_0 \g_{AB}  = \zeta_2(Q, \p^2\g_S, \p \k, \p \nu, V, \p \g_{\cC}). 
\ee
This compatibility equation couples the second order jet of $\g_{\cC}$ in the $t$-direction with the second order jet of $\g_S$ in the $x^1$ direction with data previously 
determined on $\cU_c^1$. 

  Next we discuss the compatibility conditions on $\nu_S$, followed by those on $V_S, V_\cC$. For simplicity, we assume below that the corner angle $\a=0$ along $\Si$; 
the analysis to follow works equally well for general corner angle with slightly more involved calculations. 

Recall that the timelike unit normal vector $\nu_S$ is given by $\nu_S = g^{0\a}\p_{\a}$ and it is easy to see that 
\be\label{nu11}
\p_1^2g_{0A}=-g_{AB}\p_1^2 \nu_S^B+O_1(g), \ \p_1^2 g_{01}=-\p_1^2\nu_S^1+O_1(g), \ \p_1^2 g_{00}={2}\p_1^2\nu_S^0+O_1(g)
\ee
where $O_1( g)$ is a term involving the $0,1$ jets of $g$. Now consider the bulk equation for $g_{0a}~(a=0,2, \dots, n)$ along the corner:
\bes
\begin{split}
0&=-\tfrac{1}{2}\p_0\p_0 g_{0a} +g^{01}\p_{1}\p_0g_{0a}  +\tfrac{1}{2}\p_1^2 g_{0a}+ P_{0a}(\p g) - Q_{0a} \\
&=\tfrac{1}{2}\p_1^2 g_{0a}+ P_{0a}(\p g) - Q_{0a}.
\end{split}
\ees
(If $g^{01} = \a \neq 0$, then the term $\p_1\p_0 g_{0a}$ couples to the term $\p_1 V^a$, as in the $\theta$ term below). 
Combining this with the first and last equation in \eqref{nu11} yields the following compatibility condition on $\p^2_1\nu_S^a$:
\be\label{com2nua}
\p^2_1\nu_S^a = \xi_1^a(Q,\p \g_S, \k,\p\nu,\p \g_\cC, V), \ \ a=0,2,\dots, n.
\ee
Next consider the bulk equation for $g_{01}$:
\be\label{b01}
\begin{split}
0&=-\tfrac{1}{2}\p_0^2 g_{01}  +\tfrac{1}{2}\p_1^2 g_{01}+ P_{01}(\p g) - Q_{01}. 
\end{split}
\ee
The second term is proportional to $\p_1^2 \nu_S^1$ and we claim this yields a compatibility equation of the form 
\be\label{com2nu1}
\p_1^2\nu_S^1=\xi_2^1(\p_0 V_1,\p_1 V_0,\p_S\p_\Si \nu_S, \p_S K, \p_S^2 g, \p g),
\ee
where, from \eqref{Vg},  
\be\label{Vgf}
V_\a:=g_{\a\b}V^\b=-g^{\rho\s}\p_\rho g_{\s\a}+\tfrac{1}{2}g^{\rho\s}\p_\a g_{\rho\s}.
\ee
The proof of \eqref{com2nu1}, similar to the arguments above, is rather long and so given in detail in the Appendix. 

Last, consider the bulk equation for $g_{11}$:
\bes
-\tfrac{1}{2}\p_0^2 g_{11}+\tfrac{1}{2}\p_1^2 g_{11}+P_{11}(\p g)-Q_{11}=0. 
\ees
This does not yield any compatibility condition, but instead determines $\p_0^2 g_{11}$ from the data in $\cU$. This then determines the full 
2nd order jet of $g$ at $\Si$. 

 Finally, setting $\a=0,2, \dots, n$ in \eqref{Vgf} and taking the $\p_0$ derivative then yields:
\bes\begin{split}
\p_0V_a&=-g^{\rho\s}\p_0\p_\rho g_{\s a}+\tfrac{1}{2}g^{\rho\s}\p_0\p_a g_{\rho\s}+O_1(g)\\
&=\p_0\p_0 g_{0 a}-\p_0\p_1 g_{1 a}-g^{AB}\p_0\p_A g_{B a}-\tfrac{1}{2}\p_0\p_a g_{00}+\tfrac{1}{2}\p_0\p_a g_{11}+\tfrac{1}{2}g^{AB}\p_0\p_a g_{AB}+O_1(g)\\
&=-\p_0\p_1 g_{1 a}+\tfrac{1}{2}\p_0\p_a g_{11}-g^{AB}\p_0\p_A g_{B a}+\tfrac{1}{2}g^{AB}\p_0\p_a g_{AB}+O_1(g), \\
\end{split}\ees
where all the terms on the last line are determined as described above. This gives compatibility conditions on $\p_0V_a$ along $\Si$:
\be\label{com2Va}
\p_0V_a=\chi_2^a(\p_0 V_1,\p_1 V_0,\p_S\p_\Si \nu, \p_S \k, \p_S^2 \g_S, \p \g_{\cC}, )\ \ a=0,2, \dots, n. 
\ee
In summary, the 2nd order compatibility conditions are given by the restrictions \eqref{com2g} on $\p_0 \g_{AB}$, \eqref{com2nua},\eqref{com2nu1} on $\p^2_1 \nu^\a$ 
and \eqref{com2Va} on $\p_0 V_a$.

\medskip 

The analysis above can be iterated as above to higher order: the $k$th order ($k\geq 2$) compatibility condition consists of restrictions on the choice of $\p^k_0 \g_{AB}$, 
the linear combination $\theta_1\p_1^k \nu^\a+\theta_2\p_1^{k-1} V^\a$ with $\theta_1,\theta_2$ not both zero, and $\p^{k-1}_0 V_a$.

Let $\cU_c^k~(k\geq 2)$ denote the space of initial and boundary data which satisfy the compatibility conditions up to order $k$. Then we have the following lemma.
\begin{lemma}\label{Compatibility}
For $s$ sufficiently large (depending only on the dimension $n$ and $k$), the space $\cU_c^{k}$ is a closed Hilbert submanifold of $\cU_c^{k-1}$. 
\end{lemma}
\begin{proof}
As in the analysis for the 2nd order compatibility condition, if $(\g_S, \k, \nu, V_S, \g_{\cC}, V_{\cC}) \in \cU_c^{k-1}$, then there is a unique way to determine 
all the derivatives $\p^{k-1} g_{\a\b}$ along the corner. The $k$th order compatibility condition is derived by the following steps:
\begin{enumerate}
\item Taking $\p_0^{k-2}$ of the bulk equation for $g_{AB}$ leads to the condition
\bes
\p_0^k g_{AB}=\zeta_k(\p^k_S \g_S,\p_S^{k-1}\k,\p^{k-1}_S\p_\Si \nu,\p^{k-1} g). 
\ees
\item Taking $\p_0^{k-2}$ of the bulk equation for $g_{0\a}$ leads to the compatibility condition with dominant term given by a linear combination 
\bes
\theta_1\p_1^k \nu_S^\a+\theta_2\p_1^{k-1} V^\a=\xi^\a_k(\p^{k-1}_0 V_1,\p^k_S \g_S,\p_S^{k-1}k,\p^{k-1}_S\p_\Si \nu,\p^{k-1} g), 
\ees
where $\theta_1,\theta_2$ are not both zero and depend on $k$ and the 0-jet of $g$.
\item Taking $\p_0^{k-1}$ of \eqref{Vgf} with $\a=0,2,\dots, n$ gives an equation for $\p_0^{k-1}V_a$, $a=0,2, \dots, n$:
\bes
\p_0^{k-1}V_a=\chi^{a}_{k-1}(\p^{k-1}_0 V_1,\p^k_S \g_S,\p_S^{k-1}\k,\p^{k-1}_S\p_\Si \nu,\p^{k-1} g). 
\ees
\end{enumerate}
    Define the following map 
\be \label{CompOp}
W^k: \cU_c^{k-1} \to Sym(\Si)^2 \times \bV(\Si) \times \bV(\Si),
\ee
\bes
W^k(Q, \g_S, \k, \nu_S, V_S, \g_{\cC}, V_{\cC}) = 
\left\{ 
\begin{array}{l}
\p_0^k \g_{AB}-\zeta_k\\
\theta_1\p_1^k \nu_S^\a+\theta_2\p_1^{k-1} V^\a-\xi_k^\a\\
\p_0^{k-1}V_a-\chi_{k-1}^a\\
\end{array}\right.
\ees
Since an element in $\cU_c^{k-1} $ can have arbitrary values in the leading terms above, it is easy to see that $W^k$ is a submersion and hence 
$\cU_c^{k}=(W^k)^{-1}(0)$ is a submanifold. 

\end{proof}

  To conclude this subsection, we turn to the constraint equations \eqref{Gauss}-\eqref{GC}; these induce a compatibility between 
the initial data $(\g_S, \k)$ on $\cI \times \cB$ and the bulk data $Q$ along the initial surface $S$. 

  Analogous the Lemmas above, define the constraint operator 
 $$C: \cU_c^k \times \Lambda^1(S) \to \Lambda^1(S),$$
\be \label{Const}
\begin{split}
& C(Q, \g_S, \k, V_S, \p_t V_S, \nu_S, \g_{\cC}, V_{\cC}) = 
\begin{cases}
|\k|^2 - (\tr_{\g_S}\k)^2  - R_{\g_S} + 3Q(\nu_S,\nu_S) - \tr_{\g_S}Q - 3\d^*V(\nu_S, \nu_S) + \tr_{g_S}\d^*V\\
{\rm div_{\g_S}}  \k - d_S(\tr_{\g_S}\k) - Q(\nu_S, \cdot) + \d^*V(\nu_S, \cdot). \\
\end{cases}
\end{split}
\ee

  We define 
 \be \label{TH}
 \cT^H = (\cU_c \times \Lambda^1(S))_0,
 \ee
to be the zero-set of the constraint operator $C$. By inspection in \eqref{Const}, on $\cT^H$, $\p_t V_S$ is uniquely determined by the remaining 
data in $\cU_c^k$, and so $\cT^H$ is naturally embedded as a subset of $\cU_c^k$. 

\begin{proposition}\label{manifold}
For $s$ sufficiently large (depending only on $n$ and $k$), the target space $\cT^H = (\cT^H)^k$ is a closed Hilbert submanifold of $\cU_c^k$. 
\end{proposition}

\begin{proof} 
 As in the previous Lemmas, we claim that $C$ is a submersion. Namely, consider infinitesimal deformations of the 
form $(F, 0, 0, 0, 0, 0, 0, 0) \in T(\cU)$. Then 
\be 
\begin{split}
&DC(F, 0, 0, 0, 0, 0, 0) = 
\begin{cases}
3F(\nu_S,\nu_S) - \tr_{\g_S}F \\
F(\nu_S, \cdot). \\
\end{cases}
\end{split}
\ee
This map is clearly surjective, and the proof then follows as in the previous Lemmas. 

\end{proof}

  \subsection{Localization} 
  
   As usual with hyperbolic PDE problems, the main existence results will first be proved locally, in regions where the constant coefficient 
approximation is effective. Such local solutions will then be assembled together, by an essentially standard process, via rescaling and a partition of unity,

  Given a metric $g$ on $M$ as in \S 1, the localization at a point $p \in \Si$ is a (small) neighborhood $U \subset M$ of $p$, diffeomorphic via 
a local chart to a Minkowski corner 
$$\mathbf R=\{(t=x^0,x^1,\dots ,x^n):t\geq 0, x^1\leq 0\}$$ 
with $S\cap U \subset \{t = 0\}$, $\cC \cap U \subset \{x^1 = 0\}$ and $x^{\a}(p) = 0$. 

  In such a metrically small neighborhood $U$, as usual we renormalize the metric and coordinates simultaneously by rescaling; thus 
for $\l$ small, set 
\be \label{renor}
\w g = \l^{-2}g, \ \ \w x^\a=\l^{-1}x^\a,
\ee
so that $\p_{\w x^{\a}} = \l \p_{x^{\a}}$ and  
\be\label{rescale}
\w g(\p_{\w x^\a},\p_{\w x^\b})|_{\w x} = g(\p_{ x^\a},\p_{ x^\b})|_{\l x}. 
\ee
The equation \eqref{rescale} holds in the same way for any variation $h = \frac{d}{ds}(g + sh)|_{s=0}$ of $g$. Note that while the components 
$g_{\a\b}$ of $g$ are invariant under such a rescaling, all higher derivatives become small: 
\be \label{lambda}
\p_{\w x^{\mu}}^k \w g_{\a\b} = \l^k \p_{x^{\mu}}^k g_{\a\b} = O(\l^k).
\ee
Thus the coefficients are close to constant functions in the rescaled chart, 
\be \label{eps}
||\w g - g_{\a_0}||_{C^{k}(U)} \leq \e^k = \e^k(\l, g),
\ee
where $g_{\a_0}$ is a flat (constant coefficent) Minkowski metric. This is the frozen coefficient approximation. 

  The $\l \to 0$ limit blow-up metric $g_{\a_0}$ is the Minkowski metric, given in adapted coordinates by 
  \bes
g_{\a_0} =-dt^2-\a_0 dtdx^1+\sum_{i=1}^n(dx^i)^2.
\ees
Here $\a_0 = \a(p)$ is the value of the corner angle at $p$. 

  It is important to note here however that these limit Minkowski metrics do {\it not} satisfy Assumption (*). Thus, while $\l$ is small, the 
limit $\l \to 0$ is never taken in this work. We only consider metrics $g$, and the linearizations at such $g$, which satisfy (*). 

  For later use, we record how the data in $\cT^H$ and the Sobolev norms behave under rescaling. Thus, while $\Ric_{\w g} = \Ric_g$ as abstract 
symmetric bilinear forms, in local coordinates one has 
\bes
\begin{split}
&(\Ric_{\w g})_{\w \a \w \b} = \l^2 (\Ric_g)_{\a\b}, \\
&(K_{\w g})_{\w \a \w \b} = \l (K_g)_{\a\b}, \\
&(\w \nu_S)^{\w \a} = \nu_S^{\a}, \\
&(\w V)^{\w \a} = \l V^{\a}. \\
\end{split}
\ees

  Similarly, for $D \subset S$ or $D \subset \cC$ of dimension $n$, 
\be \label{lscale}
\begin{split}
&||f||_{H_{\w g}^s(D)}^2 = \l^{n - 2s} ||f||_{H_g^s(D)}^2, \\ 
&||f||_{\cN_{\w g}^s(M)}^2 = \l^{n- 2s} ||f||_{\cN_g^s(M)}^2 . \\
\end{split}
\ee
 
  \medskip 
 
  It is essentially clear that the local versions of the results stated throughout the paper hold for $g$ if and only if they hold for $\w g$. For example 
in the context of $D\Phi^H$ in \eqref{DF2} we have  
$$D\Phi_g^H(h) = ((\Ric + \d^*V)'_h, (h_S, K'_h, (\nu_s)'_h, V'_h |_S), (g_{\cC})'_h, V'_h |_{\cC}),$$
The components of $D\Phi_g^H(h)$ and $D\Phi_{\w g}^H(\w h)$ are related in local coordinates by
\bes
\begin{split}
&((\Ric_{\w g} + \w \d^*\w V)'(\w h), (\w h_S, \w K'_{\w h}, (\w \nu)'_{\w h}, \w V'_{\w h} |_S), ((\w \g_{\cC})'_{\w h}, \w V'_{\w h} |_{\cC})) |_x = \\
&(\l^2 (\Ric_g + \d^*V)(h), (h_S, \l K'_{h}, \nu'_h, \l V'_h |_S), ((\g_{\cC})'_h, \l V'_h |_{\cC}) )|_{\l x}.\\
\end{split}
\ees

  Now for any background metric $g$ and target data $\tau' \in T(\cT)$, form $\w g$ and $\w \tau'$ by the rescaling 
above, choosing $\l = \l(g)$ small enough, so that $\w g$ is $\e$-close to the constant coefficient metric $g_{\a_0}$. It is then easy to check 
that a solution to $D\Phi_{\w g}^H(\w h)=\w \tau'$ uniquely gives rise to a solution to $D\Phi_g^H(h)=\tau'$, where $h$ and $\w h$ 
are related as in \eqref{rescale}.\footnote{The same argument holds for an arbitrary cosmological constant $\Lambda$.} The same 
remarks apply to rescalings for $D\Phi$ and $D \hat \Phi$.

   Regarding the domain $U \subset M$,  we will always assume that $U$ is embedded in a larger region $\w U$, so 
$$U \subset \w U,$$
with $\w U$ still covered by the adapted coordinates $(t, x^i)$, with $t \geq 0$ and $x^1 \leq 0$ in $\w U$ so that the initial surface $S$, boundary 
$\cC$ and corner $\Si$ in $\w U$ are an extension of the corresponding domains in $U$. We also assume \eqref{eps} still holds in $\w U$. 
All target data in $\hat \cT$ and later $\cT^H$ given in $U$ is extended off $U$ to be of compact support in $\w U$ away 
from $S\cap U$ and $\cC \cap U$. In particular, all target data vanishes in a neighborhood of the full timelike boundary of $\w U$ and in 
a neighborhood of the initial slice $\{t = 0\}$ away from $\cC \cap U$ and $S \cap U$ respectively. The same statements hold for variations of the 
target data, i.e.~in $T(\hat \cT)$ or $T(\cT^H)$. 

   For later reference, we note that for the linear systems of wave equations on $\w U$ appearing in \S 3 and \S 4, the finite propagation speed 
property implies that solutions $h$ of the equations also have compact support in $\w U$ away from $S\cap U$ and $\cC \cap U$, for some 
definite (possibly small) time $t > 0$.

\subsection{Boundary data shift.}

  Based on prior work and analogies with the Riemannian setting, it is not to be expected that well-posedness of the Dirichlet IBVP can be proved directly. 
As noted in \S 1, some extra hypothesis such as Assumption $(*)$ must be invoked. 

   Instead, we proceed indirectly, based on the analogous procedure used in the elliptic or Riemannian case. As noted in the Introduction, 
this procedure originates in the work of Nash on isometric embeddings of Riemannian metrics into $\bR^N$. 

  Given $g$, consider arbitrary variations $h = \frac{d}{dt}(g + th)$ of $g$. Let $h^{\tT}$ denote the restriction of $h$ to the boundary $\cC$, 
so $h^{\tT}$ gives the variation of the induced metric $g_{\cC}$ on the boundary. Consider the equivalence relation 
\be \label{equiv}
h^{\tT} \sim h^{\tT} + fA,
\ee
for any function $f$. The relation \eqref{equiv} is well-defined only if $A$ never vanishes; this is ensured by Assumption $(*)$. Also, for $g$ or $h$ 
in $\cN^s(M)$, the Sobolev trace theorem \eqref{Sob1} gives $g_{\cC}, h^{\tT} \in H^{s-1/2}(\cC)$, while generically $A \in H^{s-3/2}(\cC)$ which is not 
contained in $H^{s-1/2}(\cC)$. Thus (strictly speaking) the relation \eqref{equiv} is not defined on function spaces having only finite differentiability. 
For this reason, we assume henceforth that $g \in C^{\infty}(M)$. In this case, \eqref{equiv} is well-defined. 

  Let $[h]_A$ denote equivalence class of $h^{\tT}$ in \eqref{equiv}. The representative of $[h]_A$ is chosen to be $h_A$ given by 
\be \label{orth}
\<h_A, A\> = 0,
\ee
pointwise. 

 As above, let $\nu = \nu_{\cC}$ be the unit normal vector field to $\cC$ in $M$. Note that 
$$fA = [\d^*(f\nu)]^{\tT} = \tfrac{1}{2}[\cL_{(f\nu)}g]^{\tT}.$$
Thus, let $X$ be a smooth extension of $f\nu$ to a vector field on $M$ and let $h = \d^*X$. Then at $\cC$,  
\be \label{normmove}
h^{\tT} = (\d^*X)^{\tT} = fA, \ \ h(\nu, \cdot) = (\d^* X)(\nu, \cdot) = df.
\ee
Even though well-known, it is worth emphasizing the following two points here. 

$\bullet$ At the boundary $\cC$, the deformation $\d^*X$ is algebraic in $f$ in boundary tangential directions, but is a first order 
derivative operator on $f$ in normal directions. 

$\bullet$ The deformation $\d^*X$ corresponds to an infinitesimal displacement of the boundary $\cC$ in the normal 
direction $f\nu$ of the background $(M, g)$. Thus, such deformations move the boundary $\cC$ within a fixed background metric; 
see Remark \ref{moveC} for a more detailed discussion of the nonlinear setting, as opposed to the linearized setting here. 

   Although the deformation $h = \d^*X$ is perhaps the most natural extension of the boundary deformation $fA$ of $h^{\tT}$, it is important to 
note that this deformation $\d^*X$ will {\it not} be used until the later parts of the main analysis; it only appears in the proof of Theorem \ref{Direxist} 
relating the shifted boundary data (discussed next) with the Dirichlet boundary data. Instead, the extension of the deformation $fA$ of $g_{\cC}$ 
to a $h$ deformation of $g$ on $M$ is done locally in coordinates without involving any components of the form $dx^1\cdot dx^{\a}$, 
cf.~Remark \ref{fAext} below. 

  Having reduced the Dirichlet data on $\cC$ by one degree of freedom, one needs to add back another scalar degree of freedom $\eta$ 
with value $\eta_h$ depending on $h$, i.e.~ consider the shifted pair 
\be \label{bp}
(h_A, \eta_h),
\ee
in place of the pair $(h_A, f)$ which gives back Dirichlet boundary data by \eqref{equiv}. There are many possible choices for the scalar field $\eta$ on $\cC$. 
The interaction or relation of the choice of $\eta$ with the choice of gauge field $V'_h$ plays a crucial role in the analysis of well-posedness. 
Given the choice of harmonic gauge \eqref{harm}, an effective choice for $\eta$ is given by 
\be \label{eta}
\eta_h = h(\p_t, \nu) + h(\nu,\nu) := h_{0\nu} + h_{\nu \nu},
\ee
where $\nu = \nu_{\cC}$. Of course, the boundary data \eqref{eta} is not invariant under the gauge group $\Diff_0(M)$. In terms of coordinates, 
$\nu = \nu_{\cC}$ is extended into $M$ via coordinates, as 
$$\nu = \frac{1}{|\nabla x^1|}\nabla x^1 = \frac{g^{1\a}}{|\nabla x^1|} \p_{\a}.$$

  The proof of well-posedness at the linearized level relies essentially on the shifted boundary data \eqref{bp}-\eqref{eta}. Given this, the 
passage from well-posedness with shifted boundary data to well-posedness with Dirichlet boundary data, 
$$(h_A, \eta_h) \to h^{\tT}$$
is much simpler, cf.~Theorem \ref{Dirglobexist}.  

\medskip 

   The following result, analogous to Remark \ref{corner}, will be needed in \S 3. 
    
 \begin{lemma} \label{fft}
The initial data $(h_S, K'_h)$ and boundary data $h_A$ determine $f|_{\Si}$ and $\p_t f|_{\Si}$.
\end{lemma}
We note that this determination of $f|_{\Si}$ and $\p_t f|_{\Si}$ is independent of $\eta_h$. 
\begin{proof} 
To see that $f |_{\Si}$ is determined, observe that the variation $h_S$ of the initial metric on $S$ determines the variation $h_{\Si} = h^{\tT} |_{\Si}$ 
of the induced metric on $\Si$. We have $h^{\tT} = h_A + fA$ and so $h^{\tT} |_{\Si} = (h_A)|_{\Si} + fA|_{\Si}$. Since $A|_{\Si}$ is nowhere 
vanishing (by Assumption (*)), this determines $f$ uniquely at $\Si$. Thus $h_{S}$ and $h^{\tT}$ are determined at $\Si$. 

  Next, let $h_{\Si}$ denote the components of $h$ tangent to $\Si_t$ along $\cC$. Then 
\be \label{Tf}
\nabla_{\p_t} h_{\Si} = \nabla_{\p_t} (h_A)_{\Si} + \p_t(f) A_{\Si} + f\nabla_{\p_t} A_{\Si}.
\ee
Using \eqref{TC}, $\nabla_{\p_t} h_{\Si} = \sqrt{1+\a^2}\nabla_{\nu_S} h_{\Si} - \a \nabla_{N} h_{\Si}$, where $N$ is the unit outward normal to $\Si \subset 
(S, g_S)$. The angle $\a$ is given (its background data from the metric $g$). Also it is easy to see that the term $\nabla_{\nu_S} h_{\Si}$ is determined by 
$K'_h$ (and by $h_S, h^{\tT}$), while $\nabla_N h_{\Si}$ is determined in the same way by the variation of the second fundamental form $B$ of 
$\Si \subset (S, g_S)$. This shows that the left side of \eqref{Tf} is determined by the initial and boundary data.  

  For the right side, $\nabla_{\p_t}(h_A)_{\Si}$ is determined by the boundary data $h_A$, while $\nabla_{\p_t} A_{\Si}$ is determined by the 
background metric $g$ and $f$ is determined on $\Si$ as above. It follows that $\p_t(f)A_{\Si}$ is determined. Since Assumption (*) implies 
$A_{\Si}$ is not identically zero, it follows that $\p_t f$ is determined along $\Si$.  
 
\end{proof}

\begin{remark} \label{angvar} 
{\rm The variation $\a'_h$ of the corner angle is also determined by the shifted boundary data $(h_A, \eta_h)$, in fact just by $h_A$. This follows 
immediately as in Lemma \ref{alpha}, since $h^T$ and $\p_t h^T$ are determined at $\Si$ by Lemma \ref{fft}.  
}
\end{remark}

\begin{remark} \label{moveC}
{\rm The discussion above has been at the linearized level, but it worthwhile to briefly consider the full nonlinear situation. (This will not be used in the 
following however). Assume for simplicity $M$ is simply connected. Define two metrics $(M_1, g_1)$, $(M_2, g_2)$ to be equivalent if there is a domain 
$(\O, g)$ and isometric embeddings $\iota_1: (\O, g) \to (M_1, g_1)$ and $\iota_2: (\O, g) \to (M_2, g_2)$. Thus both $(M_1, g_1)$ and $(M_2, g_2)$ are 
extensions of a common metric $(\O, g)$. For example if $M_1 \subset M_2$, then $g_2$ is an extension of $g_1$. 

  The inclusion relation gives a partial order and in any equivalence class (by taking the union) there is a unique maximal representative. 
  Next, define an equivalence relation on the boundary metrics $\g = g_{\dm}$: $\g_1 \sim \g_2$ if there is a curve of metrics $g_s$ in a fixed 
equivalence class $[g]$ as above such that $\g_s$ satisfies 
$$\frac{d}{ds}\g_s = f_s A_s,$$
for some curve of functions $f_s$. Thus the boundary is being deformed in normal directions within the maximal representative metric. This leads to the 
equivalence relation $[h]_A$ in \eqref{equiv}. However, it is not clear if the shifted boundary data \eqref{bp} arise naturally as the linearization or derivative of a 
smooth nonlinear map $\w \Phi$ or $\w \Phi^H$. This won't be needed however. 
 }
 \end{remark}

    Since we have introduced the shifted boundary conditions \eqref{bp} in place of Dirichlet boundary data $h^{\tT}$, we will work with 
the modification of the derivative $D\Phi^H$ in \eqref{PhiH} given by the linear map 
 \be \label{wPhiH}
D\w \Phi_g^H: T_gMet_s(M) \to T( \w \cT^H),
\ee
$$D\w \Phi_g^H(h) = (\Ric'_h + (\d_g^* V)'_h, (h_S, K'_h, (\nu'_S)'_h, (V_S)'_h), (h_A, \eta_h), (V_{\cC})'_h).$$
Note that $D\w \Phi^H$ and $D\Phi^H$ differ only in the boundary data space.\footnote{As noted above, $D\w \Phi^H$ and $T(\w \cT^H)$ are merely (suggestive) 
notation for the linear map and target space; they do not denote the derivative of a nonlinear map.} It follows from Lemma \ref{fft} and Remark \ref{angvar} that 
the full Dirichlet boundary data $h|_{\Si}$ is determined at the corner by $(h_A, \eta_h)$; hence the compatibility conditions in \S 2.3 for the modified boundary 
data are the same as for Dirichlet boundary data $h^{\tT}$. 

   All of the linear analysis to follow, namely existence results and energy estimates, will first be carried out for the linear map $D\w \Phi^H$. Given that, 
 it is relatively easy to obtain the same results for the original maps $D\Phi^H$ as well as for $D\hat \Phi$ and $D\Phi$.

 \section{Linearized boundary equations.}
  
    In this section, we derive equations for the parts of the boundary data $h_{\a\b}|_{\cC}$ not already 
determined by the shifted target data $(h_A, \eta_h)$. These undetermined components are $f$, i.e.~$fA$, and $h(\nu, \cdot)$ of $h$ at $\cC$. 
These equations play a crucial role in the subsequent analysis of energy estimates and existence of linearized solutions. 

  Throughout this section we work along the boundary $\cC$, so that $\nu = \nu_{\cC}$. We also work primarily in the localized setting of \S 2.4. 

  To start, the Dirichlet boundary data $h_A$ are specified, (corresponding to specifying $n(n+1)/2 - 1$ degrees of freedom) at $\cC$ as well as 
the scalar $\eta$ with 
\be \label{eta'}
h_{0\nu} + h_{\nu \nu} = \eta_h.
\ee
There are $n+1$ remaining degrees of freedom which need to be specified as boundary conditions, since the full metric or variation $h_{\a\b}$ has 
$(n+1)(n+2)/2$ degrees of freedom. These remaining boundary conditions come from the choice of gauge, here the harmonic gauge $V'_h$. Using 
the expression \eqref{V'}  decompose the ambient Bianchi operator $\b_g$ into the part $\b_{\cC}$ operating only on data tangent to $\cC$ together with the remaining 
normal components. Basic computation then shows that  
\be \label{gauget}
(V'_h)(\p_t) = - \<\nabla_{\nu}h(\nu)^{\tT}, \p_t\> + \b_{\cC}h{^\tT}(\p_t) + \tfrac{1}{2}\p_t h_{\nu \nu} - \<(A + H\g)h(\nu)^{\tT}, \p_t\> - \<D^2t, h\>,
\ee
\be \label{gaugeSi}
(V'_h)(\p_a) = - \<\nabla_{\nu}h(\nu)^{\Si}, \p_a\> + (\b_{\cC}h{^\tT})(\p_a) + \tfrac{1}{2}\p_a h_{\nu \nu} - \<(A + H\g)h(\nu)^{\tT}, \p_a\> 
- \<D^2 x^a, h\>,
\ee 
\be \label{gaugen}
(V'_h)(\nu) = -\tfrac{1}{2}\nu(h_{\nu \nu}) + \tfrac{1}{2}\nu(\tr_{\cC}h^{\tT}) + \d_{\cC}(h(\nu)^{\tT}) - h_{\nu \nu}H + \<A, h^{\tT}\> - \<D^2 x^1, h\>
\ee 
Here $\p_a$ is tangent to $\Si_t$, so $a = 2, \dots n$. The left sides of \eqref{gauget}-\eqref{gaugen} are given boundary data in $T(\w \cT^H)$. 
Note that the main terms on the right are Neumann derivatives of $h(\nu)^{\tT}$, and recall that Neumann boundary data do not admit good 
energy estimates, cf.~Remark \ref{Neumann}. 

   In these $n+1$ gauge equations, or the given target data $(h_A, \eta_h)$, there is no simple equation for $f$. To derive an explicit equation for 
$f$, we use the linearization of the Hamiltonian constraint \eqref{Gauss}. Let 
 $$\w \Box_{\cC} \f = \<D^2f, \Pi\> = \<D^2 \f, Hg_{\cC} - A\>.$$
 Under the convexity Assumption $(*)$, this is a hyperbolic wave-type operator on $\f$ on $\cC$. This hyperbolicity is the main reason for 
 the use of Assumption (*). 
 
\begin{remark} \label{renorm} 
{\rm The results of this section apply in the unrenormalized or global setting $(M, g)$ or in the localized setting discussed in \S 2.4. However, a 
renormalization is needed in the localized context. 

   Namely, recall the splitting \eqref{equiv}, $h^{\tT} = h_A + f A$ used in the definition of $D\w \Phi^H$. The metric $g$ (or boundary metric 
$g_{\cC}$) and the second fundamental form $A$ scale differently under rescaling. Thus, in the localization 
setting of \S 2.4, we have $\w h^{\tT} = \w h_A + f\w A$ and $|\w A|_{\w g} \sim \l$ is small, $\l \sim \e$, while $\w h$ is $O(1)$, so that $O(f\w A) = 
\l O(f)$. For this reason, we 
renormalize $A$ by setting 
\be \label{Arenorm}
\hat A = \w A/\l 
\ee
in place of $\w A$, so that $f = O(1)$ and $\hat A = O(1)$. With this understood, we then drop the tilde from the notation, so in the following 
$\w g$, $\w h$ etc are denoted by $g$, $h$, etc. Of course $\hat A = A$ when no rescaling is performed. 

In the localized setting, the operator $\w \Box_{\cC}$ then has the form 
$$\w \Box_{\cC}\f = -\<D^2f, \hat \Pi\> = \<D^2 \f, \hat H g_{\cC} - \hat A\>.$$
}
\end{remark}

 First we derive the equation for $f$ along $\cC$. 
  
\begin{lemma}
The function $f$ satisfies the equation 
\be \label{Hlin}
\w \Box_\cC f  + q(\p f) = Y(\cT, V'_h) + E(h),
\ee
along the boundary $\cC$, where 
\be \label{finhom}
\begin{split}
Y(\cT, V'_h) &= R'_{h_A} - \tr(F-\d^* V'_h) + 2(F-\d^* V'_h)(\nu,\nu),\\
E(h) &= 2\<A'_h, A\> - 2H_\cC H'_h - 2\<A\circ A, h\> + 4\Ric(\nu,\nu'_h) + \<\Ric, h\>\\
&- \tr(\d^*)'_h V - 2(\d^*)'_h V(\nu,\nu),
\end{split}
\ee
and 
\be \label{qf} 
q(\p f) = f\Box_{\cC}\hat H + 2\<df, d\hat H\> - f\d_{\cC}\d_{\cC}\hat A + 2\d_{\cC}\hat A(df) + f\<\Ric_{\cC}, \hat A\>.
\ee
Under Assumption $(*)$, \eqref{Hlin} is a hyperbolic evolution equation. Moreover, the initial data $f$ and $\p_t f$ for $f$ at the initial surface 
$\Si$ are determined by the initial and boundary data in $T(\cT)$. 
\end{lemma} 

\begin{proof} 
The Hamiltonian constraint analogous to \eqref{Gauss} along the timelike hypersurface $\cC$ is given by 
$$|A|^2 - H^2 + R_{g_{\cC}} = R_g - 2\Ric(\nu, \nu), \ \footnote{Note the difference in the sign change in the $R_{g_{\cC}}$ term compared with the 
Hamiltonian constraint \eqref{Ham} along $S$.}$$
where again $\nu = \nu_{\cC}$. This holds either for $g$ or the rescaled $\w g$. Upon linearization, this gives 
\be \label{Hamlin2}
2\<A'_h, A\> - 2\<A^2, h\>-2H H'_h+R'_{h^{\tT}}=R'_h-2\Ric'_h(\nu,\nu)-4\Ric(\nu,\nu'_h)
\ee 
Note that 
\bes
R'_{h^{\tT}} = R'_{f\hat A}+R'_{h_A}=-\Box_\cC (f\hat H) + \d_\cC\d_\cC(f\hat A)-\<\Ric_\cC, f\hat A\> + R'_{h_A}=
-\hat H\Box_\cC f + \<D^2 f, \hat A\> + R'_{h_A} - q(\p f),
\ees
where $q(\p(f)$ is given as in \eqref{qf}. So we obtain:
\bes
\w \Box_\cC f = R'_{h_A} +2\<A'_h, A\> - 2\<A^2, h\> - 2H_\cC H'_h - R'_h + 2\Ric'_h(\nu,\nu) + 4\Ric(\nu,\nu'_h) - q(\p f). 
\ees
Also $R'_h=\tr \Ric'_h-\<h,\Ric\>$ and $\Ric'_h=F-\d^* V'_h-(\d^*)'_h V$. Thus
\bes\begin{split}
\w \Box_\cC f + q(\p f) &= R'_{h_A} - \tr(F - \d^* V'_h) + 2(F - \d^* V'_h)(\nu,\nu)\\
& + 2\<A'_h, A\> - 2H_\cC H'_h - \<A^2, h\> + \tr (\d^*)'_h V - 2(\d^*)'_h V(\nu,\nu) \\
& + 4\Ric(\nu,\nu'_h) + \<h, \Ric\>,\\
\end{split}
\ees
which gives the result. 

By Lemma \ref{fft}, the initial data $f|_{\Si}$ and $\p_t f |_{\Si}$ for $f$ at $\Si$ are determined by the target data in $T(\cT)$. 

\end{proof}

Note that the coefficients of the variation $h$ on the right hand side of $E(h)$ in \eqref{finhom} as well as the coefficients of the first order term $q(\p f)$ are all 
$O(\e)$ in the rescaled metric $\w g$, by \eqref{eps}. 

\begin{remark} \label{Id}
 {\rm  Consider the infinitesimal deformation $\hat h = \d^*(f\nu)$. When $\Ric_g = 0$, this satisfies the linearized Hamiltonian constraint \eqref{Hamlin2} 
 for any $f$, (and in any gauge). So the equation \eqref{Hamlin2} imposes no restriction on $f$. To explain this, the right hand side of \eqref{Hlin} involves 
 $A'_{\hat h}$ and $H'_{\hat h}$. These terms are first order in $\hat h$, and so second order in $f$. Thus \eqref{Hlin} is an identity between both sides 
 and involves no restrictions on $f$. 

   The term $A'_h$ in particular involves $\d^* h(\nu)^{\tT}$, (cf.~\eqref{A'} below). For $\hat h$, one has $\hat h(\nu)^{\tT} = d_{\cC}f$ and 
$\d^* \hat h(\nu)^{\tT} = D^2f$, indicating that \eqref{Hlin} is then just an identity. However, in the analysis to follow, $\w \Box_{\cC} f$ will 
be treated as second order in $h$, while $A'_h$ is treated as first order in $h$. Thus, $f$ is only determined or controlled on $\cC$ by \eqref{Hlin} 
provided the other components of $h$, in particular, $h(\nu, \cdot)$ are already determined or controlled. 
   
   In fact, as noted preceding \eqref{bp}, the boundary data $fA$ will {\it not} be extended into the bulk $M$ as $\d^*(f\nu)$, for an extension 
of $\nu$ into $M$. Instead, it will be extended component or coordinate-wise, cf.~Remark \ref{fAext}  below. 
}
\end{remark}

  Returning to \eqref{gauget}-\eqref{gaugen}, write the term $\b_{\cC}h^{\tT}$ as 
$$\b_{\cC}h^{\tT} = \b_{\cC}h_A + \b_{\cC}(f\hat A).$$ 
The term $\b_{\cC}h_A$ is part of the target data in $T(\w \cT)$, while for the second term we have 
$$\b_{\cC}(f\hat A) = \d_{\cC}(f\hat A) + \tfrac{1}{2}d\tr f\hat A = f \d \hat A - \hat A(\nabla f) + \tfrac{1}{2}(\hat Hdf + fd\hat H).$$ 
This gives   
\be \label{bA}
\b_{\cC}h^{\tT} = \b_{\cC} h_A -(\hat A - \tfrac{\hat H}{2}\g)(\nabla f) + f(\d \hat A + \tfrac{1}{2}d\hat H).
\ee
In particular, \eqref{gauget} may be rewritten in the form 
\be \label{gt}
\begin{split}
\<\nabla_{\nu}h(\nu)^T, \p_t\> = & \tfrac{1}{2}\p_t (h_{\nu \nu}) - (\hat A - \tfrac{\hat H}{2}\g)(\nabla f, \p_t) + f(\d \hat A + \tfrac{1}{2}d\hat H)(\p_t) \\
&-\<(A + H\g)h(\nu)^{\tT}, \p_t\> + \b_{\cC}(h_A)(\p_t) - V'_h(\p_t) - \<D^2t, h\>,\\
\end{split}
\ee

\begin{lemma}
The component $h_{0\nu}$ satisfies the Sommerfeld boundary condition 
\be \label{h01}
(\tfrac{1}{2}\p_t + \p_{\nu})h_{0\nu} = \hat Y(\cT) + \hat Z(f) + \hat E(h),
\ee
at $\cC$, where 
\be \label{h01terms}
\begin{split}
&\hat Y(\cT) =    \b_{\cC}(h_A)(\p_t) - V'_h(\p_t) + \tfrac{1}{2}\p_t \eta_h, \\
&\hat Z(f) = - (\hat A - \tfrac{\hat H}{2}\g)(\nabla f, \p_t) + f(\d \hat A + \tfrac{1}{2}d\hat H)(\p_t), \\ 
&\hat E(h) =  -\<(A + H\g)h(\nu)^{\tT}, \p_t\>  + \<h(\nu)^{\tT}, \nabla_{\nu}\p_t\> - \<D^2t, h\>.
\end{split}
\ee 
The component $h_{\nu \nu}$ satisfies 
$$h_{\nu \nu} = -h_{0\nu} + \eta.$$ 
\end{lemma}

\begin{proof}
  This follows from \eqref{gauget} and \eqref{gt}, using \eqref{eta'} to replace $h_{\nu \nu}$ by $h_{0\nu}$ and then rearranging the terms. 
 
 \end{proof}

\begin{lemma}
The component $h(\nu)^{\Si} = h_{\nu a} \p_a$, $a = 2, \dots n$, satisfies the Neumann boundary condition 
\be \label{hnua}
\nu (h_{\nu a}) = \w Y(\cT) + \w Z(f, h_{\nu \nu}) + \w E(h)
\ee
where 
\be \label{hnuaterms}
\begin{split}
&\w Y(\cT)_a = \b_{\cC}h_A(\p_a) - (V'_h)(\p_a),\\
&\w Z(f, h_{\nu \nu})_a = -\<(\hat A -\tfrac{\hat H}{2}\g(\nabla f), \p_a\> + \tfrac{1}{2}\p_a h_{\nu \nu} + f(\d \hat A + \tfrac{1}{2}d\hat H)(\p_a), \\ 
&\w E_a(h) =  -\<(A + H\g)h(\nu)^{\tT}, \p_a\>  + \<h(\nu)^{\tT}, \nabla_{\nu}\p_a\> - \<D^2 x^a, h\>,
\end{split}
\ee 
at $\cC$. 
\end{lemma}

\begin{proof}
This follows from \eqref{gaugeSi} and \eqref{bA} in the same way as above. 

\end{proof}

\begin{remark}
{\rm We note that the normal component \eqref{gaugen} of the harmonic gauge data (which is part of the target data) is not directly used in this 
process, cf.~\eqref{bnu2} below however. 
 
 On the other hand, \eqref{gaugen} is in fact dependent on the other $n+1$ equations above. Namely, the linearized Hamiltonian 
 constraint can be derived from the harmonic gauge equations (by taking $\d_{\cC}$ of \eqref{gauget} and $\nu$ of \eqref{gaugen}). So \eqref{gaugen} is 
embedded in the equations \eqref{gauget},\eqref{gaugeSi} and \eqref{Hlin}. 
 }
  \end{remark} 
  
    For later reference, note that the $Y$ terms above depend only on the target data $\cT$, with the exception of $\d^*V'_h(\nu, \nu)$ in \eqref{Hlin}. 
The $Z$ terms have a triangular structure in the order of the components $h_A, fA, h_{0\nu}, h_{\nu, \nu}, h(\nu)^{\Si}$ used above. The $E$ terms all have 
coefficients of order $O(\l)$ or $O(\e)$ in the localization of \S 2.4, and all are of order zero in $h$ except for the gauge-dependent terms and the terms 
involving $A'_h$ in \eqref{finhom}. 
 
\section{Energy estimates.}

   In this section, we derive (apriori) energy estimates for smooth solutions of the linear system $L(h) = F$ with shifted boundary data 
$(h_A, \eta_h)$  in the local setting of \S 2.4. This is the key to proving linearized well-posedness for both shifted boundary data and 
Dirichlet boundary data in the later sections. 

\smallskip 

  Define a Sobolev $H^s$ norm on the target space data in $T_{\w \tau}(\w \cT^H)$ as follows: for $\w \tau' \in T_{\w \tau}(\w \cT^H)$,   
\be \label{tnorm}
\begin{split}
||\w \tau'||_{H^s(\w \cT^H)} = ||F||_{H^s(\w U)} & + [ ||\g'_A||_{H^{s+1}(\cC)} + ||\eta||_{H^{s}(\cC)} + ||V'_{\cC}||_{H^s(\cC)}] \\
& + [ ||\g'_S||_{H^{s+1}(S_0)} + ||\k'||_{H^{s}(S_0)} + ||\nu'||_{H^s(S_0)} + ||V'_S||_{H^s(S_0)} ].
\end{split}
\ee
 On the right in \eqref{tnorm}, the top line consists of bulk and boundary terms, while the bottom line consists of initial data terms. Note the shift up in the 
derivative index $s$ on the target compared with the discussion in \eqref{Phi} or \eqref{PhiH}. 
 
  The main result of this section is the following.

\begin{theorem} \label{Eestthm} 
  Assume Assumption $(*)$ holds. Then for the (formal) linearization $D\w \Phi_g^H$ in \eqref{wPhiH} with $C^{\infty}$ data localized on $\w U \supset U$ 
as in \S 2.4, one has the tame energy estimate 
\be \label{Eest}
||h||_{\cN^s(\w U)}  \leq C \l^{-1/2}||D\w \Phi_g(h)||_{H^s(\w \cT^H)},
\ee
where $C$ is a constant depending only on $g$. In particular, 
$$Ker D\w \Phi_g^H = 0.$$
\end{theorem} 

We point out that the proof given here shows that the factor of $\l^{-1/2}$ cannot be dropped, indicating the breakdown of the estimate at the 
Minkowski corner. The scale factor $\l$ depends on the geometry of $g$; it must be sufficiently small, as in \eqref{eps}, to allow for absorption of 
$\l$-dependent terms in the estimates below In addition to a general (smooth) bound on the Lorentz metric $g$, the constant $C$ depends on a 
bound for the non-degeneracy of the Lorentz metric $\Pi = Hg_{\cC} - A$ from Assumption $(*)$. This enters through the equation \eqref{Hlin}. 
Throughout this section we use the notation $g$, $h$ in place of $\w g$, $\w h$ from \S 2.4. 

  We also emphasize that \eqref{Eest} is not a strong or boundary stable energy estimate, as in \eqref{DirE} or \eqref{SomE}. 
 
 \begin{proof}

  Throughout the following $h$ is a symmetric bilinear form, written in adapted local coordinates $x^{\a} = (t, x^i)$ as 
$$h = h_{\a\b}dx^{\a}dx^{\b}.$$
The full variation $h$ is decomposed into a collection of terms as follows. On the boundary $\cC$, $h^{\tT}$ has the form 
$$h^{\tT} = h_{ij}dx^i dx^j,$$
where $i, j = 0, 2, \dots, n$. Recall that $h^{\tT} = h_A + f\hat A$ on $\cC$. Then $h^{\tT}$ is extended into the bulk $\w U$ to have the same form, 
i.e.~without any $dx^1\cdot dx^{\a}$ term. Similarly, $h(\nu, \cdot)$ is defined on $\w U$ by 
\be \label{hnudef}
h(\nu, \cdot) = \frac{1}{|\nabla x^1|}g^{1 \a}h(\p_{\a}, \cdot),
\ee
which is defined by the background metric $g$ on $\w U$. 

\begin{remark} \label{fAext} 
{\rm 
Note that consequently $fA$ or $f\hat A$ is {\it not} extended into the bulk as $\d^*(f\nu)$, for some extension of $f\nu$ into the bulk. 
}
\end{remark}

    In the following, we will work with separate equations for each of the terms 
\be \label{comp}
h^{\tT}, \ h_{0\nu}, \ h_{\nu \nu}, \ h(\nu)^{\Si},
\ee
comprising $h$. The equations for $f$, $h(\nu, \cdot)$ along the boundary $\cC$ are those described in \S 3. 
The bulk equation \eqref{L} then becomes the following collection of equations: 
\be \label{e1}
-\tfrac{1}{2}\Box_g h^{\tT}  + P^{\tT}(\p h) = F^{\tT},
\ee
\be \label{e3}
-\tfrac{1}{2}\Box_g h_{0\nu}  + P_{0\nu}(\p h) = F_{0\nu},
\ee
\be \label{e4}
-\tfrac{1}{2}\Box_g h_{\nu \nu}  + P_{\nu \nu}(\p h) = F_{\nu \nu},
\ee
\be \label{e5}
-\tfrac{1}{2}\Box_g h_{a \nu} + P_{a\nu}(\p h) = F_{a\nu}.
\ee

  By Lemma \ref{DirSomm}, solutions of these equations satisfy the boundary stable energy estimate \eqref{DirE} for Dirichlet boundary data, or the 
boundary stable energy estimate \eqref{SomE} for Sommerfeld boundary data. For notational simplicity, we assume in the following that the target initial 
data $(h_S, K'_h, \nu'_S, V'_h|_S)$ as well as the bulk target data $F$, are all zero. By linearity, without any significant changes to the 
argument, the relevant terms will then be added back at the end. 

  With $F$ thus set to zero in \eqref{e1}-\eqref{e5}, the coupling term $P(\p h)$ is moved to the right side of each equation and 
treated as an inhomogeneous bulk term. Recall also from \S 2.4 that $g$ and all target data are $C^{\infty}$ and the target data 
have compact support in $\w U$, away from $\cC \cap U$. 
  
The first equation \eqref{e1} has given Dirichlet boundary data $h_A$. To determine the Dirichlet boundary data $f$, recall that $f$ solves the 
wave equation \eqref{Hlin} with initial data determined by the target data $\w \tau'$ and so is uniquely determined by $\w \tau'$. The standard energy 
estimate then gives 
\be \label{eef}
||f||_{\w H^s(\Si_t)} \leq C[||h_A||_{H^{s+1}({\cC_t)}} + ||V'_h||_{\bar H^s(\cC_t)} + \l ||h||_{\bar H^s(\cC_t)}]. \footnote{Bounds on the initial data $||f||_{H^s(\Si_0)}$, 
$||\p_t f||_{H^{s-1}(\Si_0)}$ follow from bounds on the target data $\w \tau'$.}
\ee 
Here $\w H^s(\Si_t)$ denotes the Sobolev norm with both $t$-derivatives as well as derivatives tangent to $\Si_t$, but no derivatives normal to $\cC$. 
The extra derivative required on $h_A$ comes from the $(R_{\cC})'_{h_A}$ term in \eqref{finhom}. Except for terms involving $\d^* V'_h$, the right-side 
of \eqref{Hlin}, i.e.~the inhomogeneous term, is of the form $\l(\p h)$, giving rise to the $\l$ term in \eqref{eef}. Note all derivatives, including normal 
derivatives along $\cC$ are included in the last term in \eqref{eef}. 

Throughout the following, $t \in I = [0,1]$.  The constant $C$ then depends only on the background metric $g$; recall we are working in the localized 
setting and the rescaled metric $\w g$ is denoted as $g$ from now on. Note that the $H^s$ norms on the right are monotone in $t$ and so have maximal 
value at $t = 1$; this is not necessarily the case for the $\bar \cH^s$ norm on the left. We set $\cC = \cC_1$ to simplify the notation. 
Throughout the following, $C$ is a constant which may change from line to line. 

   In addition one has the bound 
\be \label{VbarV}
||V'_h||_{\bar H^s(\cC_t)} \leq C ||V'_{\cC}||_{H^s(\cC_t)} + \l^2||h||_{H^s(\w U)}.
\ee
This follows from the boundary stable energy estimate with Dirichlet boundary data \eqref{DirE} for the gauge wave equation \eqref{V'eqn}. 
We note here that this estimate requires an extra derivative on the bulk $F$ term, due to the $\b F$ term in \eqref{V'eqn}.\footnote{As in the previous 
footnote, \eqref{VbarV} requires a bound on $||\p_t V'_h||_{H^{s-1}(S_0)}$; by the constraint equations \eqref{Const}-\eqref{TH}, this is ensured by 
an extra derivative bound on the initial data $(h_S, K'_h)$ along $S$. Together with previous remarks, this  accounts for the extra derivative on the 
initial data and bulk data assumed in \eqref{tnorm}.}

Given this control on the Dirichlet boundary value of $f$, \eqref{e1} then gives the estimate 
\be \label{ee2}
\begin{split}
||h^{\tT}||_{\bar \cH^s(S_t)} & \leq  C[||h^{\tT}||_{H^s(\cC_t)} + \l ||h||_{H^s(\w U_t)} \\
& \leq C[\int_0^t [ ||h_A||_{H^{s+1}({\cC_{t'})}} +||V'_{\cC}||_{H^s(\cC_{t'})}.+ \l ||h||_{\bar H^s(\cC_{t'})} ] +  \l ||h||_{H^s(\w U_t)}], \\
\end{split}
\ee
where the last inequality comes from integration of \eqref{eef}-\eqref{VbarV} over $t' \in [0,t]$. 
Taking the maximum over $t \in I$ then gives 
\be \label{ee2.1}
\max_{t\in [0,1]} ||h^{\tT}||_{\bar \cH^s(S_t)} \leq C[ ||h_A||_{H^{s+1}(\cC)} + ||V'_{\cC}||_{H^s(\cC_t)} +  \l ||h||_{\hat H^s(\w U)}],
\ee
where $||h||_{\hat H^s(\w U)} :=  ||h||_{H^s(\w U)} + |h||_{\bar H^s(\cC)}$. 

Next, the Sommerfeld energy estimate for $h_{0\nu}$ gives 
$$||h_{0\nu}||_{\bar \cH^s(S_t)} \leq C[||b||_{H^{s-1}({\cC_t)}} + \l ||h||_{H^s(\w U_t)} ],$$
where $b$ consists of the terms $\hat Y, \hat Z, \hat E$ on the right side of \eqref{h01} in \eqref{h01terms}. Upon inspection, this gives 
$$||h_{0\nu}||_{\bar \cH^s(S_t)} \leq C[ ||h_A||_{H^s(\cC_t)} +  ||\eta||_{H^{s}({\cC_t)}} + ||f||_{H^{s}({\cC_t)}} + ||V'_{\cC}||_{H^{s-1}(\cC_t)} + 
 \l ||h(\nu)^{\tT}||_{H^{s-1}(\cC)} + \l ||h||_{H^s(\w U_t)}].$$
Via \eqref{ee2.1}, and using the Sobolev trace theorem \eqref{Sob1} to combine the two $\e$-terms, this in turn gives 
\be \label{ee3}
\max_{t\in [0,1]} ||h_{0\nu}||_{\bar \cH^s(S_t)} \leq C[ ||h_A||_{H^{s+1}(\cC)} +  ||\eta||_{H^{s}({\cC)}} + ||V'_{\cC}||_{H^{s}(\cC)} +  \l ||h||_{\hat H^s(\w U)}],
\ee
Similarly, using the definition of $\eta$ and the Dirichlet boundary data for $h_{\nu \nu}$ gives 
$$\max_{t\in [0,1]}||h_{\nu \nu}||_{\bar \cH^s(S_t)} \leq C||h_{\nu \nu}||_{H^s({\cC)}} + ||V'_{\cC}||_{H^{s-1}(\cC)} + \l ||h||_{H^s(\w U)}, $$
and so by \eqref{ee3},
\be \label{ee4}
\max_{t\in [0,1]} ||h_{\nu \nu}||_{\bar \cH^s(S_t)} \leq C[ ||h_A||_{H^{s+1}(\cC)} + ||\eta||_{H^{s}(\cC)} + ||V'_{\cC}||_{H^{s}(\cC)} + \l  ||h||_{\hat H^s(\w U)}],
\ee
Finally for $h_{\nu a}$, from the Neumann boundary estimate \eqref{NeuE}, we have 
$$||h_{\nu a}||_{\bar H^s(S_t)} \leq C[||\p_{\nu}h_{\nu a}||_{H^{s-1/2}({\cC_t)}} + \l ||h||_{H^{s}(\w U_t)}]. $$ 
Examination of the form of the Neumann boundary data in \eqref{hnua} and \eqref{hnuaterms} gives 
\be \label{ee5}
\begin{split}
||h_{\nu a}||_{\bar \cH^s(S_t)} & \leq C[ ||f||_{H^{s+1/2}(\cC_t)} + ||h_{\nu \nu}||_{H^{s+1/2}(\cC_t)} + ||h_A||_{H^{s+1/2}(\cC_t)} \\
& + ||V'_{\cC}||_{H^{s-1/2}(\cC)} + \l(||h||_{H^{s-1/2}(\cC_t)} + ||h||_{H^{s}(\w U_t)})].
\end{split}
\ee 

 The only problematic estimate is the $h_{\nu a}$ estimate \eqref{ee5}, which loses a half-derivative through the Neumann boundary data. The coupled 
system \eqref{eef}-\eqref{ee5} does not close in $H^s$. Thus if the first two $H^{s+1/2}$ norms on the right of \eqref{ee5} were only $H^s$ norms, these 
two terms could be controlled as in \eqref{ee2.1} - \eqref{ee4}. By choosing $\l$ sufficiently small, the $\l ||h||_{H^s(\w U)}$ terms on the right could then 
be absorbed into the left side of the sum of \eqref{eef}-\eqref{ee5}, giving the required energy estimate in terms of target data. However, the 
loss-of-derivative  in \eqref{ee5} makes this impossible. An alternate path is thus necessary.

  \medskip 
 
   In place of using the Neumann boundary data estimate which loses a half-derivative, we proceed instead as follows. 
Consider the vector wave equation acting on the vector field $h(\nu)^{\Si}$ tangent to the corners $\Si_t$. Thus the equation \eqref{L} that $L(h) = F$ 
applied to the vector field $h(\nu)^{\Si}$ can be written as  
\be \label{Lnu}
\tfrac{1}{2}D^*D h(\nu)^{\Si} + P_{\Si}(\p h) = F(\nu)^{\Si}.
\ee
We now work directly with the vector equation \eqref{Lnu} in place of \eqref{e5}. As remarked above, we continue to assume $F = 0$ and 
all initial data vanish. Pairing \eqref{Lnu} with $\nabla_{\p_t} h(\nu)^{\Si}$ and integrating by parts, we obtain  
\be \label{ea}
\tfrac{1}{4}\frac{d}{dt}\int_{S_t}|\nabla_{\p_t} h(\nu)^{\Si}|^2 + |\nabla_S h(\nu)^{\Si}|^2 \leq \int_{\Si_t}\<\nabla_{\p_t} h(\nu)^{\Si}, \nabla_{\nu}h(\nu)^{\Si}\> + 
\l \int_{S_t}|\nabla_{\p_t}h(\nu)^{\Si}||\p h|,
\ee
where $\nabla_S$ denotes the spatial derivatives tangent to $S_t$. Here and throughout the following, we will ignore the $t$-derivatives of the volume 
form (omitted from the notation) and the metric; these are lower order terms which are easily handled by the methods below. 

  Now all the norms and inner products in \eqref{ea} (except the last) are with respect to the Lorentz metric $g$, and so are not positive definite. In other words, 
the estimate \eqref{ea} is not the same as a standard energy estimate that one would have for scalar functions. On the left side of \eqref{ea}, the 
negative part of the norm is the $T$ component, i.e.~the term $-(\<\nabla_{\p_t} h(\nu)^{\Si}, T\>)^2$ where $T$ is the unit normal to $S_t$, 
and the same for the $\nabla_S$ term. However, 
$$\<\nabla_{\p_t} h(\nu)^{\Si}, T\>  = -\<h(\nu)^{\Si}, \nabla_{\p_t} T\>.$$
These are lower order $\l$-terms, again easily handled by the arguments to follow, and so will be ignored. Thus the norms on the left in \eqref{ea} are 
Riemannian, modulo lower order terms. The inner product on the right of \eqref{ea} is however still Lorentzian; this will be of importance later. 

 Integrating from $0$ to $t$ and using $ab \leq a^2 + b^2$, we obtain from \eqref{ea} 
\be \label{en1}
\begin{split}
\tfrac{1}{4}\int_{S_t}||\nabla_{\p_t} h(\nu)^{\Si}||^2  +  ||\nabla_S h(\nu)^{\Si}||^2 & \leq \int_{\cC_t}\<\nabla_{\p_t} h(\nu)^{\Si}, \nabla_{\nu}h(\nu)^{\Si}\> 
+ \int_{S_0} ||\nabla_{\p_t} h(\nu)^{\Si}||^2 + ||\nabla_S h(\nu)^{\Si}||^2 \\
& + \int_{0}^t \int_{S_t}||\nabla_{\p_t} h(\nu)^{\Si}||^2 + ||\nabla_S h(\nu)^{\Si}||^2 + \l^2 \int_{\w U_t} |\p h|^2 ,
\end{split}
\ee
where the norm $|| \cdot ||$ is the Riemannian norm $g_{S_t}$. Let 
\be \label{edef}
E_t^{\nu} = ||h(\nu)^{\Si}||_{\bar H^1(S_t)}^2 = \int_{S_t}||\nabla_{\p_t} h(\nu)^{\Si}||^2 + ||\nabla_S h(\nu)^{\Si}||^2 + ||h(\nu)^{\Si}||^2,
\ee
be the $\bar H^1$ norm of $h(\nu)^{\Si}$. Since $E_0^{\nu} = 0$ (since the initial data are assumed to vanish), it is easy to see that 
$$\max_{t\in [0,1]}E_t^{\nu} \leq 2 \max_{t\in [0,1]} \int_{S_t}||\nabla_{\p_t} h(\nu)^{\Si}||^2 + ||\nabla_S h(\nu)^{\Si}||^2.$$ 
It then follows easily from \eqref{en1} and the Gronwall inequality that there is a numerical constant $c > 0$ such that 
\be \label{1.5}
E_t^{\nu} \leq  c \int_{\cC_t}\<\nabla_{\p_t} h(\nu)^{\Si}, \nabla_{\nu}h(\nu)^{\Si}\> + \l^2 \int_{\w U_t} |\p h|^2 .
\ee

   The main estimate for $h(\nu)^{\Si}$ (replacing the weaker Neumann estimate above) is then the following: 
  
 \begin{proposition} \label{mainprop}
 There exist positive constants $c, C > 0$, depending only on $g$, such that 
$$\max_{t\in [0,1]}E_t^{\nu} + c\l^{-1}||f||_{H^1(\cC)}^2 \leq C\l^{-1}[ ||h_A||_{H^{s+1}(\cC)}^2 + ||\eta||_{H^{s}(\cC)}^2 + ||V'_{\cC}||_{H^{s}(\cC)}^2 ] + \l^2  ||h||_{H^s(\w U)}^2.$$
\end{proposition}

   The proof is rather long and given through several Lemmas, some of which are proved in the Appendix. It will be important to keep track of the signs 
of the dominant terms, i.e.~ those involving derivatives of $h(\nu)^{\Si}$, for the final arguments. The proof of Theorem \ref{Eestthm} then proceeds with 
Proposition \ref{prophnua}. 

\smallskip

     To begin the proof of Proposition \ref{mainprop}, from \eqref{gaugeSi} and \eqref{bA}, we have (viewing $\nabla_{\nu}h(\nu)^{\Si}$ as a 1-form on $\Si$), 
\be \label{hnu2'}
\nabla_{\nu}h(\nu)^{\Si} = -[(\hat A - \hat H \g)(\nabla f)]^{\Si} + \tfrac{1}{2}(d_{\Si}h_{\nu \nu} - \hat H d_{\Si}f) + L 
\ee
where
\be \label{Lower}
L = (\b_{\cC}h_A)^{\Si} + f(\d \hat A + \tfrac{1}{2}d\hat H)^{\Si} - [(A+H\g)h(\nu)^{\tT}]^{\Si} - \<D^2 x^a, h\>dx^a - (V'_h)^{\Si}. 
\ee
Here $L$ consists of terms involving target data as well as lower order terms in $h$. 
Substituting \eqref{hnu2'} into the first term on the right in \eqref{1.5} gives 
\be \label{3}
\int_{\cC_t}\<\nabla_{\p_t} h(\nu)^{\Si}, \nabla_{\nu}h(\nu)^{\Si}\> = \int_{\cC_t}\<\nabla_T h(\nu)^{\Si}, (\hat \Pi(\nabla f))^{\Si}  + 
\tfrac{1}{2}(d_{\Si}h_{\nu \nu} - \hat H d_{\Si}f) + L\> .
\ee
  
  \medskip 

For notation, let  
\be \label{Lambda}
\Lambda = ||g||_{C^3(\cC)}.
\ee
Recall from the localization in \S 2.4 that $\Lambda = O(1)$ while $|\p^k g| = O(\l^k)$. Also let 
\be \label{J} 
B_{\cT} = [||h_A||_{H^{2}(\cC)}^2 + ||\eta_h||_{H^1(\cC)}^2 + ||V'_{\cC}||_{H^1(\cC)}^2 , 
\ee
denote the bound on the target data in $T(\cT^H)$ and  
\be \label{Jeps} 
B_{\l} = B_{\cT} + \l^2 ||h||_{\hat H^1(\w U)}^2,
\ee
where the $\hat H^1(\w U)$ norm is defined following \eqref{ee2.1}. For later use, we recall from \eqref{eef} that 
\be \label{fJ}
||f||_{H^1(\Si_t)}^2 \leq C B_{\l}.
\ee

\medskip 

  First, the lower order $L$ terms are easier to handle and so discussed in Lemmas \ref{Xlemma}-\ref{Blemma} given in the Appendix. These give 
\be \label{Lower1}
\int_{\cC_t}\<\nabla_{\p_t} h(\nu)^{\Si}, L\> \leq \d \max_{t \in [0,1]}E_t ^{\nu}+ \d^{-1}[ \l \Lambda B_{\l} + B_{\cT}], 
\ee
where $\d > 0$ is a small positive parameter. 

   We focus here on the leading order terms in \eqref{3}, beginning with the term $(d_{\Si}h_{\nu \nu} - \hat H d_{\Si}f)$. 

\begin{lemma}\label{divlemma}
For any smooth functions $\f$ and $q$, 
\be \label{fq}
| \int_{\cC_t} \<\nabla_{\p_t} h(\nu)^{\Si}, qd_{\Si}\f\> | \leq \d |q|_{C^0}\max_{t\in [0,1]}E_t^{\nu} + \d^{-1}\max_{t\in [0,1]}||q||_{C^2(\Si_t)}^2 ||\f||_{H^1(\cC)}^2 + 
\d\int_{\cC}(\d_{\Si}h(\nu)^{\Si})^2.
\ee
Moreover, 
\be \label{div}
\int_{\cC}(\d_{\Si}h(\nu)^{\Si})^2  \leq C B_{\l}.
\ee 
Hence, 
\be \label{fq1}
| \int_{\cC_t} \<\nabla_{\p_t} h(\nu)^{\Si}, d_{\Si}h_{\nu \nu} - \hat H d_{\Si}f\> | \leq \d \max_{t\in [0,1]}E_t^{\nu} + \d^{-1} C B_{\l}.
\ee
\end{lemma}

\begin{proof} 
Integrating by parts in the $t$-direction gives 
$$\int_{\cC_t}\<\nabla_{\p_t} h(\nu)^{\Si}, q d_{\Si}\f \> = \frac{d}{dt}\int_{\cC_t}\<h(\nu)^{\Si}, q d_{\Si}\f\> - \int_{\cC_t}\<h(\nu)^{\Si}, q d_{\Si}\dot \f+ \dot q d_{\Si}\f\>,$$
where here (and in the following) we ignore ignore lower order terms involving the $t$-derivative of the (background) metric and volume form as well as 
terms resulting from commuting derivatives. Applying the divergence theorem then gives 
\be \label{5}
\int_{\cC_t}\<\nabla_{\p_t} h(\nu)^{\Si}, qd\f \> = \int_{\Si_t}\<h(\nu)^{\Si}, qd^{\Si}\f\> - \int_{\cC_t}\dot \f \d_{\Si}(q h(\nu)^{\Si}) + \f \d_{\Si}(\dot q h(\nu)^{\Si}). 
\ee 
For the first term on the right in \eqref{5}, write 
$$|\int_{\Si_t}\<h(\nu)^{\Si}, qd^{\Si}\f\>| \leq |q|_{C^0}||h(\nu)^{\Si}||_{H^{1/2}(\Si_t)}||\f||_{H^{1/2}(\Si_t)} \leq 
 |q|_{C^0}(\d E_{\nu}^t + \d^{-1} ||\f||_{H^{1}(\cC)}^2),$$  
where we have used the Cauchy-Schwarz type inequality \eqref{Sob3} together with \eqref{Sob2}. 

   Next, $\d_{\Si}(q h(\nu)^{\Si}) = q \d_{\Si}h(\nu)^{\Si} - \<h(\nu)^{\Si}, d q\>$ and $\d_{\Si}(\dot q h(\nu)^{\Si}) = \dot q \d_{\Si} h(\nu)^{\Si} - 
\<h(\nu)^{\Si}, d \dot q\>$. Then 
$$\int_{\cC_t}\dot \f q \d_{\Si}h(\nu)^{\Si} \leq \d \int_{\cC_t}(\d_{\Si}h(\nu)^{\Si})^2 + \d^{-1} \int_{\cC_t} \dot \f^2 q^2,$$
and
$$\int_{\cC_t}\dot \f \<h(\nu)^{\Si}, dq\> \leq \d \int_{\cC_t}|h(\nu)^{\Si}|^2 + \d^{-1} \int_{\cC_t} \dot \f^2 |dq|^2.$$
This gives 
$$| \int_{\cC_t}\dot \f \d_{\Si}(q h(\nu)^{\Si}) | \leq \d \int_{\cC_t}(\d_{\Si}h(\nu)^{\Si})^2 + \d \int_{\cC_t}|h(\nu)^{\Si}|^2 + 
\d^{-1} ||q||_{C^1(\cC_t)}^2 \int_{\cC_t} \dot \f^2.$$
The second term $\d_{\Si}(\dot q h(\nu)^{\Si})$ is treated in the same way. Combining the estimates above then gives \eqref{fq}. 

  To prove \eqref{div}, we use the normal component of the harmonic gauge equation \eqref{gaugen}, i.e.  
\be \label{bnu2}
(V'_h)(\nu) = -\tfrac{1}{2}\nu(h_{\nu \nu}) + \tfrac{1}{2}\nu(\tr_{\cC}h^{\tT}) + \d_{\cC}(h(\nu)^{\tT}) - h_{\nu \nu}H + \<A, h^{\tT}\> - \<D^2 x^1, h\>.
\ee
On the right side of \eqref{bnu2}, the last three terms are of zero order in $h$, while the first three terms are first order in $h$. We derive the  
estimates for the three first order terms below; the estimates for the three zero order terms then follows easily. 

  Observe that \eqref{bnu2} implies that  
$$||\d_{\cC}h(\nu)^{\tT}||_{L^2(\cC_t)}^2 \leq C[ ||h_{\nu \nu}||_{\bar H^1(\cC_t)}^2 + ||\tr_{\cC}h^{\tT}||_{\bar H^1(\cC_t)}^2 + ||V'_{\cC}||_{L^2(\cC)}^2],$$
where we have ignored the zero order terms in \eqref{bnu2}. By the boundary stable estimates 
\eqref{ee4}, \eqref{ee1} and \eqref{ee2.1}, this gives 
\be \label{dc}
\max_{t\in [0,1]} ||\d_{\cC}h(\nu)^{\tT}||_{L^2(\cC_t)}^2 \leq C[ ||h_A||_{H^2(\cC)}^2 + ||\eta||_{H^1(\cC)}^2 + ||V'_{\cC}||_{H^1(\cC)}^2 + 
\l^2 ||h||_{\hat H^1(\w U)}^2 ].
\ee

  Next, modulo lower order $\l$ terms, we have 
$$\d_{\cC}h(\nu)^{\tT} = \p_t h_{0\nu} + \d_{\Si}h(\nu)^{\Si},$$
and using the boundary stable Sommerfeld estimate \eqref{SomE} for $h_{0\nu}$ gives 
$$\int_{\cC_t}|\p h_{0\nu}|^2 \leq C [ ||h_A||_{H^2(\cC_t)}^2 + ||\eta||_{H^1(\cC_t)}^2 + \l^2 ||h||_{H^1(\w U_t)}^2.$$
Combining this with \eqref{dc} gives \eqref{div}. 

  Finally to prove \eqref{fq1}, set $q = \hat H$ and $\f = f$ and use the estimate \eqref{eef}; also set $q = 1$ and $\f = h_{\nu \nu}$ and 
use the boundary stable energy estimate \eqref{ee4}. Combining the estimates above completes the proof. 

\end{proof}

Combining the results above with \eqref{3} gives the bound 
\be \label{alp}
\int_{\cC_t}\<\nabla_{\p_t} h(\nu)^{\Si}, \nabla_{\nu}h(\nu)^{\Si}\>   \leq K + \int_{\cC_t}\<\nabla_{\p_t} h(\nu)^{\Si}, (\hat \Pi(\nabla f))^{\Si} \>, 
\ee
where 
\be \label{K}
K = \d \Lambda \max_{t\in [0,1]}E_t^{\nu} + \d^{-1} CB_{\l}.
\ee
  Finally, we estimate the last remaining term in \eqref{alp}. This is the most significant estimate. 
 
\begin{lemma}\label{mainlemma}
There are fixed numerical constants $c_1, c_2 > 0$, independent of $h$ and $g$, such that 
\be \label{Pi1}
 \int_{\cC_t}\<\nabla_{\p_t} h(\nu)^{\Si}, (\hat \Pi(\nabla f))^{\Si} \>  \leq c_1 K + c_2\l^{-1}\d^{-1} B_{\cT} - c_2 \l^{-1}||f||_{\w H^1(\Si_t)}^2 . 
\ee
\end{lemma} 

\begin{proof}
Let 
\be \label{D}
D =  \int_{\cC_t}\<\nabla_{\p_t} h(\nu)^{\Si}, (\hat \Pi(\nabla f))^{\Si} \> .
\ee
As noted above, it will be important to pay attention to the signs related to the dominant term $D$. 
As in the previous Lemma, integrating by parts in the $t$-direction and ignoring lower order terms gives 
\be \label{D2}
 D = \int_{\Si_t}\< h(\nu)^{\Si}, (\hat \Pi(\nabla f))^{\Si} \> - \int_{\cC_t}\<h(\nu)^{\Si}, (\nabla_{\p_t} \hat \Pi)(\nabla f)\> - \int_{\cC_t}\<h(\nu)^{\Si}, [\hat \Pi(\nabla \dot f)]^{\Si} \>.
\ee
The first two terms on the right in \eqref{D2} are bounded by $cK$, by using the Sobolev trace inequality \eqref{Sob2} and \eqref{fJ}. Thus we have 
\be \label{Pi2}
D \leq -\int_{\cC_t}\<h(\nu)^{\Si}, [\hat \Pi(\nabla \dot f)]^{\Si} \> + cK. 
\ee

  Next, note that $\hat \Pi(\nabla \dot f) = \hat \Pi(\nabla_{\cC}\dot f) = -\d_{\cC}(\dot f \hat \Pi)  + \dot f \d(\hat \Pi) = -\d_{\cC}(\dot f \hat \Pi) - 
\dot f \hat \Ric(\nu, \cdot)$, using \eqref{GC}. 
The last term here is bounded by $K$ in the same way as above. The remaining issue is then to bound the term 
$\int_{\cC_t}\<h(\nu)^{\Si}, \d_{\cC}(\dot f \hat \Pi) \>$. By the divergence theorem,  
$$\int_{\cC_t}\<h(\nu)^{\Si}, \d_{\cC}(\dot f \hat \Pi) \> =  \int_{\cC_t}\<\d^*h(\nu)^{\Si}), \dot f \hat \Pi\> + \int_{\Si_t} \dot f \hat \Pi(T, h(\nu)^{\Si}),$$
(using the fact that the initial data on $\Si_0$ vanish). The last term here is again bounded by $K$ by the previous 
arguments, so that 
\be \label{hatPi}
D \leq \int_{\cC_t}\<\d^*(h(\nu)^{\tT}), \dot f \hat \Pi\> + cK. 
 \ee

   To handle this term, we use the standard formula for variation $A'_h$ of the second fundamental form: $A = \frac{1}{2}\cL_{\nu}g$, so 
that $A'_h = \frac{1}{2} \cL_{\nu}h + \cL_{\nu'_h}g$. A standard further computation gives the formula 
\be \label{A'}
2A'_h = \nabla_{\nu}h^{\tT} + 2A\circ h^{\tT} - 2 \d^* h(\nu)^{\tT} - h_{\nu \nu}A,
\ee
where the middle term appears in \eqref{hatPi}. Note here that $A$ is not renormalized, so that we are working with $A$ and not $\hat A$. 
As above, the zero order terms in $h$ will be ignored. Now for the term $\nabla_{\nu}h^{\tT}$, using \eqref{ee1} and \eqref{ee2.1}, we have the estimate 
$$|\int_{\cC_t}\<\nabla_{\nu} h^{\tT}, \dot f \hat \Pi\> | \leq C[||h_A||_{\bar H^1(\cC_t)}^2 + ||f||_{\bar H^1(\cC_t)}^2 + \Lambda^2 \int_{\cC_t} (\dot f)^2 ] \leq cK.$$
Thus we have shown so far that 
\be \label{APi}
D \leq -\int_{\cC_t}\<A'_h, \dot f \hat \Pi\> + cK.
\ee

  At this point, we return to the formulas \eqref{Hlin}-\eqref{finhom}. On the right hand side of \eqref{finhom}, we have the term $2(\<A'_h, A\> - H'_h H_{\cC}) = 
  -2\<A'_h, \Pi\>$. In contrast to the right side of \eqref{Hlin} however, here in \eqref{APi} we have $\hat \Pi = \l^{-1}\Pi$ in place of $\Pi$ with $\l^{-1} > > 1$. 
Using \eqref{Hlin}, write then 
\be \label{hatPi2}
-\int_{\cC_t}\dot f \<A'_h, \hat \Pi\> dv_{g_{\cC}} = \tfrac{1}{2}\l^{-1}\int_{\cC_t} \dot f \w \Box_{\cC}f dv_{g_{\cC}} + \dot f \l^{-1}R,
\ee
where $R$ consists of the remaining (lower order) terms in \eqref{Hlin}, analysed later. 

  Recall that $\w \Box_{\cC}$ is the wave operator with respect to the Lorentz metric $\zeta := \hat \Pi$ and $\zeta(\p_t, \p_t) < 0$. Let $\sqrt{|g_{\zeta}|}$ be 
the volume density of $\zeta$, and write $\sqrt{|g_{\cC}|} = \o \sqrt{|g_{\zeta}|}$, where $\o > 0$ and $\sqrt{|g_{\cC}|}$ is the volume density of $g_{\cC}$. 

  This gives 
$$-\int_{\cC_t} \dot f \<A'_h, \hat \Pi\> dv_{g_{\cC}} = \tfrac{1}{2}\l^{-1}\int_{\cC_t} \o\dot f \Box_{\zeta}f dv_{g_{\zeta}} + \dot f \l^{-1}R.$$
The term $\int_{\cC_t}\dot f \Box_{\zeta}f dv_{g_{\zeta}}$ is just  the term used to compute the $H^1$ energy of $f$ along $\cC$. Thus, applying 
the divergence theorem (as in the usual proof for scalar energy estimates) gives 
\be \label{linverse}
\l^{-1}\int_{\cC_t}\o\dot f \Box_{\zeta}f dv_{g_{\zeta}} = - \tfrac{1}{2}\l^{-1}\int_{\Si_t} (|\dot f|^2 + |d_{\Si} f|^2) dv_{\Si_t} - \l^{-1}\int_{\cC_t} \dot f \<d \ln \o, df\>_{\zeta}dv_{g_{\cC}} 
+ \dot f \l^{-1}R.
\ee
Now the main point is that the first term on the right is negative, giving the (beneficial) negative term on the right side of \eqref{Pi1}. Next, for the $d\ln \o$ term, we have 
$|d\ln \o| \sim \l$, since $\p(\hat \Pi) = O(\l)$. So here the $\l$ terms cancel, leaving the usual $H^1$ energy of $f$, which may be easily absorbed into the 
previous large negative term. 

  The proof is then completed by estimating the last $\dot f \l^{-1}R$ term. Here $R$ consists of the terms $q(\p f)$ in \eqref{Hlin} and the terms in 
  \eqref{finhom} besides the $\<A'_h, \Pi\>$ term. A brief inspection shows that the terms in $\l^{-1}\dot f q(\p f)$ are all bounded by $||f||_{H^1(\cC)}^2$, 
  and so these terms may be absorbed as above. 
  
      For the terms in \eqref{finhom}, beginning with $\dot f \l^{-1}R'_{h_A}$, we have 
 $$|\int_{\cC_t} \dot f \l^{-1}R'_{h_A} | \leq \d \l^{-1}||f||_{H^1(\cC_t)}^2 + \d^{-1}\l^{-1}||h_A||_{H^2(\cC_t)}^2.$$
 The first term here is small compared with the $\l^{-1}$ term in \eqref{linverse} and so again may be absorbed, while the second term is of the form 
 $\d^{-1}\l^{-1}B_{\cT}$. (The same analysis holds when $F \neq 0$ and for the initial data). Next the terms involving $A^2$, $\Ric$ are $O(\l^2)$ so 
 these contribute terms controlled by $B_{\l}$. 
 
  Next we come to the gauge terms: $\d^* V'_h$ and $(\d^*)'_h V$. The normal derivatives $\d^*V'_h(\nu, \nu)$ are controlled by the Dirichlet boundary data 
$V'_{\cC}$ as in \eqref{VbarV} and so are bounded by $B_{\l}$. Since $V = O(\l)$, the same is true for $(\d^*)'_h V$ as well as for the remaining trace 
terms in \eqref{finhom}. 

  This completes the proof of Lemma \ref{mainlemma} and so also completes the proof of Proposition \ref{mainprop}. 

\end{proof}  

  Combining the results above leads then quite easily to the following main estimate on $h(\nu)^{\Si}$. 
  
 \begin{proposition} \label{prophnua}
 For $(\w U, g)$ as in Theorem \ref{Eestthm}, the term $h(\nu)^{\Si}$ satisfies the estimate 
\be \label{hnuSiest}
\max_{t\in [0,1]}||h(\nu)^{\Si}||_{\bar H^s(S_t)}^2 \leq C\l^{-1}[ ||h_A||_{H^{s+1}(\cC)}^2 + ||\eta||_{H^{s}(\cC)}^2 + ||V'_{\cC}||_{H^{s}(\cC)}^2 ] + 
\l^2  ||h||_{H^s(\w U)}^2 ,
\ee
provided the bulk data $F$, the initial data for $h$ and the initial data for the gauge field $V'_h$ vanish at $S$. 
\end{proposition}  

\begin{proof}
We first prove the result for $s = 1$. By combining the results above, one obtains the estimate 
\be \label{finest}
\max_{t\in 0,1}E_t^{\nu}  + c\l^{-1}||f||_{H^1(\cC)}^2 \leq  \d \Lambda \max_{t\in [0,1]}E_t^{\nu} + 
\d^{-1} cB_{\l} + c\l^{-1}B_{\cT}. 
\ee
Recall that $\Lambda$ is a fixed $C^3$ bound for $g$. Choose $\d$ sufficiently small, depending on $\Lambda$, so that the coefficient $\d \Lambda < \frac{1}{2}$. 
The term $\max_{t\in [0,1]}E_t^{\nu}$ on the right in \eqref{finest} may be absorbed into the left on \eqref{finest}, which then gives 
\be \label{Ehnu}
\begin{split}
\max_{t \in [0,1]}||h(\nu)^{\Si}||_{\bar H^1(S_t)}^2 + \l^{-1}||f||_{H^1(\cC)}^2 & \leq  C\l^{-1}[||h_A||_{H^{2}(\cC)}^2 + ||\eta_h||_{H^1(\cC)}^2 + ||V'_{\cC}||_{H^1(\cC)}^2 ] \\
& + c\l^2 ||h||_{H^1(\w U)}^2  + c\l^2 ||h||_{H^1(\cC)}^2 . \\
\end{split}
\ee

  The first $\l^2$ term over $\w U$ comes from the bulk $P(\p h)$ term. Regarding the last $\l^2$ term on the right over $\cC$, a simple inspection shows that this term 
comes only from the original $f$ estimate in \eqref{ee2.1} and in the Sommerfeld boundary data estimate \eqref{ee3}. For the $f$ estimate, the $h$ terms appear on the 
right of \eqref{finhom}.  These terms are exactly the terms analysed in Lemma \ref{mainlemma} from \eqref{A'} onwards, but now with $\Pi$ in place of $\hat \Pi$, i.e.~without 
the $\l^{-1}$ term. These then are easily absorbed into the other terms of \eqref{Ehnu} as previously. Similarly, for the Sommerfeld boundary data, the $h$ terms appear 
only as $H^0 = L^2$ terms, so at lower order, and so again are easily absorbed. 

  Thus, \eqref{Ehnu} may be written as 
\be \label{Ehnu1}
\max_{t \in [0,1]}||h(\nu)^{\Si}||_{\bar H^1(S_t)}^2  \leq  C\l^{-1}[||h_A||_{H^{2}(\cC)}^2 + ||\eta_h||_{H^1(\cC)}^2 + ||V'_{\cC}||_{H^1(\cC)}^2 ]
+ \l^2 ||h||_{H^1(\w U)}^2 , 
\ee
which proves \eqref{hnuSiest} with $s = 1$. 

  Taking derivatives of \eqref{Lnu} and commuting in the usual way for such energy estimates, cf.~\cite{BS}, \cite{Sa}, then gives the higher order estimates 
\eqref{hnuSiest}.

\end{proof}

The estimate \eqref{hnuSiest} is now an effective improvement and replacement for \eqref{ee5}. Summing up the previous estimates 
\eqref{eef}-\eqref{ee4} together with \eqref{hnuSiest} replacing \eqref{ee5}, by choosing $\l$ sufficiently small depending on $g$, one may absorb 
the $\l ||h||_{H^s(\w U)}$ terms on the right side of these equations into the total sum on the left, using the fact that 
$|h||_{H^s(\w U)} \leq \max_{t\in [0,1]}||h||_{\bar H^s(S_t)}$, giving then 
\be \label{Etot}
\max_{t\in [0,1]} ||h||_{\bar H^s(S_t)}^2 \leq  C\l^{-1}[ ||h_A||_{H^{s+1}\cC)}^2 + ||\eta_h||_{H^{s}(\cC)}^2 + ||V'_{\cC}||_{H^s(\cC)}^2]. 
\ee

In the discussion above, we have assumed that $F = L(h) = 0$, and that the initial data vanish. By linearity, and by the remarks given previously regarding the 
regularity required for $F$ and the initial data, it is then easy to see that \eqref{Etot} gives \eqref{Eest} when this vanishing assumption is dropped. 

The bound on the left in \eqref{Etot} gives a bound on 
\be \label{Linf}
h \in L^{\infty}(I, \bar H^s(S)).
\ee
However, from the regularity theory for wave equations with Dirichlet or Sommerfeld boundary data as in Lemma \ref{DirSomm}, by performing the 
computations in the proof of Theorem \ref{Eestthm} over arbitrarily small time intervals $I' = [t_0, t_1] \subset I$, it is easy to verify that this can be improved 
to a bound on 
$$h \in C^0(I, \bar H^s(S)),$$
so that \eqref{Etot} gives in fact \eqref{Eest}. This completes the proof of Theorem \ref{Eestthm}.  

\end{proof}

\begin{remark} 
{\rm We do not if know the $\bar H^s(S_t)$ estimate \eqref{hnuSiest} for $h(\nu)^{\Si}$ can be strengthened into a boundary stable estimate, where the norm is 
the $\bar \cH^s(S_t)$ norm. This lack of a boundary stable estimate, which is important in the general well-posedness theory for the IBVP, 
is of less importance here. Of course by the Sobolev trace theorem \eqref{Sob1}, there is a boundary stable estimate along $\cC$ 
with a loss of half-a-derivative. 
}
\end{remark}

\section{Construction of linearized solutions.} 

\subsection{Shifted boundary data.}
In this subsection, we construct solutions $h$ to the linear problem $D\w \Phi_g^H(h) = \w \tau'$ by solving Dirichlet, Sommerfeld and Neumann 
initial boundary value problems for the boundary data determined by the equations \eqref{Hlin}, \eqref{h01} and \eqref{hnua} of \S 3. This is done 
by a Picard-type iteration process whose convergence is assured by the presence of the energy estimates from \S 4. 

     The main result for this subsection is the following surjectivity or existence theorem. Let 
\be \label{tau'1}
\w \tau' = (F, (\g_S', \k', \nu', V_S'), (\g_A', \eta, V_{\cC}')) 
\ee
be a general element in $T_{\w \tau}(\w \cT^H)$, $\w \tau = \w \Phi^H(g)$. 
  
\begin{theorem}\label{modexist}
There exists $\e > 0$ such that, for any smooth metric $g$ on $\w U$ which is $C^{\infty}$ close to a standard Minkowski corner metric $g_{\a_0}$ 
as in \eqref{eps} and satisfies Assumption (*), the linearization $D\w \Phi_g^H$ at $g$ satisfies: given any $C^{\infty}$ target data 
$\w \tau' \in T_{\w \tau}(\w \cT^H)$ on $\w U$, there exists a variation $h \in T_g Met(\w U)$, such that 
\be \label{DwPhisol}
D\w \Phi_g^H(h) = \w \tau'.
\ee
Thus, the equation 
\be \label{LhF}
L(h) = F,
\ee
has a $C^{\infty}$ solution $h$ such that 
\be \label{IC1}
(g_S, K)'_h = (\g_S', \k'), \ \ (\nu_S)'_h = \nu', \ \ V'_h = V_S' \ \mbox{ on }S,
\ee
and
\be \label{BC1}
(([g_\cC]_A)'_h, \eta_h) = (\g'_A, \eta), \ \ V'_h = V_\cC' \ \mbox{ on }\cC.
\ee
Further the constructed solution $h$ depends smoothly on the data $\w \tau'$. 
\end{theorem}
  
 \begin{proof}

  We first prove Theorem \ref{modexist} in a special case; following that, the proof is completed by means of the energy estimate in Theorem \ref{Eestthm}. 
 To begin, let $\cB \subset T_{\w \tau}T(\w \cT^H)$ be the subset of target data $\w \tau'$ for which the $C^k$ norms of $\w \tau'$ are uniformly bounded in $k$: thus 
$\cB = \cup_{M\in \bZ^+} \cB_M$, where $\cB_M$ is the space of $\w \tau'$ such that 
\be \label{Mbound}
||\w \tau'||_{C^k} \leq M,
\ee
for all $k \in \bZ^+ \cup \{0\}$. Similarly, we assume in this special case that the background metric $g$ satisfies $g \in \cB$, i.e.~the $C^k$ derivatives 
of $g$ are uniformly bounded, independent of $k$. 

   Theorem \ref{modexist} will be proved for $\w \tau', g \in \cB$ by a standard Picard-type iteration process. As in \S 4, we work with separate equations 
for each of the terms 
\be \label{comp}
h^{\tT}, \ h_{0\nu}, \ h_{\nu \nu}, \ h(\nu)^{\Si},
\ee
comprising $h$. The equations for $f$, $h(\nu, \cdot)$ along the boundary $\cC$ are those described in \S 3.

\medskip

0. In this step, we solve all the relevant equations in \S 3 without all the coupling terms, i.e.~the $P$-terms in the bulk and the $E$-terms on $\cC$ are 
set to zero, and we use only the given target data in $T(\w \cT^H)$ (i.e.~$F$ and $Y$ terms) for bulk, initial and boundary data. This is done step-by-step 
for each of the terms in \eqref{comp}, 

  We begin with the construction of the gauge field. Thus let $V_0'$ be the solution to the modified gauge equation 
\be \label{modV'}
-\tfrac{1}{2}(\Box_g V'_0+ \Ric_g (V'_0)) + \tfrac{1}{2}\nabla_{V} V'_0 = \b F,
\ee
(so the `error' term $\b'_h \Ric + O_{2,1}(V,h)$ is dropped from \eqref{V'}) and with initial and boundary data given by the target data $V'_S$ and $V'_{\cC}$.  
By the Dirichlet energy estimate \eqref{DirE}, we have 
$$||V_0'||_{\bar \cH^s(S_t)} \leq C||\w \tau'||_{H^s(\w \cT^H)}.$$

  Next solve the modified evolution equation \eqref{Hlin} for $f_0$ along $\cC$, 
\be \label{Hlin2} 
\w \Box_\cC f _0 + q(\p f_0) = R'_{h_A} - \tr(F-\d^* V_0') + 2(F-\d^* V_0')(\nu,\nu),
\ee
so that $E$ is set to zero and $V'_h$ is set to $V_0'$. Recall from Lemma \ref{fft} that the initial conditions $f_0 |_{\Si}$ and $\p_t f_0 |_{\Si}$ are determined 
by the target data in $T(\w \cT)$. Solving \eqref{Hlin2} on $\cC$ with this data gives the ($C^{\infty}$) Dirichlet boundary value for $f_0$ along 
$\cC$. Standard energy estimates then give 
$$||f_0||_{H^s(\cC)} \leq C||\w \tau'||_{H^s(\w \cT^H)}.$$
Since the Dirichlet data $h_A$ is given in $\w \tau'$, the Dirichlet boundary data for $h_0^{\tT} = h_A + f\hat A$ is thus also given. 
With such Dirichlet boundary data, we then solve the bulk equation 
\be \label{hT0}
-\tfrac{1}{2}\Box_g h_0^{\tT}   = F^{\tT}, 
\ee
with the given initial conditions from $\w \tau'$ on $\w U$. By \eqref{DirE}, the solution $h_0^{\tT}$ satisfies the energy estimate 
\be \label{H0Test}
||h_0^{\tT}||_{\bar \cH^s(S_t)} \leq C||\w \tau'||_{H^s(\w \cT^H)}.
\ee

   Turning next to $h_{0\nu}$, define $(h_0)_{0\nu}$ in $\w U$ to be the unique solution to 
$$-\tfrac{1}{2}\Box_g (h_0)_{0\nu} = -\tfrac{1}{2}\Box_g( \frac{1}{|\nabla x^1|}g^{1 \a}h_0(\p_{\a}, \p_t)) = F_{0\nu},$$
which satisfies the Sommerfeld boundary condition \eqref{h01} for $(h_0)_{0\nu}$, with $\hat Y$ given by the target data in $\w \tau'$, $\hat Z 
= \hat Z(f_0)$ given by $f_0$ above and $\hat E = 0$ on $\cC$ together with the initial conditions from $\w \tau'$. The Sommerfeld energy 
estimate \eqref{SomE} then gives 
\be \label{00nu}
||(h_0)_{0\nu}||_{\bar \cH^s(S_t)} \leq C||\w \tau'||_{H^s(\w \cT^H)}.
\ee
The induced Dirichlet boundary data for $(h_0)_{0\nu}$ on $\cC$ then gives Dirichlet data for $(h_0)_{\nu \nu}$ on $\cC$ in terms of 
the prescribed $\eta$, cf.~\eqref{eta'}. We then solve for $(h_0)_{\nu \nu}$ in the bulk $\w U$, 
$$-\tfrac{1}{2}\Box_g (h_0)_{\nu \nu}  = F_{\nu \nu},$$
with these given Dirichlet boundary conditions for $h_{\nu \nu}$ and with the given initial conditions from $\w \tau'$. Via \eqref{00nu}, the 
Dirichlet boundary data for $h_{\nu\nu}$ are controlled along $\cC$ and so from the Dirichlet energy estimate \eqref{DirE}, we have the bound 
$$||(h_0)_{\nu\nu}||_{\bar \cH^s(S_t)} \leq C||\w \tau'||_{H^s(\w \cT^H)}.$$
 
   Finally we carry out the same process for the components $h_{\nu a}$ of $h(\nu)^{\Si}$. Let $h_0(\nu)^{\Si}$ be the solution to 
$$-\tfrac{1}{2}\Box_g (h_0)_{\nu a} = F_{\nu a},$$
with the Neumann data $\nu((h_0)_{\nu a}) \sim (\nabla_{\nu}h(\nu))^a$ given by \eqref{hnua}. Here we set $\w Y(\cT)$ from 
the target data in $T(\w \cT)$, so in particular $V'_h = V_0'$, $f = f_0$ and $h_{\nu \nu} = (h_0)_{\nu\nu}$ and $\w E = 0$. From the Neumann estimate 
\eqref{NeuE} we obtain 
\be \label{0nua}
||(h_0)_{\nu a}||_{\bar \cH^s(S_t)} \leq C||\w \tau'||_{H^{s+1}(\w \cT^H)}.
\ee
Note the significant loss-of-derivative compared with the previous estimates. 

 Summing up the solutions above gives a solution $h_0$ to 
\be \label{h0F}
-\tfrac{1}{2}\Box_g h_0 = F,
\ee
in $\w U$ with the prescribed initial and boundary conditions from $\w \tau'$. Further 
$$||h_0||_{\bar \cH^s(S_t)} \leq C||\w \tau'||_{H^{s+1}(\w \cT^H)}.$$

 The next stages are error, so $O(\e)$ adjustments to $h_0$, and so are denoted $E_1$, $E_2$, etc. In contrast to 
Step 0, we now set all the target data in the equations from \S 3 to zero while the $E$ terms in these equations become non-zero. 

1.  First determine the correction to the gauge by letting $V_1'$ be the solution to the equation
$$-\tfrac{1}{2}(\Box_g V'_1+ \Ric_g (V'_1)) + \tfrac{1}{2}\nabla_{V} V'_1 = \b'_{h_0} \Ric + O_{2,1}(V,h_0),$$
with zero initial and Dirichlet boundary data. As above, we then have the estimate 
$$||V_1'||_{\bar \cH^s(S_t)} \leq C\e ||h_0||_{H^s(\w U)} \leq C\e ||\w \tau'||_{H^{s+1}(\w \cT^H)}.$$

  Next, solve the evolution equation \eqref{Hlin} for $f_1$ along $\cC$ with $Y(\cT, V'_h) = Y(0, V_1')$, $E = E(h_0)$ and zero initial data. This gives 
Dirichlet boundary data for $f_1$ on $\cC$ and since the coefficients of $E$ are $O(\e)$ in $C^{\infty}$ by \eqref{eps}, 
$$||f_1||_{H^s(\cC)} \leq C\e ||h_0||_{H^s(\cC)} \leq C\e ||\w \tau'||_{H^{s+1}(\w \cT^H)}.$$
We set $(E_1)_A = 0$ on $\cC$ and so $E_1^{\tT} = f_1 \hat A$. Using this Dirichlet boundary value, we then solve 
$$-\tfrac{1}{2}\Box_g E_1^{\tT} = -P^{\tT}(h_0),$$
in $\w U$ with zero initial data. The Dirichlet energy estimate then gives 
$$||E_1^{\tT}||_{\bar \cH^s(S_t)} \leq C \e ||\w \tau'||_{H^{s+1}(\w \cT^H)}.$$

  Next, for $(E_1)_{0\nu}$, solve the IBVP 
$$-\tfrac{1}{2}\Box_g (E_1)_{0\nu} = -P_{0\nu}(h_0),$$
on $\w U$ with Sommerfeld boundary data in \eqref{h01} with $\hat Y = 0$, $\hat Z = \hat Z(f_1)$, $\hat E = \hat E(h_0)$ and with zero 
initial data and with the Sommerfeld data given above. As above this gives 
$$||(E_1)_{0\nu}||_{\bar \cH^s(S_t)} \leq C \e ||\w \tau'||_{H^{s+1}(\w \cT^H)}.$$
Next set $(E_1)_{\nu \nu} = -(E_1)_{0\nu}$ on $\cC$ (since there is no correction to $\eta$) and solve the Dirichlet IBVP 
$$-\tfrac{1}{2}\Box_g (E_1)_{\nu \nu} = -P_{\nu \nu}(h_0)$$
in $\w U$, with zero initial data. Then  
$$||(E_1)_{\nu\nu}||_{\bar \cH^s(S_t)} \leq C \e ||\w \tau'||_{H^{s+1}(\w \cT^H)}.$$

  Finally, solve the IBVP 
$$-\tfrac{1}{2}\Box_g (E_1)_{\nu a} = -P_{\nu a}(h_0)$$
in $\w U$, with zero initial data and with Neumann boundary data given by \eqref{hnua} with $\w Y(\cT, V'_h) = (0, V_1')$, $\w Z = Z(f_1, (E_1)_{\nu \nu})$ 
and $\w E = \w E(h_0)$. As previously, we then have the bound  
$$||(E_1)_{\nu a}||_{\bar \cH^s(S_t)} \leq  C\e||\w \tau'||_{H^{s+2}(\w \cT^H)}.$$

  As in Step 0, summing these components together then gives $E_1$ solving 
$$-\tfrac{1}{2}\Box_g E_1  = -P(h_0),$$
in $\w U$. The term $E_1$ has zero target initial data and zero target boundary data and satisfies the estimate 
$$||E_1||_{\bar \cH^s(S_t)} \leq C\e ||\w \tau'||_{H^{s+2}(\w \cT^H)}.$$

2. Next, $E_2$ is constructed in the same way as $E_1$ above, with $E_1$ in place of $h_0$ and $E_2$ in place of 
$E_1$. Thus, first $V_2'$ solves 
$$-\tfrac{1}{2}(\Box_g V'_2+ \Ric_g (V'_2)) + \tfrac{1}{2}\nabla_{V} V'_2 = \b'_{E_1} \Ric + O_{2,1}(V,E_1),$$
and with zero initial and Dirichlet boundary data. As previously, we then have the estimate 
$$||V_2'||_{\bar \cH^s(S_t)} \leq C\e^2  ||\w \tau'||_{H^{s+2}(\w \cT^H)}.$$
Similarly, $E_2$ solves 
$$-\tfrac{1}{2}\Box_g E_2  = -P(E_1),$$
in $\w U$, with zero target  initial data and zero target boundary data. The boundary data for $E_2$ is constructed by 
setting $Y(\cT, V'_h) = (0, V'_2)$, $E = E(E_1)$ in \eqref{finhom}, $\hat Z = \hat Z(f_2, (E_2)_{\nu \nu})$, $\hat E = \hat E(E_1)$ in \eqref{h01terms} and 
$\w Z = \w Z(f_2)$, $\w E = \w E(E_1)$ in \eqref{hnua}.  As before, we then obtain the bound 
$$||E_2||_{\bar \cH^s(S_t)} \leq C_2\e ||E_1||_{\bar \cH^{s+2}(S_t)} \leq C_2 C_1\e^2||\w \tau'||_{H^{s+4}(\w \cT^H)},$$
where the constants $C_i$ reflect the dependence of the estimates on $s$. 

   Continuing inductively in this way, it follows that if $\w \tau' \in \cB$, i.e.~$\w \tau' \in \cB_M$ for some $M$, then the sequence $\{E_m\}$, 
$m \geq 1$, satisfies the bound 
$$||E_m||_{\bar \cH^s(S_t)} \leq C_m\e ||E_{m-1}||_{\bar \cH^{s+2}(S_t)} \leq C_m C_{m-1} \e^2 ||E_{m-2}||_{\bar \cH^{s+4}(S_t)} \leq \cdots \leq 
(\prod_1^m C_i) \e^m ||\w \tau'||_{H^{s+2m}(\w \cT^H)}.$$
The constants $C_i$ derive from the energy estimates \eqref{DirE}, \eqref{SomE} and \eqref{NeuE}. These depend on the $C^s$ norm of the 
background metric $g$ at order $s$, due to the commutator $[\p_{x^{\a}}, \Box_g] = O(\p^2 g)$ and its iterates. However, since by assumption 
$g \in \cB$ also, there is a fixed constant $C$, independent of $s$, such that $\prod_1^m C_i \leq C^m$, so that  
$$||E_m||_{\bar \cH^s(S_t)} \leq C^m \e^m M.$$
By choosing $\e$ sufficiently small, it follows that the sequence $h_k = h_0 + \sum_1^k E_i$ converges to a limit $h \in C^{\infty}$. 
By construction we have 
$$-\tfrac{1}{2}\Box_g(h_0 + \sum E_i)  = F - P(h_0 + \sum E_i),$$
so that 
\be \label{Lbarh}
L(h) = F.
\ee
Hence $h$ solves \eqref{LhF} as well as the initial and boundary conditions in \eqref{IC1}-\eqref{BC1} except those involving the gauge field. The 
sequence of gauge fields $(V')^n  = V_0' + \sum_1^{n}V_m'$ also converges in $C^{\infty}$ to a limit field $V'$ 
satisfying \eqref{V'} and the initial and boundary conditions $V_S', V_{\cC}'$. It remains to show that 
$$V' = V'_h.$$ 
By construction, both $V'$ and $V'_h$ have the same initial values on $S$ and $(V')^{\tT} = (V'_h)^{\tT}$ on $\cC$. The construction of $h$ does not use 
any equation for $V'_h(\nu)$. To determine such an equation, note first that it follows by taking limit of equation \eqref{Hlin} that $f = \lim_{i\to \infty}\sum_0^i f_j$ 
satisfies
\be \label{Hlin22}
\begin{split} 
\w \Box_{\cC} f + q(\p f) & =  R'_{h_A} - \tr(F - \d^* V') + 2(F - \d^* V')(\nu,\nu)\\
 &+2\<A'_h, A\> - 2H_\cC H'_h - 2\<A\circ A, h\> + 4\Ric(\nu,\nu'_h) + \<\Ric, h\>\\
&+ \tr [(\d^*)'_h V] - 2[(\d^*)'_h V](\nu,\nu), \\
\end{split}
\ee
along $\cC$. On the other hand, by \eqref{Lbarh}, $(\Ric + \d^*V)'_h = F$ and applying the linearization of the Hamiltonian constraint shows 
that \eqref{Hlin22} holds with $V'_h$ in place of $V'$. This gives 
$$\tr(\d^*V')-2\d^* V'(\nu,\nu) = \tr(\d^* V'_h) - 2\d^*  V'_h(\nu,\nu) \mbox{ on }\cC.$$
Since, as noted above, $(V')^{\tT} = (V'_h)^{\tT}$, this implies that 
$$\nu \<V', \nu\> = \nu\<V'_h, \nu\>.$$
Since both $V'$ and $V'_h$ satisfy the same gauge equation \eqref{V'}, it follows from uniqueness of solutions to the mixed Dirichlet-Neumann 
boundary value problem that $V' = V'_h$. This completes the proof of Theorem \ref{modexist} in the case $g \in \cB$ and $\w \tau' \in \cB$. 

   Next, observe that $\cB \subset {\rm Im} \, D\w \Phi^H$ is dense in $\w \cT^H$ in the $C^{\infty}$ semi-norm topology on $\w \cT^H$. To see this, by decomposing the 
relevant tensors into their components, it suffices to show that the space of $C^{\infty}$ functions on a compact domain $D \subset \bR^{n+1}$ for which all $C^{k}$ norms 
 are uniformly bounded by a constant (cf.~\eqref{Mbound}), is a dense subspace of $C^{\infty}(D)$ with the semi-norm topology. However, this is a standard 
 result, which follows easily from the Stone-Weierstrass theorem; all $\w \tau'$ with polynomial component functions are in $\cB$. 
 The same discussion holds for the background metric $g$. 
 
    Finally, Theorem \ref{Eestthm} implies that the range of $D_g\w \Phi^H$ is closed in the $C^{\infty}$ topology of $T_{\w \tau}(\w \cT^H)$. To see this, suppose 
$h_i$ is $C^{\infty}$ and $\w \tau_i = D_g\w \Phi^H(h_i)$. Suppose $\w \tau_i \to \w \tau$ in $C^{\infty}$. Then for each fixed $s$, $||\w \tau_i||_{H^s(\w \cT^H)} \leq K$, 
for some $K$, possibly depending on $s$. The energy estimate \eqref{Eest} then implies that $||h_i||_{\cN^s(\w U)} \leq CK$ is also uniformly bounded for each 
fixed $s$ and hence a subsequence converges in $C^{\infty}$ to a limit $h \in C^{\infty}$. Clearly $\w \tau' = D_g \w \Phi^H(h)$. Since ${\rm Im} \, D_g\w \Phi^H$ 
is both dense and closed, it follows that $D_g\w \Phi^H$ is surjective in $C^{\infty}$, for $g \in \cB$. The same argument proves surjectivity for general $g \in C^{\infty}$. 
Thus, for $g \in C^{\infty}$, choose $g_i \in \cB$ such that $g_i \to g$ in $C^{\infty}$. For each fixed $i$, $D_{g_i}\w \Phi^H$ is surjective, so the equation 
$D_{g_i}\w \Phi^H(h_i) = \w \tau_i'$ is solvable for any $\w \tau_i'$. Given any $\w \tau' \in T_{\w \tau}(\w \cT^H)$, choose $\w \tau_i' \to \w \tau'$ in $C^{\infty}$. 
Applying the energy estimate \eqref{Eest} as before shows that the sequence $\{h_i \}$ 
converges to a limit $h$ in $H^s$ for each fixed $s$.  Thus $h_i \to h$ in $C^{\infty}$ and $D_g \w \Phi^H(h) = \w \tau'$.  

Also, the construction above  is clearly smooth in the data $T(\w \cT^H)$ and in $g$. This completes the proof of Theorem \ref{modexist}.  

\end{proof}

\begin{remark}
{\rm One might consider trying to use the equation \eqref{gaugen} in place of \eqref{Hlin} to determine $f$. This relates the normal derivative 
$\nu(f)$, (or more precisely $\nu(Hf)$, in terms of the other components of $h$. Thus, one might consider using this to determine 
the Neumann derivative of $f$ at $\cC$ and hence determine $f$ in the bulk, in terms of other data. An advantage of this approach would 
be that one does not then need Assumption (*), only $A \neq 0$ is needed (for the equivalence relation) and $H \neq 0$. 
  
   However. this approach cannot be expected to work, since the leading order terms $\nu(h_{\nu \nu})$ and $\d_{\cC}(h(\nu)^{\tT})$ in 
\eqref{gaugen} do not have $\e$-coefficients, in contrast to \eqref{Hlin}, so that the iteration process cannot be expected to converge.  
   }
\end{remark}

The linear solutions constructed in Theorem \ref{modexist} are unique, by the energy estimate \eqref{Eest}.

 \subsection{Construction of linearized solutions: Dirichlet boundary data.} 
 
   We now use Theorem \ref{modexist} to prove the analogous result for Dirichlet boundary data. 
 Let 
 \be \label{tau'2}
\tau' = (F, (\g_S', \k', \nu', V_S'), (\g_{\cC}', V_{\cC}')),
\ee
be a general element in $T_{\tau}(\cT^H)$, $\tau = \Phi^H(g)$.  

\begin{theorem}\label{Direxist}
There exists $\e > 0$ such that, for any smooth vacuum Einstein metric $g$ in harmonic gauge ($V_g = 0$) on $\w U$ which is $C^{\infty}$ close to a 
standard Minkowski corner metric $g_{\a_0}$ as in \eqref{eps} and satisfies Assumption (*), the linearization $D \Phi_g^H$ at $g$ satisfies: given 
any $C^{\infty}$ target data $\tau' \in T_{\tau}(\cT^H)$ on $\w U$, there exists a variation $h \in T_g Met(\w U)$, such that 
\be \label{DPhisol}
D\w \Phi_g^H(h) =  \tau'.
\ee
Thus, the equation 
\be \label{LhF1}
L(h) = F,
\ee
has a $C^{\infty}$ solution $h$ such that 
\be \label{C1}
(g_S, K)'_h = (\g_S', \k'), \ \ (\nu_S)'_h = \nu', \ \ V'_h = V_S' \ \mbox{ on }S,
\ee
and
\be \label{C2}
h^{\tT}  = \g_{\cC}', \ \ V'_h = V_\cC' \ \mbox{ on }\cC.
\ee
Further the constructed solution $h$ depends smoothly on the data $\tau'$. 
\end{theorem}
  
 \begin{proof}

  Any given choice of Dirichlet boundary data $\g_{\cC}'$ on $\cC$ may be decomposed as in \eqref{hA} uniquely as 
\be \label{hT}
\g_{\cC}' = \g'_A + f' A,
\ee
for some $f'$. Also, the derivatives $a_k = \p_t^k(h_{0\nu} + h_{\nu \nu})$, $k = 0, 1, 2,\dots$ at $\Si$ are uniquely determined by the target data $\tau'$, 
through the compatibility conditions at $\Si$ as in Lemma \ref{Compatibility}. By the Borel Lemma, (cf.\cite[Thm.~1.2.6]{Ho} for example), there exists a 
smooth function $\eta'$ on $\cC$ such that $\p_t^k \eta' = a_k$. Thus, choose a fixed extension operator from such corner data to $C^{\infty}(\cC)$ for 
which $||\eta'||_{C^k(\cC)} \leq c_k \sum_{j=0}^k ||a_j||_{C^0}$. 

   Now by Theorem \ref{modexist}, for any $\w \tau'$ as in \eqref{tau'1}, there exists a solution $\w h$ of \eqref{DwPhisol}, i.e. 
$$D\w \Phi^H(\w h) = \w \tau'.$$
Recall that target data $\w \tau'$ and $\tau'$ differ only in their boundary data for the metric. 
Choose $\w \tau'$ equal to $\tau'$, but with $\g_{\cC}'$ replaced by $(\g'_A, \eta')$, so that 
 $$\w h^T = \g'_A + \w f A, \ \  \w h_{0\nu} +  \w h_{\nu \nu} = \eta'.$$
     
 Choose the vector field $X$ on $\w U$ such that 
 \be \label{Xgauge}
V'_{\d^*X} = 0  \ {\rm on} \ \w U,  \ \ X|_{\cC} = (f' - \w f)\nu ,
\ee
and with $X = \p_t X = 0$ on $S$. We note that the compatibility condition holds: 
\be\label{comf}
\p^k_t(f'-\w f)=0 \ \mbox{ for all } k=0,1,2,...\  \mbox{ on }\Si.
\ee 
For $k=0,1$, this follows directly from Lemma \ref{fft}. The derivative $\p^k_t \w f ~(t\geq 2)$ is determined by the $\p_t^{k-2}$ derivative of the Hamiltonian constraint 
equation \eqref{Hamlin2} at the corner. On the other hand, since $\tau'$ is chosen to satisfy compatibility conditions of all orders, it follows that $f'$ also satisfies the 
Hamiltonian constraint equation and any order $\p_t$ derivatives of it along the corner. This gives \eqref{comf}.  

    Setting 
$$h = \w h + \d^*X,$$
one then has 
\be\label{error}
D\Phi^H(h) = \tau' +(L(\d^* X), 0,...,0) = \tau' +(\cL_X \Ric_g+(\d^*)'_{\d^* X}V_g, 0,...,0).
\ee
Since $g$ is vacuum and in harmonic gauge, $\Ric_g = V_g = 0$, which completes the proof. 
 \end{proof}

   The proof above does not directly generalize to non-vacuum metrics, due to the error term $L(\d^*X) = \cL_X \Ric_g+(\d^*)'_{\d^* X}V_g$ in \eqref{error}. 
This will be small when $g$ is near-vacuum, which is sufficient to apply the Nash-Moser theory as discussed in \S 6.

 \section{Linear global existence and energy estimates.}
 
  In this section, we prove the global versions of the results of \S 3 and \S 4, both for the shifted boundary data and for Dirichlet boundary data spaces. 
This is done by patching together local solutions to obtain global in space solutions. 

  The first main result proves existence and uniqueness for the shifted boundary data in the linearized setting. 

\begin{theorem}\label{modglobexist}
Let $(M, g)$ be a smooth globally hyperbolic Lorentzian manifold, $M \simeq I \times S$, with compact Cauchy surface $S$ having compact boundary 
$\p S = \Si$ and with timelike boundary $(\cC, g_{\cC})$. Suppose $g$ satisfies the convexity Assumption $(*)$ at $\cC$. Then for any smooth 
$\w \tau' \in T_{\w \tau}(\w \cT^H)$, the equation 
\be \label{wL3}
D\w \Phi_g^H(h) = \w \tau'
\ee
has a unique $C^{\infty}$ solution $h$. The constructed solution $h$ depends smoothly on the data $\w \tau'$. 

Moreover, any $h$ satisfying \eqref{wL3} satisfies the tame energy estimate 
\be \label{Eest2}
||h||_{\cN^s(M)}  \leq C\l^{-3/2} ||D\w \Phi_g^H(h)||_{H^s(\w \cT^H)}. 
\ee
\end{theorem} 
 
 \begin{proof}

  Given the previous local results, the proof of Theorem \ref{modglobexist} is essentially standard. We provide the details below for completeness. 
    
  Given $(M, g)$ as in Theorem \ref{modglobexist}, first choose an open cover $\{U_i\}_{i=1}^N$ of the corner $\Si$, where each $U_i$ is sufficiently 
small so that the local existence result Theorem \ref{modexist} holds for each $U_i$. Also, choose an open set $U_0$ in the interior $M\setminus \cC$ 
such that $(\cup_{i=1}^N U_i)\cup U_0$ also covers a tubular neighborhood of the initial surface, i.e. $S\times [0,t_0]$ for some time $t_0>0$.

Let $\{\w U_i\}_{i=0}^N$ be a thickening of the cover $\{U_i\}_{i=0}^N$, so that $U_i \subset \w U_i$, as described in \S 2.4. Let 
$\rho_i$ be a partition of unity subordinate to the cover $\{\w U_i\}_{i=0}^N$, so that ${\rm supp}\rho_i\subset \w U_i$, $\rho_i = 1$ on 
$U_i$ and  $\sum_i \rho_i=1$. Let 
$$\w \tau' =\big(F, (\g', \k', \nu', V_S')_S, (h_A, \eta,V_{\cC}')_{\cC} )\big),$$
be arbitrary $C^{\infty}$ data in $T(\w \cT^H)$ on $M$. Then the data $\rho_i \w \tau'$ has compact support in the sense of \S 2.1 in $\w U_i$ 
and by Theorem \ref{modexist} there exists a solution $h_i$ in $\w U_i$ satisfying 
\be \label{hitau}
D\w \Phi^H_g(h_i) = \rho_i \w \tau'. 
\ee
Here $h_i$ and $\w \tau'$ are rescaled (by the relevant factors of $\l$) back to the original coordinates and metric on $M$. For simplicity, we do not 
change the notation to reflect this. Moreover, as noted at the end of \S 2.4, by the finite propagation 
speed property, the solution $h_i$ has compact support in the sense of \S 2.4 for $t \in [0,t_i]$ for some $t_i > 0$. Thus, $h_i$ extends smoothly as the 
zero solution on $M_{t_i} \setminus \w U_i$. 

   Similarly, in the interior region $U_0 \subset \w U_0$, let $h_0$ be the solution to the Cauchy problem 
$$L(h_0) = \rho_0 F \ \mbox{ in } \w U_0, \ \ (g_S, K, \nu_S, V_g)'_h = \rho_0(\g', \k', \nu', V_S') \ \mbox{ on }S_0.$$
Then again $h_0$ has compact support in $\w U_0$ for some time interval $[0, t_0]$ and as above extends to $M_{t_0}$. We relabel so that 
$t_0$ is a common time interval for all $h_i$, $i \geq 0$. 

  The sum 
\be \label{sum}
h = \sum_{i=0}^N h_i,
\ee
is thus well-defined on $M_{t_0}$ and by linearity
\be \label{htau}
D\w \Phi_g^H(h) = \w \tau'.
\ee
By the local energy estimate in Theorem \ref{Eestthm}, the solution $h_i$ of \eqref{hitau} is 
unique, and hence $h$ is the unique smooth solution \eqref{htau}, since if $h'$ is any other solution of \eqref{htau}, one may apply the local decomposition 
\eqref{sum} in the same way. Moreover, the partition of unity property shows that \eqref{Eest} gives \eqref{Eest2}. 

  We note that the various $H^s$ norms in \eqref{Eest} scale in different ways with $\l$, but the factor $\l^{-3/2}$ is the dominant factor. 
Compared with \eqref{Eest}, the extra factor of $\l^{-1}$ comes from the difference in the scaling of $H^s$ and $H^{s+1}$ norms on the left and right in 
\eqref{tnorm}.  The smoothness statement follows from that in Theorem \ref{modexist}. 

\end{proof} 

  Next we pass from the shifted boundary data $(h_A, \eta)$ to the case of interest, Dirichlet boundary data. The analog of Theorem \ref{modglobexist} 
 also holds in this setting, with a slight further loss of derivative in the target space. As in Theorem \ref{Direxist}, we restrict to vacuum Einstein metrics $g$. 
  
  \begin{theorem}\label{Dirglobexist} 
  Let $(M, g)$ be a smooth globally hyperbolic vacuum Einstein metric in harmonic gauge, $M \simeq I \times S$, with compact Cauchy surface $S$ having 
compact boundary $\p S = \Si$ and with timelike boundary $(\cC, g_{\cC})$. Suppose $g$ satisfies the convexity Assumption $(*)$ at $\cC$. Then for any 
smooth $\tau' \in T_{\tau}(\cT^H)$, the equation 
\be \label{DDir}
D \Phi_g^H(h) = \tau'
\ee
has a unique $C^{\infty}$ solution $h$. The constructed solution $h$ depends smoothly on the data $\tau'$. 

Moreover, any $h$ satisfying \eqref{DDir} satisfies the tame energy estimate 
\be \label{DirEest}
||h||_{\cN^s(M)}  \leq C\l^{-3/2} ||D\Phi_g^H(h)||_{H^{s+1}(\cT^H)}. 
\ee
  \end{theorem} 

 \begin{proof}

    Theorem \ref{modglobexist} gives the existence and tame energy estimates for the shifted linearized boundary data $([h]_A, \eta)$.  We show 
how this gives existence and tame energy estimates for the linearized Dirichlet boundary data $h^{\tT}$ in a manner analogous to the passage from 
Theorem \ref{modexist} to Theorem \ref{Direxist}.  

 As in the proof of Theorem \ref{Direxist}, given $\tau' \in T(\cT^H)$, let $\w \tau'$ equal $\tau'$ in all terms, except the boundary data for $\w \tau'$ is 
given by $(\g'_A, \eta')$ in place of $\g'$ from $\tau'$, where $\g' = \g'_A + f'A$ and $\eta'$ is uniquely constructed from the corner information of $\tau'$, 
as in the proof of Theorem \ref{Direxist}.  
 
 Let $\w h$ be the unique solution to $D\w \Phi_g^H(\w h) = \w \tau'$ given by Theorem \ref{modglobexist}. The estimate \eqref{Eest2} then gives 
$$||\w h||_{\cN^s(M)} \leq C\l^{-3/2}||\w\tau'||_{H^{s}(\w T^H)} \leq C\l^{-3/2}||\tau'||_{H^{s}( T^H)}.$$
A solution to \eqref{DDir} can then be constructed by setting $h=\w h+\d^* X$, where $X$ is given by the solution of \eqref{Xgauge}. As before, the 
Dirichlet boundary data of $X$ on $\cC$ is given by $(f'-\w f)\nu$, where $f'$ is part of the data $\tau'$ and $\w f$ is part of the solution $\w h$. By 
the Dirichlet energy estimate \eqref{DirE}, 
\bes
||X||_{\cN^{s+1}(M)}\leq C(||f'||_{H^{s+1}(\cC)}+||\w f||_{H^{s+1}(\cC)})\leq C||\tau'||_{H^{s+1}( T^H)}.
\ees
For the last inequality, we have applied the observation from the proof in Proposition \ref{prophnua} that the Dirichlet data $||\w f||_{H^{s+1}(\cC)}$ is 
bounded by $\w\tau'$ (and also by $\tau'$). Thus we have
\bes
||h||_{\cN^s(M)} \leq C(||\w h||_{\cN^s(M)} + ||X||_{\cN^{s+1}(M)})\leq C\l^{-3/2}||\tau'||_{H^{s+1}( T^H)}.
\ees

It remains to show uniqueness of solutions to \eqref{DDir}, so that any solution must equal to the solution constructed in Theorem \ref{Direxist} and hence 
satisfies the energy estimate above. Suppose $\hat h$ solves
\bes
D\Phi^H_g(\hat h)=0.
\ees
Let $Y$ be the vector field such that 
\be\label{eqtnY}
V'_{\d^* Y}=0 \ \mbox{ in }U, \ \ Y=\p_t Y=0 \mbox{ on }S, \ \ Y^\intercal =0, \ \ \eta_{\hat h+\d^* Y}=0 \ \mbox{ on }\cC. 
\ee
It is easy to check that this system is well-posed and admits a unique smooth solution. Furthermore, the tensor $\hat h+\d^* Y$ is now a zero solution 
for $D\w\Phi_g^H$, i.e. $D\w\Phi_g^H(\hat h+\d^*Y)=0$. It then follows from the energy estimate \eqref{Eest2} that $\hat h=-\d^* Y$. Since $h$ has zero 
Dirichlet data, we also have $(\d^* Y)^{\tT} =-h^{\tT}=0$, which further implies (using Assumption $(*)$) that the normal component $Y^1=\<Y,\nu_\cC\>$ of 
$Y$ must vanish. Combining this with \eqref{eqtnY}, we obtain $Y=0$ in $U$ and hence $\hat h=0$. This completes the proof of uniqueness.

\end{proof}
 
\begin{remark} \label{toe}
{\rm The time of existence $t^*$ of the linear solutions and energy estimates in Theorems \ref{modglobexist} and \ref{Dirglobexist} is of order $\l \sim \e$, 
\be \label{te}
t^* = O(\l),
\ee
where we recall $\l$ must be chosen sufficiently small to allow for the absorption of $\l$ terms as in \S 4. Of course $\l = \l(g)$.  
The reason for \eqref{te} is that the linear energy estimate \eqref{Eest} in the localized setting $\w g$ holds for a time interval $\w t^*$ of 
order $1$. Upon rescaling from $\w g$ back to $g$, this gives \eqref{te}. To extend the time $t^*$ to much longer times requires at least 
an understanding of the global-in-time stability behavior of the linearized solutions, which is not addressed here. 

  In the bound \eqref{DirEest}, both $\l$ and $C$ depend smoothly on $g$, although they depend on different aspects of $g$. For the nonlinear 
theory discussed next in \S 7, they are combined giving a single smooth function of $g$. 

}
\end{remark}

\section{Nash-Moser theory}
   
   The analysis in \S 6 shows that the linearized IBVP with Dirichlet boundary data is locally-in-time well-posed and has tame energy estimates on-shell, 
if one allows for a loss of derivatives for the boundary data. This is the setting of the Nash-Moser implicit function theorem, which we now apply to 
prove Theorem \ref{mainthm}. 

  Before starting the proof, we need the following result. Let 
$$\cT^H_0 = \{ (0, (\g, \k), \nu, 0, \g_{\cC}, 0) \} \subset \cT^H,$$
so $Q$ and the initial and boundary data of the gauge field in $\cT^H$ are set to zero. 

\begin{lemma} \label{T0}
$\cT^H_0$ is a smooth submanifold of $\cT^H$. 
\end{lemma} 

\begin{proof}
The proof is similar to that of Proposition \ref{manifold}. We work with $H^s$ spaces of finite differentiability but will neglect this in the notation. 
Consider then the map 
$$Z: \cT^H \to Sym^2(M) \times \bV(S) \times \bV(\cC),$$
\be \label{Zsub} 
\begin{split}
&Z(Q, \g_S, \k, V_S, \nu_S, \g_{\cC}, V_{\cC}) = 
\begin{cases}
Q \\
V_S \\
V_{\cC}.\\
\end{cases}
\end{split}
\ee
As before, it suffices to prove that $Z$ is a submersion, i.e.~the derivative $DZ$ is surjective, on $\cT_0^H$. Recall from Proposition \ref{manifold} that 
$$T(\cT^H) = Ker DC,$$
where $C$ is the constraint map in \eqref{Const}. To see that arbitrary $(Q', V'_S, V'_{\cC})$ may be chosen within $Ker DC$, it suffices then to show that 
the usual, non-gauged constraint operator 
\be \label{Const2}
\begin{split}
& C(\g_S, \k) = 
\begin{cases}
|\k|^2 - (\tr_{\g_S}\k)^2  - R_{\g_S} \\
{\rm div_{\g_S}}  \k - d_S(\tr_{\g_S}\k)  \\
\end{cases}
\end{split}
\ee
is a submersion. Now if $DC$ in \eqref{Const2} is not surjective, choose a vector field $X$ which is $L^2(S)$ orthogonal to ${\rm Im} \, DC$. Integrating by parts gives then 
$$0 = \int_S\< DC(\g_S', \k'), X\> = \int_S \<(\g_S',\k'), DC^*X\> + \int_{\Si} \b(\g_S',\k')(X),$$
where the boundary term $\b(\g_S',\k')$ consists of zero and first order differential operators on $X$ at $\Si$, with coefficents depending on $(\g_S',\k')$. 
Since $(\g',\k')$ may be chosen arbitrarily, this implies first that $DC^*X = 0$ on $S$. This is an overdetermined elliptic operator on $X$; writing $X = 
X_S + u\nu_S$, $DC^*X$ is first order in $X_S$ and second order in $u$. A well-known result of Moncrief \cite{Mon} implies that 
$X$ is a Killing field of the solution of the Cauchy problem generated by $(\g, \k)$. Moreover, since $(\g_S', \k')$ are arbitrary at $\Si$, it is 
straightforward to check by inspection that, $X_S = 0$ and $u = \p_1 u = 0$ on $\Si$. It then follows from well-known unique continuation results, cf.~\cite{Ar}, 
that $X = 0$ on $S$, which completes the proof. 

\end{proof} 

  We are now in position to complete the proof of Theorem \ref{mainthm}. All data below are assumed to be $C^{\infty}$; for basic details regarding Frechet manifolds, 
tame maps and the Nash-Moser theorem we refer to \cite{Ham}, \cite{Z2}. 

\begin{proof}

Form the map 
$$\Psi: Met^*(M)\times \cT_0^H \to \cT^H ,$$
\be
\Psi(g, \tau) = (\Ric_g + \d_g^* V, (g_S, K_S, \nu_S, V|_S), g_{\cC}, V|_{\cC}) - \tau, 
\ee
$\tau \in \cT_0^H$. Here $Met^*(M) \subset Met(M)$ is the space of metrics defined on $[0, t^*)\times S$, where $t^*$ is the time-of-existence function 
given as in Remark \ref{toe}. 

By Proposition \ref{manifold} and Lemma \ref{T0}, the domain and target spaces of $\Psi$ are smooth tame Frechet manifolds and it is easy to verify that 
$\Psi$ is a smooth tame mapping between tame Frechet manifolds. Let $\pi: Met^*(M) \times \cT_0^H \to Met^*(M)$ be the projection onto the first factor 
and let 
$$\bZ = \pi(\Psi^{-1}({\bf 0})).$$
This is the space of Lorentz metrics $g$ on $M = [0, t^*)\times S$ with $Q  = 0$ on $M$ and $V = 0$ on $S \cup \cC$. Since then $\Ric_g + \d^*V = 0$, 
it follows from Remark \ref{gaugerem} that $V = 0$ on $M$ and thus 
$$\bZ = \bE_0$$
is the space of smooth vacuum Einstein metrics in harmonic gauge, $V = 0$ on $M$. 

   Let $\bZ^* = \bE_0^*$ be the open subset of metrics satisfying Assumption $(*)$. Theorem \ref{Direxist} shows that $D_1 \Psi = D\Phi^H$ is surjective 
on $\bE_0^*$, while Theorem \ref{Dirglobexist} shows that $D_1 \Psi$ is injective on $\bE_0^*$. The energy estimate \eqref{DirEest} shows that 
$D_1\Psi$ has a tame inverse and hence $D_1\Psi$ is a tame isomorphism, i.e.~$D_1 \Psi$ is a tame isomorphism on-shell. 

  It is well-known however that this is not yet sufficient to apply the Nash-Moser theorem. One needs in addition an approximate inverse for $D_1 \Psi$ for any 
metric $g$ close to a metric in $\bE_0^*$. 

    Let then $g^0$ be any given vacuum Einstein metric $g^0$ in harmonic gauge, so $V_{g^0} = 0$ on $M$ as above and satisfying Assumption $(*)$ at $\cC$. Let 
 $\tau^0 = \Phi^H(g^0)$. Consider the linearization $D\Phi_g^H$ for $g$ in a neighborhood $\cU$ of $g^0$ in $Met^*(M)$, with $g$ satisfying Assumption $(*)$. 
 Given an arbitrary $\tau' \in T(\cT^H)$, recall that \eqref{error} gives the existence of $h \in T_g(Met^*(M))$ and a vector field $X$ with $V'_{\d^*X} = 0$, such that 
\be\label{error2}
D\Phi_g^H(h) = \tau' +(L(\d^* X), 0,...,0) = \tau' +(\cL_X \Ric_g+(\d^*)'_{\d^* X}V_g, 0,...,0).
\ee
Recall from its definition in \eqref{Xgauge} that $X$ is uniquely determined by $\tau'$. This defines the map 
$$\bar I:T(\cT^H) \to T(Met^*(M)$$
$$\bar I(\tau') = h.$$
The energy estimates from Theorem \ref{modglobexist} and \ref{Dirglobexist} imply that $\bar I$ is a smooth, tame Fredholm map. As in \eqref{Xgauge}, 
$X$ vanishes to first order on $S$ and has Dirichlet boundary data determined by $\tau'$. Hence by the energy estimate \eqref{DirE},  
\be \label{Xbnd}
||X||_{\cN^s} \leq C ||\tau'||_{H^s(\cT^H)},
\ee
for a fixed constant $C = C(g)$, $g \in \cU$. Further, since $V'_{\d^*X} = 0$, $L(\d^*X) = \cL_X \Ric_g+(\d^*)'_{\d^* X}V_g = \cL_X(\Ric_g + \d^*V_g) = 
\cL_X Q - \d^*\cL_X V$. We then have  
$$||\cL_X Q||_{\cN^s} \leq C||\cL_X(\Psi(g, \tau))||_{H^s(\cT^H)} \leq C||\Psi(g, \tau)||_{H^{s+1}(\cT^H)} ||X||_{\cN^{s+1}},$$
and also
$$||\d^* \cL_X V||_{\cN^s}  \leq C||\Psi(g, \tau)||_{H^{s+2}(\cT^H)} ||X||_{\cN^{s+2}},$$
By construction, $\bar I$ then satisfies the estimate   
\be \label{quaderror}
||D\Phi_g^H\circ \bar I(\tau')-\tau'||_{H^s(\cT^H)} \leq C||\Psi(g, \tau)||_{H^{s+2}(\cT^H)} ||\tau'||_{H^{s+3}(\cT^H)},
\ee
Similarly, 
$$||\bar I \circ D\Phi_g^H(h) - h||_{\cN^s} \leq C||\Psi(g, \tau)||_{H^{s+2}(\cT^H)}  ||h||_{\cN^{s+3}}.$$
  
   Thus, $\bar I$ is an approximate inverse with quadratic error. The Nash-Moser implicit function theorem \cite[Thms.~3.3.1, 3.3.3]{Ham}, or \cite[Thm.~6.3.3]{Z2}, 
then gives the existence of a neighborhood $\cW$ of $\tau_0 \in \cT_0^H$ and a smooth tame map $B:  \cW \to Met^*(M)$, such that 
$$\Psi(B(\tau), \tau) = 0,$$
for any $\tau \in \cW$. This exhibits a neighborhood of $g_0$ in $\bE_0^*$ as the graph of $B$ over $\cW$. In particular, this shows that $\bE_0^*$ is a smooth
Frechet manifold and 
$$\Phi^H: \bE_0^* \to \cT_0^H,$$
is everywhere a smooth, tame local diffeomorphism between Frechet manifolds. 

  Next, as noted in Remark \ref{gaugerem}, any metric in $\bE^*$ can be brought into harmonic gauge by a unique diffeomorphism $\f \in \Diff_1(M)$. 
The action of the group $\Diff_1(M)$ on $\bE^*$ is free, smooth and tame and hence the quotient $\hat \cE^* = \bE^* / \Diff_1(M)$ is again a smooth 
tame Frechet manifold, cf.~\cite{Ham}. On $\hat \bE$, the map $\Phi^H$ agrees with the map $\hat \Phi$ from \eqref{hatPhi} so 
that 
$$\hat \Phi: \hat \cE^* \to [\cI_0 \times Met(\cC) \times \bV']_c,$$
is again a local diffeomorphism between smooth Frechet manifolds. Thus for any given data $\tau \in [\cI_0 \times Met(\cC) \times \bV']_c$, there 
is a solution $g \in \cE^*$. As in Remark \ref{toe}, the solution $g$ exists for a time interval $t^*$ depending on the data $\tau$ and $\hat \Phi(g) = 
\tau$ when the boundary data in $\tau$ is suitably restricted. 
     
       Lastly, the action of the gauge group ${\rm Diff}_0(M)$ on $\hat \cE^*$ is free, smooth and tame and hence the quotient $\cE^* = 
\hat \cE / {\rm Diff_0(M)}$ is again a smooth tame Frechet manifold. The map $\hat \Phi$ descends to the quotient and it follows that the induced map 
\be \label{lastPhi}
\Phi: \cE^* \to \cI_0 \times_c Met(\cC),
\ee
is everywhere a tame local diffeomorphism between smooth Frechet manifolds. 

  To conclude the proof, we show that $\Phi$ in \eqref{lastPhi} is not only locally one-to-one but in fact globally one-to-one. To see this, suppose $g_1, g_2$ are two vacuum 
Einstein metrics with the same gauged target data, i.e.~$\Phi^H(g_1) = \Phi^H(g_2)$. If $g_1 \neq g_2$, then there exists an arbitrarily small neighborhood 
$U$ of some $p \in \Si$ such that $g_1 \neq g_2$ in $U$. One may then perform a localization or rescaling of each $g_i$ as in \S 2.4 so that the almost 
constant coefficient metrics $\w g_i$ satisfy $\w g_1 \neq \w g_2$ in $U$. However, by \eqref{eps}, the metrics $\w g_i$ satisfy 
$$|| \w g_2 - \w g_1||_{C^{\infty}} \leq \e,$$
for any prescribed $\e > 0$. It then follows from the local uniqueness iabove that 
$$\w g_2 = \w g_1$$
in some $U' \subset U$. It then follows that $g_2 = g_1$ in $U'$, which proves the global uniqueness. This completes the proof of Theorem \ref{mainthm}.

\end{proof}

\begin{remark} \label{findiff} 
{\rm We have phrased all of the discussion above in the $C^{\infty}$ context. However, Theorem \ref{mainthm} (and all other results) may be extended 
to Sobolev spaces $H^s$ of finite differentiability. Thus, as discussed in \S 2, the maps $\Phi$, $\hat \Phi$ and $\Phi^H$ are well-defined smooth maps 
between a scale of Banach manifolds, 
$$\Phi^H: Met^s(M) \to (\cT^H)^s,$$
indexed by $s$. The existence in Theorem \ref{Direxist} and energy estimates in Theorem \ref{DirEest} hold also for suitable choices of $s$ and 
for suitable choices of $H^s$ smoothness of the background metric $g$. One may then apply the finite derivative version of the Nash-Moser inverse 
function theorem due to Zehnder \cite{Z1}, \cite[Theorem 6.3.3]{Z2};  in the notation used there, the loss of derivative is $\g = 4$. 
However, we will not carry out here the details that the hypotheses of \cite[Theorem 6.3.3]{Z2}, do hold. 

}
\end{remark}

 \section{Concluding Remarks}

  We conclude this work with several remarks concerning the breakdown of the well-posedness in Theorem \ref{mainthm} when Assumption $(*)$ is 
 dropped. First however we discuss some of the difficulties that occur even in the simple case of $2+1$ dimensions. 
 
   To begin, recall from the discussion in \S 1 that a vacuum metric $g$ is Ricci-flat and so flat in $2+1$ dimensions. Assuming the domain $M \simeq I\times S$ 
is also simply connected, the developing map gives an immersion 
\be \label{emb}
(M, g) \to \bR^{1,2}.
\ee
It is easy to see that $F$ is an embedding, so $M$ may be viewed as a domain in Minkowski space $\bR^{1,2}$ with timelike boundary $\cC$. The Dirichlet 
IBVP in $2+1$ dimensions is thus closely related to the problem of existence and uniqueness of isometric embedding of timelike surfaces $(\cC, \g)$ into 
$\bR^{1,2}$. 

\begin{remark}
{\rm 
There is a major difference in the behavior of the extrinsic curvature between surfaces in $\bR^3$ and timelike surfaces in $\bR^{1,2}$. For a surface in $\bR^3$, 
the second fundamental form operator $A$ (aka Weingarten map) is self-adjoint with respect to the induced Riemannian metric $\g$ on the surface, and so 
$A$ may be diagonalized by an orthonormal basis with respect to $\g$, giving the principal curvatures.  Consequently, for instance $K > 0$ implies 
$H = \tr_{\g}A > 0$ (with respect to the usual orientation). 

  However, this is not the case for Lorentz metrics; $A$ may or may not be diagonalizable with respect to $\g$, cf.~\cite{Lop} 
for a general discussion. For instance, there are examples of Lorentz metrics on $\cC \subset \bR^{1,2}$ with $K_{\g} > 0$ but $H = 0$ 
(timelike minimal surfaces). 
}
\end{remark}

   We note the following: 
 
 \begin{lemma} \label{21}
In $2+1$ dimensions, the Assumption $(*)$ implies that the Gauss curvature $K_{g_{\cC}}$ is positive:  
\be \label{K>0}
K_{g_{\cC}} > 0.
\ee
\end{lemma}

\begin{proof}
Choose an orthonormal basis $\{e_{a}\}$ for $g_{\cC}$. The matrix of $g_{\cC}$ in this basis is the identity $I$ so the matrix $\Pi_{g_{\cC}}$ of $\Pi$ in 
this basis is $HI - A_{g_{\cC}}$, where $A_{g_{\cC}}$ the matrix of $A$. Then 
$$det \Pi_{g_{\cC}}  = det(HI - A_{g_{\cC}} ) = H^2 - (\tr A_{g_{\cC}})H + det A_{g_{\cC}} = det A_{g_{\cC}}.$$
The Assumption $(*)$ implies the matrix $\Pi_{g_{\cC}}$ is positive definite, so that $det \Pi_{g_{\cC}} > 0$. The result then follows from Gauss' Theorema 
Egregium.

\end{proof} 

  Note that the same proof using the Gauss equation (Hamiltonian constraint) in place of the Theorema Egregium gives the following. For 
$$Ric - \frac{R}{2}g + \Lambda g = 0$$
with $\Lambda = -1$ or $\Lambda = +1$, (anti-deSitter or de Sitter spacetimes), Assumption $(*)$ implies 
$$K_{g_{\cC}} > -1, \ {\rm or} \  K_{g_{\cC}} > 1,$$
respectively. 

  However, the converse of Lemma \ref{21}, i.e.~whether Assumption $(*)$ follows from \eqref{K>0}, does not hold in general. There are many examples 
of timelike surfaces $(\cC, \g)$ in $\bR^{1,2}$ with $K_{\g} > 0$ but where Assumption $(*)$ fails, cf.~again \cite{Lop} for example. 

  If $\Pi$ is diagonalizable with respect to $g_{\cC}$, then it is easy to see as in Lemma \ref{21} that $K_{g_{\cC}} > 0$ implies Assumption $(*)$. The condition 
$$H^2 > 4K_{g_{\cC}},$$
ensures diagonalizability, but this condition remains extrinsic to the boundary geometry. Thus even in $2+1$ dimensions, it is not clear whether Assumption $(*)$ 
can be effectively related to an intrinsic condition on the boundary metric. 

\medskip 

\begin{remark} \label{flip}
{\rm We include here a remark regarding the orientation discussion in Lemma \ref{alpha}. Let $\cC = I\times S^1$ be an embedded timelike surface in $\bR^{1,2}$ 
bounding $M \simeq I\times S$, where $S$ is a spacelike disc. (The discussion below applies to all dimensions). Consider the Cauchy surface $S_0 = M\cap \{t = 0\}$ in 
$\bR^{1,2}$. Since $S_0$ is totally geodesic, $K_{S_0} = 0$, so that $\tr_{g_{S_0}} K_{S_0} = 0$ and hence $\a$ is uniquely determined as in \eqref{vel}. 

  Suppose now the boundary $\cC$ meets $S_0$ at the corner orthogonally so $\a = 0$. Keeping the boundary fixed, consider a distinct Cauchy surface $S^+$ for $M$, 
with $\p S^+ = \p S_0 = \Si$, with $\a_{S^+}$ given arbitrarily. Let $M^+$ be the domain (i.e.~the vacuum solution) to the future of $S^+$. The reflection of $M^+$ 
through the plane $\{t = 0\}$ gives another domain $M^-$ to the past of the Cauchy surface $S^-$. The domains $M^+$ and $M^-$ are isometric via the isometric 
reflection through $\{t = 0\}$ and have the same induced initial and boundary data, but $\a_{S^-} = -\a_{S^+}$.  

  We also note that Theorem \ref{mainthm} and the other results above hold equally well with respect to the `exterior' boundary value problem, where for instance 
the Cauchy surface $S$ is the complement of a bounded domain in $\bR^n$. In fact the decomposition \eqref{equiv} is invariant under the sign or orientation 
change $(f, A) \to (-f, -A)$ and it is easy to verify that the proof of Theorem \ref{mainthm} remains the same under such a change. 

}
\end{remark}

  Next we discuss some examples of the breakdown of well-posedness when Assumption $(*)$ is dropped. 
  
   \smallskip 

$\bullet$ Let $(M, g)$ be a vacuum Einstein metric with $A = 0$ on some (small) open set $U \subset \cC$ with $U \cap \Si = \emptyset$. Let $f: U \to \bR$ be 
any function of compact support in $U$ and let $X$ be a vector field on $M$ with $X = f\nu_{\cC}$ on $\cC$ and with (small) compact support, disjoint from $S$, 
near $U$ in $M$. Then the infinitesimal vacuum deformations $h = \d^*X$ satisfy $h^{\tT} = fA = 0$ on $U$. Since the Cauchy data on $S$ are unchanged by such 
deformations $h$, this shows that $Ker D\Phi$ is infinite-dimensional at such $(M, g)$. Thus there is an infinite-dimensional indeterminacy or non-uniqueness at 
the linear level. 

  A related non-uniqueness also holds at the non-linear level. Namely, suppose $f < 0$ above and for $s > 0$ small, let $G_{sf} = \{\exp_x (sf(x)\nu_{\cC}): x \in U\} \subset M$ 
 be the graph of $sf: U \to \bR$ by the exponential map normal to $\cC$ in $(M, g)$. Let $V_{sf}$ be the open set in $M$ bounded by $U$ and $G_{sf}$ and let $(M^-,g^-) = 
 M \setminus V_{sf}$. Thus $\cC$ has been pushed inward by a small bump. Now let $(M^+, g^+)$ be the reflection of $(M^-, g^-)$ through the wall 
 $U \subset \cC$. Then the timelike boundaries $\cC^{\pm}$ of $M^{\pm}$ are isometric, i.e.~have the same Dirichlet boundary data and also have the same 
 Cauchy data on $S$. However, $M^{\pm}$ are not isometric. This shows a breakdown of uniqueness and so well-posedness for a large class of boundary data. 
 
    The discussion above is an exact analog of the same phenomena that occur for the isometric embedding problem for surfaces in $\bR^3$, cf.~\cite{Sp}. 
We expect that the complicated behaviors possible in the isometric embedding of surfaces in $\bR^3$ lead to similar complications for the 
Dirichlet IBVP in three and higher dimensions. 

\smallskip 
 
$\bullet$ Next we discuss the more non-trivial examples presented in \cite{BS} and \cite{AGAM}. Consider the (flat) Rindler vacuum spacetime
$$g_R = -z^2 dt^2 + z^2 + g_{\bR^2},$$
where $g_{\bR^2}$ is the Euclidean metric on $\bR^2$. Choose the timelike surface $\cC = \{z = z_c\}$ for some fixed $z_c > 0$ and Cauchy slice 
$S = \{t = 0\}$. For this background solution, both an explicit obstruction to linearized existence and an explicit infinite dimensional space of 
solutions in $Ker D\Phi_{g_R}$ are identified in \cite{AGAM}. To see that Assumption $(*)$ does not hold, we have $A = 
\frac{1}{2}\cL_{\p_z}g_R = -z_c dt^2 \neq 0$. However, it is easy to see that the Brown-York stress tensor $\Pi$ is degenerate, so Assumption $(*)$ fails. 
Similar remarks hold for the near-horizon (Schwarzschild-type) solutions discussed in \cite{BS}, \cite{AGAM}.

\section{Appendix} 

  In this Appendix, we provide the detailed proofs of several statements not yet proved in the main text. 
  
    We begin with the proof of \eqref{com2nu1}, i.e.
$$\p_1^2\nu_S^1=\xi_2^1(\p_0 V_1,\p_1 V_0,\p_S\p_\Si \nu_S, \p_S K, \p_S^2 g, \p g).$$

\begin{proof}
First recall \eqref{Vgf}, i.e.   
\be\label{Vgf2}
V_\a:=g_{\a\b}V^\b=-g^{\rho\s}\p_\rho g_{\s\a}+\tfrac{1}{2}g^{\rho\s}\p_\a g_{\rho\s}
\ee
Setting $\a=1$ and taking the $\p_0$ derivative gives  
\bes
\begin{split}
\p_0V_1&=-g^{\rho\s}\p_0\p_\rho g_{\s1}+\tfrac{1}{2}g^{\rho\s}\p_0\p_1 g_{\rho\s}+O_1(g)\\
&=-g^{00}\p_0\p_0 g_{01}-g^{11}\p_0\p_1 g_{11}-g^{AB}\p_0\p_A g_{B1}+\tfrac{1}{2}g^{00}\p_0\p_1 g_{00}+\tfrac{1}{2}g^{11}\p_0\p_1 g_{11}+\tfrac{1}{2}g^{AB}\p_0\p_1 g_{AB}+O_1(g)\\
&=\p_0\p_0 g_{01}-\tfrac{1}{2}\p_0\p_1 g_{00}-\tfrac{1}{2}\p_0\p_1 g_{11}-g^{AB}\p_0\p_A g_{B1}+\tfrac{1}{2}g^{AB}\p_0\p_1 g_{AB}+O_1(g)\\
\end{split}
\ees
Since $2K=\cL_{\nu_S}g$ by definition, we have 
\bes
2K_{ij}=\nu_S^\a\p_\a g_{ij}+\p_i \nu_S^\a g_{\a j}+\p_j \nu_S^\a g_{\a i}.
\ees
It further implies that along the corner:
\bes
\p_0\p_1 g_{11}=2\p_1 K_{11}-2\p_1^2\nu_S^1+O_1(g), 
\ees
as well as $\p_0\p_A g_{B1}=2\p_A K_{B1}+O_1(g)+O(\p_S\p_\Si \nu_S)$, $\p_0\p_1 g_{AB}=2\p_1 K_{AB}+O_1(g)+O(\p_S\p_\Si \nu_S)$ 
where $O(\p_S\p_\Si \nu_S)$ involves up to 2nd order spacelike derivatives of $\nu_S^\a$ except for $\p_1^2 \nu_S^\a$.
In summary, we have 
\bes
\p_0 V_1=\p_0\p_0 g_{01}-\tfrac{1}{2}\p_0\p_1 g_{00}+\p_1^2\nu_S^1+O(\p K,\p_S\p_\Si \nu_S)+O_1(g). 
\ees
As for the term $\p_0\p_1 g_{00}$ above, set $\a=0$ in \eqref{Vgf2} and take the $\p_1$ derivative to obtain 
\bes
\begin{split}
\p_1V_0&=-g^{\rho\s}\p_1\p_\rho g_{\s0}+\tfrac{1}{2}g^{\rho\s}\p_1\p_0 g_{\rho\s}+O_1(g)\\
&=\tfrac{1}{2}\p_1\p_0 g_{00}-\p^2_1 g_{10}+\tfrac{1}{2}\p_1\p_0 g_{11}-g^{AB}\p_1\p_A g_{B0}+\tfrac{1}{2}g^{AB}\p_1\p_0 g_{AB}+O_1(g)\\
&=\tfrac{1}{2}\p_1\p_0 g_{00}+O(\p K,\p_S\p_\Si \nu_S)+O_1(g)\\
\end{split}
\ees
where we have applied the previous equations to derive that $-\p^2_1 g_{10}\sim \p_1^2\nu_S^1$ and $\tfrac{1}{2}\p_1\p_0 g_{11}\sim \p_1^2\nu_S^1$ cancel with each other.
Thus we conclude that
\bes
\p_0\p_0 g_{01}=\p_0 V_1+\p_1 V_0+O(\p K,\p_S\p_\Si \nu_S)+O_1(g)
\ees
Combining this with \eqref{b01} gives \eqref{com2nu1}. 

\end{proof}
  
 Next we prove \eqref{Lower1} in the following two Lemmas. 
    
 \begin{lemma}\label{Xlemma}
For any vector field $X$ along $\cC$ and $\d > 0$, 
\be \label{Xest}
\max_{t \in [0,1]} \int_{\cC_t} \<\nabla_{\p_t} h(\nu)^{\Si}, X\> \leq \d \max_{t \in [0,1]}E_t ^{\nu}+ \d^{-1} ||X||_{H^1(\cC)}^2.
\ee 
Consequently 
\be \label{Xest1}
\max_{t \in [0,1]} \int_{\cC_t} \<\nabla_{\p_t} h(\nu)^{\Si}, (\b_{\cC}h_A)^{\Si}\> \leq \d \max_{t \in [0,1]}E_t ^{\nu}+ \d^{-1} B, 
\ee 
and 
\be \label{Xest2}
\max_{t \in [0,1]} \int_{\cC_t} \<\nabla_{\p_t} h(\nu)^{\Si}, f(\d \hat A + d\hat H)^{\Si}\> \leq \d \max_{t \in [0,1]}E_t ^{\nu}+ \d^{-1} \l \Lambda B_{\l} . 
\ee 

\end{lemma}

\begin{proof} 
To begin, transfer the time derivative $\p_t$ on $h(\nu)^{\Si}$ over to the $X$ term: 
$$\int_{\cC_t} \<\nabla_{\p_t} h(\nu)^{\Si}, X\> = \frac{d}{dt}\int_{\cC_t}\<h(\nu)^{\Si}, X\> - \int_{\cC_t}\<h(\nu)^{\Si}, \nabla_{\p_t} X\>.$$
Here as before, the $t$-derivative of the metric and the volume form give lower-order $\l$ terms which may be ignored. We have  
$$\frac{d}{dt}\int_{\cC_t}\<h(\nu)^{\Si}, X\> = \int_{\Si_t}\<h(\nu)^{\Si}, X\>  \leq \d \int_{\Si_t}|h(\nu)^{\Si}|^2 + \d^{-1}\int_{\Si_t}|X|^2.$$
By the Sobolev trace inequality \eqref{Sob2}, the first term on the right is bounded by $\d E_t^{\nu}$. 

  Next, 
$$\int_{\cC_t}\<h(\nu)^{\Si}, \nabla_{\p_t} X\> \leq \d \int_{\cC_t}|h(\nu)^{\Si}|^2 + \d^{-1}\int_{\cC_t}|\nabla_{\p_t} X|^2.$$
For the first term here,  
$$ \d \int_{\cC_t}|h(\nu)^{\Si}|^2 \leq \d \max_{t\in [0,1]}\int_{\Si_t}|h(\nu)^{\Si}|^2 \leq \d \max_{t \in [0,1]}E_t^{\nu}.$$ 
This gives \eqref{Xest}, using the Sobolev trace inequality \eqref{Sob2} for $\Si_t \subset \cC$. 

  The estimate \eqref{Xest1} then follows easily, using again the Sobolev trace inequality, as does \eqref{Xest2} via the Cauchy-Schwarz inequality, 
Sobolev trace inequality and \eqref{eef}. 
  
\end{proof}

\begin{lemma}\label{Blemma}
Set $B = A+H\g$ and note that $|B|_{C^1} \leq \l \Lambda$.Then 
$$2\int_{\cC_t}\<\nabla_{\p_t} h(\nu)^{\Si}, [B h(\nu)^{\tT}]^{\Si}\> = \frac{d}{dt}\int_{\cC_t}\<h(\nu)^{\Si}, [Bh(\nu)]^{\Si}\> - \int_{\cC_t}\<h(\nu)^{\Si}, 
[(\nabla_{\p_t} B)h(\nu)]^{\Si}\>,$$
so that 
\be \label{B}
\max_{t \in [0,1]} |\int_{\cC_t}\<\nabla_{\p_t} h(\nu)^{\Si}, [Bh(\nu)^{\tT}]^{\Si}\> | \leq \l \Lambda \max_{t \in [0,1]}E_t^{\nu}. 
 \ee
\end{lemma}

\begin{proof}
We have 
$$\int_{\cC_t}\<\nabla_{\p_t} h(\nu)^{\Si}, Bh(\nu)^{\Si}\> = \frac{d}{dt}\int_{\cC_t}\<h(\nu)^{\Si}, Bh(\nu)^{\Si}\> - \int_{\cC_t}\<h(\nu)^{\Si}, 
\nabla_{\p_t} Bh(\nu)^{\Si}\> = $$
$$\frac{d}{dt}\int_{\cC_t}\<h(\nu)^{\Si}, Bh(\nu)^{\Si}\> - \int_{\cC_t}\<h(\nu)^{\Si}, B\nabla_{\p_t} h(\nu)^{\Si}\> - \int_{\cC_t}\<h(\nu)^{\Si}, 
(\nabla_{\p_t} B)h(\nu)^{\Si}\> = $$
$$\frac{d}{dt}\int_{\cC_t}\<h(\nu)^{\Si}, B)h(\nu)^{\Si}\> - \int_{\cC_t}\<\nabla_{\p_t} h(\nu)^{\Si}, Bh(\nu)^{\Si}\> - \int_{\cC_t}\<h(\nu)^{\Si}, 
(\nabla_{\p_t} B)h(\nu)^{\Si}\> .$$
This gives the first statement. For the second statement, proceed in the same way as in Lemma \ref{Xlemma}. 

\end{proof}

\bibliographystyle{plain}

\end{document}